\documentclass[12pt]{amsart}
\allowdisplaybreaks
\usepackage{caption}
\usepackage{listings}

\usepackage[dvipsnames]{xcolor}
\usepackage{fancyhdr}
\usepackage[english]{babel}
\usepackage{amsthm,amssymb}
\usepackage{amsmath}
\usepackage{enumerate}
\usepackage{enumitem}
\setlist[enumerate]{itemsep=0mm}
\usepackage[colorlinks=true,linkcolor=blue,anchorcolor=blue,citecolor=blue,filecolor=blue,menucolor=blue,runcolor=blue,urlcolor=Blue, backref]{hyperref}
\usepackage{amsfonts}
\usepackage{amsbsy}
\usepackage{bm}
\usepackage{graphicx}
\graphicspath{ {./images/} }
\newcommand{\bparen}[1]{\left(#1\right)}

\newcommand{\R}{\mathbb{R}}
\newcommand{\C}{\mathbb{C}}
\newcommand{\N}{\mathbb{N}}
\newcommand{\Z}{\mathbb{Z}}

\renewcommand{\mod}[1]{{\ifmmode\text{\rm\ (mod~$#1$)}\else\discretionary{}{}{\hbox{ }}\rm(mod~$#1$)\fi}}
\newcommand{\ep}{\varepsilon}

\usepackage{pgfplots}\usepgfplotslibrary{fillbetween}
\pgfplotsset{compat=1.18}
\usepackage{colonequals}
\usepackage{fullpage}

\usepackage{graphicx}
\usepackage{amssymb}
\usepackage{color}

\numberwithin{equation}{section}
\newtheorem{theorem}{Theorem}[section]
\newtheorem{lemma}[theorem]{Lemma}
\newtheorem{corollary}[theorem]{Corollary}
\newtheorem{proposition}[theorem]{Proposition}
\theoremstyle{definition}
\newtheorem{definition}[theorem]{Definition}
\newtheorem{notation}[theorem]{Notation}
\newtheorem{remark}[theorem]{Remark}

\newcommand{\B}{\mathcal{B}} 
\newcommand{\lA}{\mathfrak{A}} 
\newcommand{\lB}{\mathfrak{B}} 
\newcommand{\la}{\mathfrak{a}} 
\newcommand{\lb}{\mathfrak{b}} 
\newcommand{\lAb}{\mathbf{\tilde A}} 
\newcommand{\lBb}{\mathbf{\tilde B}} 
\newcommand{\lL}{\mathfrak{L}} 

\newcommand{\li}{\operatorname{li}}
\newcommand{\Err}{\operatorname{Err}}
\newcommand{\Ex}{\mathbb{E}}

\newcommand{\PV}{\mathop{\rm P.V.}}

\makeatletter
\newcommand{\addresseshere}{%
  \enddoc@text\let\enddoc@text\relax
}
\makeatother

\begin{document}
\baselineskip=17pt

\title{Correlations of error terms for weighted prime counting functions}
\author{Shubhrajit Bhattacharya}
\address{Department of Mathematics, University of Chicago, Chicago, IL, 60637, USA}
\email{shubhrajit@uchicago.edu}
\author{Greg Martin}
\address{Department of Mathematics, University of British Columbia, Vancouver, BC, V6T 1Z2, Canada}
\email{gerg@math.ubc.ca}
\author{Reginald M.~Simpson}
\address{Department of Mathematics, University of British Columbia, Vancouver, BC, V6T 1Z2, Canada}
\email{rs@mailc.net}

\begin{abstract}
Standard prime-number counting functions, such as $\psi(x)$, $\theta(x)$, and $\pi(x)$, have error terms with limiting logarithmic distributions once suitably normalized. The same is true of weighted versions of those sums, like $\pi_r(x) = \sum_{p\le x} \frac1p$ and $\pi_\ell(x) = \sum_{p\le x} \log(1-\frac1p)^{-1}$, that were studied by Mertens. These limiting distributions are all identical, but passing to the limit loses information about how these error terms are correlated with one another.

In this paper, we examine these correlations, showing, for example, that persistent inequalities between certain pairs of normalized error terms are equivalent to the Riemann hypothesis (RH). Assuming both RH and LI, the linear independence of the positive imaginary parts of the zeros of $\zeta(s)$, we calculate the logarithmic densities of the set of real numbers for which two different error terms have prescribed signs. For example, we conditionally show that $\psi(x) - x$ and $\sum_{n\le x} \frac{\Lambda(n)}n - (\log x - C_0)$ have the same sign on a set of logarithmic density $\approx 0.9865$.
\end{abstract}

\maketitle

\tableofcontents 



\section{Introduction}

A central topic in comparative prime number theory is the relationship between the prime counting function $\pi(x)$ and the logarithmic integral $\li(x)=\int_0^x\frac{dx}{\log x}$, which are asymptotic to each other by the prime number theorem. While empirical data led many to believe that $\pi(x) < \li(x)$ for all $x\ge2$, Littlewood~\cite{1914.Littlewood} proved that $\pi(x) - \li(x)$ has infinitely many sign changes. Bays and Hudson~\cite{MR1752093} gave convincing evidence that the first value of~$x$ for which $\pi(x)>\li(x)$ (the true ``Skewes number'' for this prime number race) is approximately $1.4\times10^{316}$. Assuming the Riemann hypothesis (RH), Rubinstein and Sarnak~\cite{Rubinstein1994ChebyshevsB}, extending work of Wintner~\cite{MR1507933}, showed that a normalized version of the difference $\pi(x)-\li(x)$ has a limiting logarithmic distribution (see Definition~\ref{defn1} below); under the additional hypothesis (LI) that the positive numbers~$\gamma$ for which $\zeta(\frac12+i\gamma)=0$ are linearly independent over the rationals, they calculated that the logarithmic density (see Definition~\ref{logdensity} below) of the set of~$x$ for which $\pi(x)>\li(x)$ is a constant~$\eta_1$ which is approximately $2.6\times10^{-6}$.

These methods extend to various weighted versions of the prime counting function, such as $\theta(x) = \sum_{p\le x} \log p$, which is asymptotic to~$x$ again by the prime number theorem. Assuming RH and LI, a suitably normalized version of $\theta(x)-x$ has the same limiting logarithmic distribution as the one associated to $\pi(x)-\li(x)$, and Rubinstein and Sarnak's work also shows that the logarithmic density of the set of~$x$ for which $\theta(x)>x$ is the same constant~$\eta_1$. Moreover, Platt and Trudgian~\cite{MR3454375} gave evidence that the first value of~$x$ for which $\theta(x)>x$ (the ``Skewes number'' for this weighted prime number race) is also approximately $1.4\times10^{316}$, though presumably not exactly the same value as for the inequality $\pi(x)>\li(x)$. (See Remark~\ref{coincidence remark} below for an insight as to why these limiting distributions are the same.)

The same phenomenon can be observed when examining the Mertens sum $\sum_{p\le x} \frac1p \sim \log\log x - C_2$ (for a particular constant~$C_2$) and the Mertens product $\prod_{p\le x} (1-\frac1p)^{-1} \sim e^{C_0}\log x$ (for the Euler--Mascheroni constant~$C_0$). In both cases, all calculated values of the prime-related quantities were greater than the corresponding values of their asymptotic main terms; but Robin~\cite{MR0729173} showed that both differences between the prime-related quantities and their main terms have infinitely many sign changes. Work of B\"uthe~\cite{MR3414306} suggests that the first negative value of $\sum_{p\le x} \frac 1p - (\log\log x - C_2)$ (the ``Skewes number'' for this Mertens sum race) is approximately $1.9\times10^{215}$. In addition, assuming RH and LI, Lay~\cite{lay2015sign} and Lamzouri~\cite{lamzouri2016bias} have shown that appropriate normalizations of these differences have limiting logarithmic distributions that are identical (up to an additive constant) to the distribution associated to $\pi(x)-\li(x)$ and $\theta(x)-x$, and that the logarithmic density of the set of rare sign reversals is the same constant~$\eta_1$.

All of these observations raise the question of how these various races are correlated. Each of the four races has an exceptional set (of logarithmic density~$\eta_1$); for $\pi(x)$ and $\theta(x)$ these sets corresponding to regions with significantly more primes than average, while for the Mertens sum and product they correspond to regions with significantly fewer primes than average. Thus, we might expect the first two exceptional sets to be highly correlated (and the same for the second two), but also for the two pairs of exceptional sets to be highly anticorrelated (this idea appears, for example, at the beginning of~\cite{MR3414306}). Moreover, each of these races has an ``unbiased'' version involving counting prime powers rather than only primes---the simplest such versions (which were examined by Kaczorowski~\cite{Kaczorowski}, for example) involve $\psi(x) = \sum_{n\le x} \Lambda(n) \sim x$ and the corresponding Mertens sum $\sum_{n\le x} \frac{\Lambda(n)}n \sim \log x - C_0$, where $\Lambda(n)$ is the von Mangoldt function. Both of these functions lie above and below their asymptotic main terms equally often (on the logarithmic scale, assuming RH), and we can ask how correlated the sets are where the two functions exceed their respective main terms, for example.

The purpose of this paper is to answer questions like these concerning the correlations of these error terms in prime number-related races. For example, assuming RH and LI, we show:
\begin{itemize}
\item the sets $\{x\ge2\colon \pi(x)>\li(x)\}$ and $\{x\ge2\colon \theta(x)>x\}$ are the same up to a set of logarithmic density~$0$;
\item the sets $\{x\ge2\colon \pi(x)>\li(x)\}$ and $\{x\ge2\colon \sum_{p\le x} \frac 1p > \log\log x - C_2\}$ have bounded (and presumably empty) intersection;
\item but the $\psi(x)$ and $\sum_{n\le x} \frac{\Lambda(n)}n$ races are highly correlated yet not perfectly correlated: the set of~$x\ge2$ such that $\psi(x)-x$ and $\sum_{n\le x} \frac{\Lambda(n)}n - (\log x - C_0)$ have the same sign has a logarithmic density which is approximately $0.9865$.
\end{itemize}
We now proceed to set up the notation and terminology required to make these statements rigorous.

\subsection{Preliminaries}

This section introduces the basic definitions and notation that will be used and referred to frequently throughout the rest of this article.

\begin{definition} \label{defn:function_sets_zeta}
We use the following well-known prime-counting functions:
\begin{align*}
\psi(x) &= \sum_{n \le x} \Lambda(n),\quad
\theta(x) = \sum_{p \le x} \log p,\\
\Pi(x) &= \sum_{n \le x} \frac{\Lambda(n)}{\log n} = \sum_{p^k \le x} \frac 1k, \quad
\pi(x) = \sum_{p \le x} 1.
\end{align*}
We call $\B = \{\psi,\theta,\Pi,\pi\}$ the set of these four \emph{standard functions}.

Furthermore, we define the following \emph{weighted versions} of the above standard functions:
\begin{align*}
\psi_r(x) &= \sum_{n \le x} \frac{\Lambda(n)}n, \quad
\theta_r(x) = \sum_{p \le x} \frac{\log p}p,\\
\Pi_r(x) &= \sum_{n \le x} \frac{\Lambda(n)}{n\log n} = \sum_{p^k \le x} \frac 1{kp^k},\\
\pi_r(x) &= \sum_{p \le x} \frac1p.
\end{align*}
We also define the \emph{logarithm of the Mertens product}
\[
\pi_\ell(x) = \log \prod_{p\le x} \biggl( 1-\frac1p\biggr)^{-1} = -\sum_{p \le x} \log\biggl(1-\frac 1p\biggr) = \sum_{p\le x} \sum_{k=1}^\infty \frac1{kp^k},
\]
which behaves in many ways like the four weighted versions. We call $$\B_r = \{\psi_r,\theta_r,\Pi_r,\pi_r,\pi_\ell\}$$ the set of these five \emph{reciprocal functions}.
\end{definition}

Each of the five reciprocal functions in~$\B_r$ has a well-known asymptotic formula, proved by Mertens with elementary arguments, while each of the four standard functions in~$\B$ also has a well-known asymptotic formula that is equivalent to the prime number theorem. When the main terms of these asymptotic formulas are subtracted, the resulting error terms have explicit formulas involving a sum over the zeros of the Riemann zeta function (see Section~\ref{explicit section}); when further divided by the order of magnitude of any one of these summands, the resulting normalized error terms are known (conditionally) to have limiting logarithmic distribution functions. We record these normalized error terms explicitly.

\begin{definition} \label{defn:error_bias_zeta}
The normalized error terms of the four standard functions are
\begin{align*}
E^\psi(x) & = \frac{\psi(x)-x}{\sqrt x}, &
E^\Pi(x) & = \frac{\Pi(x)-\li(x)}{\sqrt x/\log x}, \\
E^\theta(x) & = \frac{\theta(x)-x}{\sqrt x}, &
E^\pi(x) & = \frac{\pi(x)-\li(x)}{\sqrt x/\log x}.
\end{align*}
The normalized error terms of the five reciprocal functions are
\begin{align*}
E^{\psi_r}(x) & = \bigl( \psi_r(x)-(\log x-C_0) \bigr)\sqrt{x},\\ 
E^{\Pi_r}(x) & = \bigl( \Pi_r(x)-(\log\log x + C_0) \bigr)\sqrt{x}\log x , \\ 
E^{\theta_r}(x) & = \bigl(\theta_r(x)- ( \log x - C_1 ) \bigr)\sqrt{x},\\ 
E^{\pi_r}(x) & = \bigl( \pi_r(x)- ( \log\log x - C_2 ) \bigr)\sqrt{x}\log x, \\ 
E^{\pi_\ell}(x) & = \bigl( \pi_\ell(x)- (\log\log x + C_0) \bigr) \sqrt{x}\log x . 
\end{align*}
Here $C_0$ is Euler's constant, while $C_1 = C_0 + \sum_p \frac{\log p}{p(p-1)}$ and $C_2 = -C_0 + \sum_{k=2}^\infty\sum_p \frac 1{kp^k}$ (although the exact values will usually not be important to us).
\end{definition}

Some of these normalized error terms are biased towards positive or negative values. (Loosely speaking, nature wants to count all prime powers, and it is our human insistence on throwing away squares and higher powers of primes that creates biases in functions that count only primes themselves.) These biases manifest as constant terms in the explicit formulas given in Section~\ref{explicit section}, which we now name.

\begin{definition} \label{bias terms def}
For the nine functions in~$\B\cup\B_r$, we define the {\em bias terms}
\[
\beta_\theta = \beta_\pi = -1, \qquad
\beta_\psi = \beta_\Pi = \beta_{\psi_r} = \beta_{\Pi_r} = 0,\qquad
\beta_{\theta_r} = \beta_{\pi_r} = \beta_{\pi_\ell} = 1.
\]
These constants are the average values of the corresponding normalized error terms (in ways that can be made precise).
\end{definition}

We can now state our first result, which concerns the ranges of values taken by differences of these normalized error terms. The following theorem is proved at the end of Section~\ref{explicit section}.

\begin{theorem} \label{o(1) theorem}
Assume RH.
\begin{enumerate}[label={\rm(\alph*)}]
\item If~$f$ and~$g$ are both standard functions, then
\[
E^f(x) - E^g(x) = \beta_f - \beta_g + o(1).
\]
The same is true if~$f$ and~$g$ are both reciprocal functions.
\item If~$f$ is a reciprocal function and~$g$ is a standard function, then
\begin{align*}
\liminf_{x\to\infty} \bigl( E^f(x) - E^g(x) \bigr) &\ge \beta_f - \beta_g - w
\\ 
\limsup_{x\to\infty} \bigl( E^f(x) - E^g(x) \bigr) &\le \beta_f - \beta_g + w,
\end{align*}
where
\begin{equation} \label{w def}
w = 2+C_0-\log 4\pi \approx 0.0461914.
\end{equation}
\end{enumerate}
\end{theorem}

\begin{remark} \label{LI gives exact support}
If we assume LI in addition to RH, then both inequalities in Theorem~\ref{o(1) theorem}(b) are actually equalities, as we show in Theorem~\ref{theorem:joint_dist_zeta}(b) below.
\end{remark}

Next, we examine the question of when one normalized error term always exceeds a second when~$x$ is sufficiently large; in some cases, such a persistent inequality can be proved unconditionally, but in other cases it is equivalent to RH. The two parts of the following theorem are proved at the ends of Section~\ref{unconditional section} and~\ref{mellin section}.

\begin{theorem} \label{RH theorem} \
\begin{enumerate}[label={\rm(\alph*)}]
\item Unconditionally, there exists an absolute constant~$x_0$ such that each of the inequalities
\begin{multline*}
\qquad\qquad E^\theta(x) < E^\psi(x),\, E^\pi(x) < E^\Pi(x),\, E^{\psi_r}(x) < E^{\theta_r}(x), \\
\text{and } E^{\Pi_r}(x) < E^{\pi_\ell}(x) < E^{\pi_r}(x)
\end{multline*}
holds for all~$x>x_0$.
\item Suppose the two functions $f,g\in \B\cup\B_r$ satisfy $\beta_f < \beta_g$. If
\begin{equation} \label{not these five}
(f,g) \notin \bigl\{ (\theta,\psi),\, (\pi,\Pi),\, (\psi_r,\theta_r),\, (\Pi_r,\pi_r),\, (\Pi_r,\pi_\ell) \bigr\}
\end{equation}
(that is, if~$f$ and~$g$ do not form a pair of functions covered by the inequalities in part~(a) of this theorem), then RH is equivalent to the existence of an absolute constant~$x_0$ such that $E^f(x) < E^g(x)$ for all~$x>x_0$. More generally, RH is equivalent to the difference $E^g(x) - E^f(x)$ being bounded either above or below.
\end{enumerate}
\end{theorem}

\subsection{Main results} \label{main results}

The appropriate method of measuring sets of real numbers involving these error terms turns out to be the (limiting) logarithmic density.

\begin{definition} \label{logdensity}
For any set $P\subset[2,\infty)$, we define
\begin{align*}
\delta(P) = \lim_{X\to\infty}\frac{1}{\log X}\int_{t\in P\cap[2,X]}\frac{dt}{t}
\end{align*}
(when the limit exists) to be the \emph{logarithmic density} of~$P$. Given functions~$f$ and~$g$, we use the shorthand notation
\[
\delta(f > g) = \delta\bigl( \{ x\ge2 \colon f(x) > g(x) \} \bigr),
\]
and similarly for chained inequalities involving three or more functions.
\end{definition}

\begin{notation}
Throughout this paper, $\gamma_1\le\gamma_2\le\cdots$ denotes the sequence of positive imaginary parts of $\zeta(s)$, listed with multiplicity if necessary, so that $\gamma_1\approx 14.135$ for example.
\end{notation}

\begin{definition} \label{eta1 def}
Define the real-valued random variable
\[
V_1 = \sum_{m=1}^\infty 2\Re \frac{Z_m}{\frac12+i\gamma_m}
\quad\text{or equivalently}\quad
V_1 = \sum_{m=1}^\infty \frac{2\Re Z_m}{\sqrt{\frac14+\gamma_m^2}},
\]
where the~$Z_m$ are independent random variables, each uniformly distributed on the unit circle in~$\C$. Then, define~$\mu_1$ to be the probability distribution associated to~$V_1$ (which is symmetric about the origin), and define the constant~$\eta_1$ by
\[
\eta_1 = \Pr(V_1 > 1) = \int_1^\infty d\mu_1.
\]
Assuming RH, Rubinstein and Sarnak~\cite{Rubinstein1994ChebyshevsB} proved that~$\mu_1$ has support equal to all of~$\R$ and is absolutely continuous with respect to Lebesgue measure, and calculated that $\eta_1 \approx 2.6\times10^{-6}$.
\end{definition}

\begin{remark}
Assuming RH and LI, Rubinstein and Sarnak further proved that the limiting logarithmic distributions of~$E^\pi$ and $E^\theta$ (in the sense of Definition~\ref{defn1} below) are the same as the distribution of $V_1 - 1$, so that in particular $\delta(\pi(x)>\li(x)) = \delta(\theta(x)>x) = \eta_1$.

Indeed, for each of the nine functions $f\in\B\cup\B_r$, it is known (assuming RH and LI) that the limiting logarithmic distribution of~$E^f$ is the same as the distribution of $V_1 + \beta_f$ (which justifies the use of the name ``bias term''), so that all nine individual distributions are identical up to the additive constants~$\beta_f$. We collect these results in Corollary~\ref{nine mu1s cor}.
\end{remark}

\begin{definition} \label{eta2 def}
Define the $\R^2$-valued random variable
\[
V_2 = \sum_{m=1}^\infty \biggl( 2\Re \frac{Z_m}{\frac12+i\gamma_m}, 2\Re \frac{Z_m}{-\frac12+i\gamma_m} \biggr),
\]
with the~$Z_m$ as in Definition~\ref{eta1 def}.
Note that~$V_2$ is invariant under reflection in either diagonal line $y=\pm x$ (since the distribution of~$Z_m$ is invariant under conjugation and negation).
Note similarly that each coordinate has the same distribution as~$V_1$, although the two coordinates are not independent.

Then, define~$\mu_2$ to be the probability distribution associated to~$V_2$, and define the constant~$\eta_2$ by
\[
\eta_2 = \Pr\bigl( \text{the two coordinates of $V_2$ have different signs} \bigr)
\]
(so that~$\eta_2$ is the mass that~$\mu_2$ assigns to the union of the second and fourth open quadrants).
\end{definition}

Assuming RH, we can calculate (see the end of Section~\ref{error section}) that $\eta_2 \approx 0.01355$. Assuming RH and LI, this distribution is connected to the limiting logarithmic distribution of vector-valued functions such as $\bigl( E^\psi(x), E^{\psi_r}(x) \bigr)$, as evidenced in Theorems~\ref{log_density part 2} and~\ref{theorem:joint_dist_zeta}(c) below.

Our first three main considerations are the probabilities (in the sense of logarithmic density) of two normalized error terms being jointly positive, jointly negative, or having mixed signs. These probabilities depend upon the bias terms of the functions in question, as well as whether the two functions are standard or reciprocal functions. These three results are proved in Section~\ref{joint dist section}.

The first result of this type, which uses the quantity~$\eta_1$ from Definition~\ref{eta1 def}, covers the cases where the bias terms are unequal.

\begin{theorem}
\label{theorem:log_density_zeta}
Assume RH and LI. Let $f,g \in \B \cup \B_r$.

\begin{enumerate}[label={\rm(\alph*)}]
\item If $\beta_f = -1$ and $\beta_g = 0$ (for example, when $f = \pi$ and $g = \psi$), then
\begin{align*}
\delta(0 < E^f < E^g) = \eta_1,  \quad
\delta(E^f < 0 < E^g) = \tfrac12-\eta_1,  \quad
\delta(E^f < E^g < 0) = \tfrac12.
\end{align*}

\item If $\beta_f = 0$ and $\beta_g = 1$ (for example, when $f = \psi$ and $g = \pi_r$), then
\begin{align*}
\delta(0 < E^f < E^g) = \tfrac12,  \quad
\delta(E^f < 0 < E^g) = \tfrac12-\eta_1,  \quad
\delta(E^f < E^g < 0) =\eta_1.
\end{align*}

\item If $\beta_f = -1$ and $\beta_g = 1$ (for example, when $f = \pi$ and $g = \pi_r$), then
\begin{align*}
\delta(0 < E^f < E^g) = \eta_1, \,
\delta(E^f < 0 < E^g) = 1-2\eta_1, \,
\delta(E^f < E^g < 0) = \eta_1.
\end{align*}
\end{enumerate} 
\end{theorem}

\noindent
Note that in all cases covered by Theorem~\ref{theorem:log_density_zeta}, we have $\delta(E^f \ge E^g) = 0$ by Theorem~\ref{RH theorem}(b), and so the given sets of inequalities are exhaustive.

Our second such result concerns two functions with equal bias terms, where either both functions are standard or both functions are reciprocal.

\begin{theorem}
\label{theorem:log_density_zeta 1.5}
Assume RH and LI.

\begin{enumerate}[label={\rm(\alph*)}]
\item We have $\delta(0 < E^\pi \text{ and } 0 < E^\theta) = \eta_1$ and
$\delta(E^\pi < 0 \text{ and } E^\theta < 0) = 1-\eta_1$.
\item We have $\delta(0 < E^\psi \text{ and } 0 < E^\Pi) = \delta(E^\psi < 0 \text{ and } E^\Pi < 0) = \frac12$, and the same if~$\psi$ and~$\Pi$ are both replaced by~$\psi_r$ and~$\Pi_r$.
\item If $f,g\in\{\theta_r,\pi_r,\pi_\ell\}$ then
\[
\delta(0 < E^f \text{ and } 0 < E^g) = 1-\eta_1
\quad\text{and}\quad
\delta(E^f < 0 \text{ and } E^g < 0) = \eta_1.
\]
\end{enumerate}
\end{theorem} 

Finally, our third and most novel main result concerns two functions with equal bias terms, where one function is standard and the other reciprocal. The statement uses the constant~$\eta_2$ from Definition~\ref{eta2 def}.

\begin{theorem}\label{log_density part 2}
Assume RH and LI. Suppose that~$f$ is a standard function and~$g$ is a reciprocal function with $\beta_f=\beta_g=0$ (for example, $f=\psi$ and $g=\psi_r$). Then
\begin{align*}
\delta(E^f < 0 < E^g) &= \delta(E^g < 0 < E^f) = \frac{\eta_2}2
\end{align*}
and
\begin{multline*}
\delta(E^f < E^g < 0) = \delta(E^g < E^f < 0) \\ = \delta(0 < E^f < E^g) = \delta(0 < E^g < E^f) = \frac{1-\eta_2}{4}.
\end{multline*}
\end{theorem}

\noindent
The regions corresponding to these six sets of inequalities can be visualized in the right-hand graph in Figure~\ref{main pair of graphs} below.

Our primary tool for evaluating these densities is the limiting logarithmic distribution of a real- or vector-valued function, which we now define formally.

\begin{definition}\label{defn1}
We say that a vector-valued function $\vec\phi\colon[2,\infty)\to\R^d$ has a {\em limiting logarithmic distribution}~$\mu$ on~$\R^d$ if~$\mu$ is a probability measure on $\R^d$ and 
\[
\lim_{X\to\infty}\int_2^Xf\bigl(\vec\phi(x)\bigr){\frac{dx}{x}}=\int_{\R^\ell}f(\vec{x})\,d\mu(\vec{x})
\]
for all bounded continuous functions $f\colon\R^d\to\C$.
\end{definition}

\begin{remark}
Logarithmic density measures a set of inputs to a function, while logarithmic distributions correspond to the frequency of the function's outputs in various regions. One way they are related is as follows: for any vector-valued function $\vec \phi(t)$ possessing a limiting logarithmic distribution~$\mu$, and any subset $T\subset\R^d$, set $S = \{ x \in [2,\infty) \colon \vec \phi(x) \in T \}$. Then $\delta(S) = \mu(T)$, provided that the boundary of~$T$ satisfies $\mu(\partial T) = 0$. (This proviso is necessary because the characteristic function of~$T$ is not continuous.) In practice, $\mu(\partial T) = 0$ is automatically satisfied if the limiting distribution~$\mu$ is absolutely continuous with respect to Lebesgue measure on~$\R^d$, as long as~$\partial T$ has Lebesgue measure~$0$ (which, for example, is guaranteed if~$\partial T$ is contained in the union of finitely many hyperplanes).
\end{remark}

We are able to determine the precise joint limiting logarithmic distribution of any pair of normalized error terms under consideration, as our final main result describes.

\begin{theorem}
\label{theorem:joint_dist_zeta}
Assume RH and LI. Let $f,g \in \B \cup \B_r$.
\begin{enumerate}[label={\rm(\alph*)}]
\item If $f$ and~$g$ are both standard functions, then the limiting logarithmic distribution of $\bigl( E^{f}(x), E^{g}(x) \bigr)$ is singular in~$\R^2$ with its support equaling the line $y-\beta_g = x-\beta_f$. The same is true if~$f$ and~$g$ are both reciprocal functions.
\item If $f \in \B$ and $g \in \B_r$ (or vice versa) then the joint limiting
logarithmic distribution of $\bigl( E^{f}(x), E^{g}(x) \bigr)$ is absolutely continuous on~$\R^2$, with support equal to the narrow diagonal strip $$\bigl\{ (x,y) \in \R^2 \colon \bigl| (x-\beta_f) - (y-\beta_g) \bigr| \le w \bigr\}.$$ 
In this case, the joint distribution of $\bigl( E^{f}(x), E^{g}(x) \bigr)-(\beta_f,\beta_g)$ equals $\mu_2$ as in Definition~\ref{eta2 def}.
\end{enumerate}
\end{theorem}

\noindent
Two examples of Theorem~\ref{theorem:joint_dist_zeta}(b) are shown in Figure~\ref{main pair of graphs} (other examples of such graphs appear in Appendix~\ref{graph appendix}).
This theorem is an immediate consequence of Proposition~\ref{nine prop}, which in fact determines the joint logarithmic distribution of our nine normalized error terms simultaneously.

\begin{figure}[ht]
\includegraphics[width=3in]{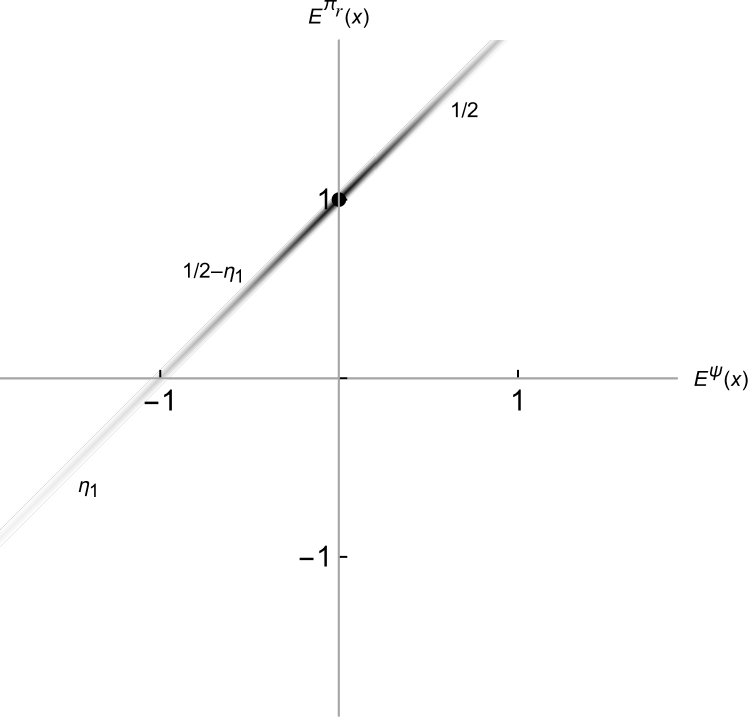}
\qquad
\includegraphics[width=3in]{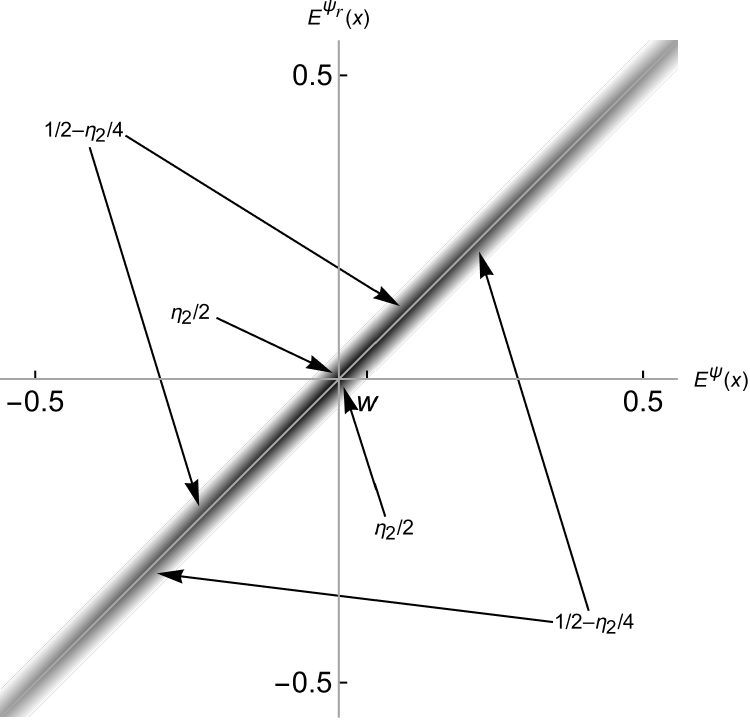}
\caption{Two manifestations of~$\mu_2$. The left-hand graph shows the joint distribution of $E^\psi(x)$ and $E^{\pi_r}(x)$, whose centre is $(\beta_\psi,\beta_{\pi_r}) = (0,1)$; certain regions in the support of $\mu_2+(0,1)$ are labeled with their densities from Theorem~\ref{theorem:log_density_zeta}(b). The right-hand graph shows the joint distribution of $E^\psi(x)$ and $E^{\psi_r}(x)$, whose centre is $(\beta_\psi,\beta_{\psi_r}) = (0,0)$; certain regions in the support of~$\mu_2$ are labeled with their densities from Theorem~\ref{theorem:joint_dist_zeta}(b).}
\label{main pair of graphs}
\end{figure}


In Section~\ref{explicit section} we exhibit the explicit formulas for our nine normalized error terms, which allows us to prove Theorem~\ref{o(1) theorem}. Using elementary methods combined with Landau's theorem for Mellin transforms, we establish Theorem~\ref{RH theorem} in Section~\ref{equivalent-to-RH section}. The Fourier transforms of these normalized error terms are examined in Section~\ref{joint dist section}, which leads to the determination of their joint limiting logarithmic distribution from which all the theorems in Section~\ref{main results} are then derived. Finally, Sections~\ref{density calc section} and~\ref{error section} contain the numerical calculation of the new constant~$\eta_2$, first writing it in terms of the principal value of a two-variable integral, and then numerically calculating an approximation to that integral with a rigorous error term. Finally, Appendix~\ref{graph appendix} contains many graphs like the ones in Figure~\ref{main pair of graphs}, illustrating all of the possible cases of Theorem~\ref{theorem:joint_dist_zeta}.

\section{Explicit formulas for normalized error terms} \label{explicit section}

In this section, we list explicit formulas for the error terms of the four standard functions and the five reciprocal functions. In this section,~$\gamma$ denotes an ordinate of a nontrivial zero of $\zeta(s)$.

\begin{proposition}
\label{corollary:explicit_standards}
Assume RH. In the notation of Definition~\ref{defn:error_bias_zeta}, for $x,X\ge2$:
\begin{enumerate}[label={\rm(\alph*)}]
\item $\displaystyle E^{\psi}(x)=-\Re\sum_{0 < \gamma \le X} \frac{2x^{i\gamma}}{\frac 12+i\gamma} + O\biggl( \frac{\sqrt{x}\log^2(xX)}{X}+\frac{\log x}{\sqrt{x}} \biggr)$
\item $\displaystyle E^{\theta}(x) = -1-\Re\sum_{0 < \gamma \le X} \frac{2x^{i\gamma}}{\frac 12+i\gamma}+O\biggl(\frac{\sqrt x\log^2(xX)}X + \frac{1}{x^{1/6}}\biggr)$
\item $\displaystyle E^{\Pi}(x) = -\Re \sum_{0 < \gamma \le X} \frac{2x^{i\gamma}}{\frac 12+i\gamma} + O\biggl(\frac{\sqrt x\log^2(xX)}X + \frac 1{\log x}\biggr)$
\item $\displaystyle E^{\pi}(x) = -1 - \Re \sum_{0 < \gamma \le X} \frac{2x^{i\gamma}}{\frac 12+i\gamma} + O\biggl(\frac{\sqrt x\log^2(xX)}X + \frac 1{\log x}\biggr)$.
\end{enumerate}
\end{proposition}

\begin{proof} \
\begin{enumerate}[label={\rm(\alph*)}]
\item This formula is equivalent to~\cite[equation~(1.6)]{akbary2013limiting}.
\item The standard comparison
\begin{align} \label{theta psi 1/3}
\psi(x) &= \sum_{k=1}^\infty \theta(x^{1/k}) = \theta(x) + \theta(\sqrt x) + O(x^{1/3})\\\Longrightarrow\theta(x) &= \psi(x) - \psi(\sqrt x) + O(x^{1/3})\notag
\end{align}
shows that this formula follows from part~(a).
\item Partial summation gives us
\begin{align*}
\Pi(x) - \li(x) &= \int_{2-}^x \frac1{\log t} \,d(\psi(t)-t) \\ &= \frac{\psi(t)-t}{\log t} \bigg|_{2-}^x + \int_2^x \frac{\psi(t)-t}{t\log^2t} \,dt \\
&= \frac{\psi(x)-x}{\log x} + O\biggl( 1 + \biggl| \int_2^x \frac{\psi(t)-t}{t\log^2t} \,dt \biggr| \biggr),
\end{align*}
which, upon dividing both sides by $\sqrt x/\log x$ becomes
\[
E^\Pi(x) =  E^\psi(x) + O\biggl( \frac{\log x}{\sqrt x} \biggl| \int_2^x \frac{E^\psi(t)}{\sqrt t\log^2t} \,dt \biggr| \biggr).
\]
Using the formula from part~(a) in the integrand, we see that
\begin{align}
E^\Pi(x) - E^\psi(x) & \ll \frac{\log x}{\sqrt x} \biggl| 2\Re\int_2^x \frac1{\sqrt t\log^2t} \sum_{0 < \gamma \le X} \frac{t^{i\gamma}}{\frac 12+i\gamma} \,dt \biggr| \notag \\
&\qquad{}+ \frac{\log x}{\sqrt x} \int_2^x \frac1{\sqrt t\log^2t} \biggl( \frac{\sqrt{t}\log^2(tX)}{X}+\frac{\log t}{\sqrt{t}} \biggr)\,dt \label{gotcha} \\
&\ll \frac{\log x}{\sqrt x} \biggl| \int_2^x \frac{g_X(t)}{\log^2t} \,dt \biggr| \notag \\ & \quad + \frac{\log x}{\sqrt x} \int_2^x \frac1{\sqrt t\log^2t} \biggl( \frac{\sqrt{t}\log^2(tX)}{X}+\frac{\log t}{\sqrt{t}} \biggr)\,dt \notag
\end{align}
where we have defined
\[
g_X(t) = \sum_{0 < \gamma \le X} \frac{t^{-1/2+i\gamma}}{\frac 12+i\gamma}.
\]
If we set $G_X(x) = \int_2^x g_X(t)\,dt$, then
\[
G_X(x) = \sum_{0 < \gamma \le X} \frac{t^{1/2+i\gamma}}{(\frac 12+i\gamma)^2} \bigg|_2^x \ll \sqrt x,
\]
from which it follows by another partial summation that
\begin{multline*}
\int_2^x \frac{g_X(t)}{\log^2t} \,dt = \frac{G_X(t)}{\log^2t} \bigg|_2^x  + 2 \int_2^x \frac{G_X(t)}{t\log^3t}\,dt \\ \ll \frac{\sqrt x}{\log^2 x} + \int_2^x \frac1{\sqrt t \log^3t}\,dt \ll \frac{\sqrt x}{\log^2 x} + \frac{\sqrt x}{\log^3 x}.
\end{multline*}
Therefore, the first term on the right-hand side of equation~\eqref{gotcha} is $\ll 1/\log x$; it remains only to estimate the second term on the right-hand side by
\begin{align*}
\frac{\log x}{\sqrt x} \int_2^x \frac1{\sqrt t\log^2t} & \biggl( \frac{\sqrt{t}\log^2(tX)}{X}+\frac{\log t}{\sqrt{t}} \biggr)\,dt \\
&\ll \frac{\log x}{X\sqrt x} \int_2^x \biggl( 1 + \frac{\log^2X}{\log^2t} \biggr) \,dt + \frac{\log x}{\sqrt x} \int_2^x \frac1{t\log t} \,dt \\
&\ll \frac{\sqrt x\log x}{X} + \frac{\sqrt x\cdot \log^2X}{X\log x} + \frac{\log x\log\log x}{\sqrt x}
\end{align*}
which is an acceptable error term.
\item Again, the standard comparison
\begin{align} \label{pi Pi 1/3}
\Pi(x) &= \sum_{k=1}^\infty \frac{\pi(x^{1/k})}k = \pi(x) + \frac{\pi(\sqrt x)}2 + O\biggl( \frac{x^{1/3}}{\log x} \biggr) \\\Longrightarrow \pi(x) &= \Pi(x) - \frac{\Pi(\sqrt x)}2 + O\biggl( \frac{x^{1/3}}{\log x} \biggr)\notag
\end{align}
shows that this formula follows from part~(c). \qedhere
\end{enumerate}
\end{proof}

Our goal for the rest of this section is to exhibit analogous explicit formulas for the five reciprocal functions.

\begin{lemma}
\label{lemma:rho_squared}
As a sum over all nontrivial zeros of $\zeta(s)$,
\[
\sum_{\rho}\frac{1}{\rho(1-\rho)} = \sum_{\rho} \biggl( \frac{1}{\rho}+\frac{1}{1-\rho} \biggr) = 2+C_0-\log4\pi = w
\]
in the notation of equation~\eqref{w def}.
\end{lemma}

\begin{proof}
This formula is a restatement of~\cite[equation~(10.30)]{montgomery_vaughan_2006}.
\end{proof}

\begin{proposition}
\label{prop:psi_r_explicit_formula}
For all $x,T\ge2$,
\begin{align*}
\psi_r(x)=-\sum_{|\gamma|\leq T}\frac{x^{\rho-1}}{\rho-1}+\log x-C_0+O\bparen{\frac{\log x}{x}+\frac{(\log x)^2}{T}+\frac{(\log T)^2}{T\log x}}.
\end{align*}
\end{proposition}

\begin{proof}
In~\cite[Lemma 2.4]{lamzouri2016bias}, the author proves that for any $\alpha>1$,
\begin{equation} \label{no longer 3.6}
\sum_{n\leq x}\frac{\Lambda(n)}{n^\alpha} = -\frac{\zeta'}{\zeta}(\alpha) + \frac{x^{1-\alpha}}{1-\alpha} - \sum_{|\Im{\rho}|\leq T}\frac{x^{\rho-\alpha}}{\rho-\alpha}+O\bigl( F(x;\alpha,T) \bigr)
\end{equation}
where 
\begin{equation*}
F(x;\alpha,T)=\frac{\log x}{x^\alpha}+\frac{x^{1-\alpha}}{T}\biggl( 4^{\alpha}+(\log x)^2+\frac{(\log T)^2}{\log x} \biggr) + \frac{1}{T}\sum_{n=1}^\infty\frac{\Lambda(n)}{n^{\alpha+{1}/{\log x}}}.
\end{equation*}
We immediately see that
\begin{align*}
\lim_{\alpha\to1^+}F(x;\alpha,T)&=F(x;1,T)\\ &= \frac{\log x}{x}+\frac1T \biggl( 4+(\log x)^2+\frac{(\log T)^2}{\log x} \biggr) - \frac1T \frac{\zeta'}{\zeta}\biggl( 1+\frac1{\log x} \biggr) \\
&\ll \frac{\log x}{x} + \frac{(\log x)^2}T + \frac{(\log T)^2}{T\log x} + \frac{\log x}T.
\end{align*}
We have the formula~\cite[Corollary~10.14]{montgomery_vaughan_2006}
\begin{equation*}
-\frac{\zeta'}{\zeta}(\alpha)=\frac{C_0}{2}+1-\log2\pi-\frac{1}{1-\alpha}+\frac{1}{2}\frac{\Gamma'}{\Gamma}\bparen{\frac{\alpha}{2}+1}-\sum_{\rho}\bparen{\frac{1}{\alpha-\rho}+\frac{1}{\rho}},
\end{equation*}
as well as the Laurent expansion around $\alpha=1$ that gives
\begin{equation*}
\frac{x^{1-\alpha}}{1-\alpha}=\frac{1}{1-\alpha}+\log x + O_x\bigl( |1-\alpha| \bigr).
\end{equation*}
Adding these two equations and taking the limit yields
\begin{align*}
\lim_{\alpha\to1^+}\bparen{-\frac{\zeta'}{\zeta}(\alpha)+\frac{x^{1-\alpha}}{1-\alpha}} &= \frac{C_0}{2}+1-\log2\pi + \log x \\ &\quad + \frac{1}{2}\frac{\Gamma'}{\Gamma}\biggl( \frac32 \biggr) - \sum_{\rho}\bparen{\frac{1}{1-\rho}+\frac{1}{\rho}}\\ &= \log x-C_0
\end{align*}
by Lemma~\ref{lemma:rho_squared} and the special value $\frac{\Gamma'}{\Gamma}(\frac{3}{2})=-C_0-2\log 2+2$.
With these evaluations, we can now take the limit of both sides of equation~\eqref{no longer 3.6} to find that
\begin{multline*}
\psi_r(x) = \lim_{\alpha\to1^+}\sum_{n\leq x}\frac{\Lambda(n)}{n^\alpha} \\ = \log x-C_0 -\sum_{|\gamma|\leq T}\frac{x^{\rho-1}}{\rho-1} +O\bparen{\frac{\log x}{x}+\frac{(\log x)^2}{T}+\frac{(\log T)^2}{T\log x}}
\end{multline*}
as claimed.
\end{proof}

\begin{proposition}
\label{corollary:explicit_reciprocals}
Assume RH. In the notation of Definition~\ref{defn:error_bias_zeta}, for $x,X\ge2$:
\begin{enumerate}[label={\rm(\alph*)}]
\item $\displaystyle E^{\psi_r}(x)=-\Re\sum_{0 < \gamma \le X} \frac{2x^{i\gamma}}{-\frac 12+i\gamma}+O\biggl( \frac{\log x}{\sqrt{x}}+\frac{\sqrt{x}(\log x)^2}{X}+\frac{\sqrt{x}(\log X)^2}{X\log x} \biggr)$ \label{eq1}

\item $\displaystyle E^{\pi_\ell}(x) = 1 - \Re \sum_{0 < \gamma \le X} \frac{2x^{i\gamma}}{-\frac 12+i\gamma} + O\biggl(\frac{\sqrt x\log^2(xX)}X + \frac 1{\log x}\biggr)$ \label{eq2}

\item $\displaystyle E^{\Pi_r}(x) = -\Re \sum_{0 < \gamma \le X} \frac{2x^{i\gamma}}{-\frac 12+i\gamma} + O\biggl(\frac{\sqrt x\log^2(xX)}X + \frac 1{\log x}\biggr)$ \label{eq3}

\item $\displaystyle E^{\pi_r}(x) = 1 - \Re \sum_{0 < \gamma \le X} \frac{2x^{i\gamma}}{-\frac 12+i\gamma} + O\biggl(\frac{\sqrt x\log^2(xX)}X + \frac 1{\log x}\biggr)$ \label{eq4}

\item $\displaystyle E^{\theta_r}(x) = 1-\Re\sum_{0 < \gamma \le X} \frac{2x^{i\gamma}}{-\frac 12+i\gamma}+O\biggl(\frac{\sqrt x\log^2(xX)}X + \frac 1{\log x}\biggr).$ \label{eq5}
\end{enumerate}
\end{proposition}

\begin{proof} \
\begin{enumerate}[label={\rm(\alph*)}]
\item Assuming RH, we can write
\begin{align*}
\sum_{1<|\Im(\rho)|\leq T}\frac{x^{\rho-1}}{\rho-1}=\sum_{1<|\gamma|\leq T}\frac{x^{-\frac{1}{2}+i\gamma}}{-\frac{1}{2}+i\gamma}.
\end{align*}
Then the formula for $E^{\psi_r}(x)$ is a straightforward consequence of Proposition~\ref{prop:psi_r_explicit_formula}.
\item This formula is~\cite[Corollary 2.2]{lamzouri2016bias} or~\cite[Proposition 1]{lay2015sign}, although both those statements contain a sign mistake (confirmed by the author of~\cite{lamzouri2016bias}) that is corrected here.
\item This formula follows from part~(b) and~\cite[Lemma 2.3]{lamzouri2016bias}, which in our notation is
\begin{align*}
\pi_\ell(x)=\Pi_r(x)+\frac{1}{\sqrt{x}\log x}+O\bparen{\frac{1}{\sqrt{x}(\log x)^2}}.
\end{align*}
\item This formula follows from part~(b) and~\cite[Lemma 5]{lay2015sign}, which (when both sides are multiplied by $\sqrt x\log x$) states in our notation that
\[
E^{\pi_r}(x) = E^{\pi_\ell}(x)+O\biggl(\frac{\log x}{\sqrt x}\biggr).
\]
\item This formula is~\cite[Lemma 6]{lay2015sign}. \qedhere
\end{enumerate}
\end{proof}

These explicit formulas enable us to establish the ranges of values taken by differences of the normalized error terms from Definition~\ref{defn:error_bias_zeta}.

\begin{proof}[Proof of Theorem~\ref{o(1) theorem}]
Part~(a) follows immediately from Propositions~\ref{corollary:explicit_standards} and~\ref{corollary:explicit_reciprocals}, in light of the bias terms from Definition~\ref{bias terms def}. Also from those propositions (setting $X=x$):
\begin{align}
\bigl| E^{\psi}(x) - E^{\psi_r}(x) \bigr| &= \biggl| \Re\sum_{0 < \gamma \le x} \frac{2x^{i\gamma}}{\frac 12+i\gamma} - \Re\sum_{0 < \gamma \le x} \frac{2x^{i\gamma}}{-\frac 12+i\gamma} \biggr| + O\biggl( \frac{\log^2x}{x} \biggr) \notag \\
& = \bigg|2\Re\sum_{0 < \gamma \le x} \frac{x^{i\gamma}}{\frac14+\gamma^2}\bigg| + o(1) \label{same comp} \\
& < 2 \sum_{\gamma >0} \frac1{\frac14+\gamma^2} + o(1) = \sum_\rho \frac 1{\rho(1-\rho)} + o(1) = w + o(1) \notag
\end{align}
by Lemma~\ref{lemma:rho_squared}. This inequality implies the special case of part~(b) where $f=\psi_r$ and $g=\psi$, and the general statement follows from this special case and part~(a).
\end{proof}

\section{Persistent inequalities between normalized error terms} \label{equivalent-to-RH section}

We begin this section by establishing some unconditional relationships between pairs of normalized error terms, after which we write down the Mellin transforms of those functions for use in Landau's theorem-type arguments. We prove
Theorem~\ref{RH theorem}(a) at the end of Section~\ref{unconditional section} and
Theorem~\ref{RH theorem}(b) at the end of Section~\ref{mellin section}.

\subsection{Unconditional relationships} \label{unconditional section}

We first establish some relationships among the error terms from Definition~\ref{defn:error_bias_zeta} that are independent of any assumptions on the zeros of $\zeta(s)$; Theorem~\ref{RH theorem}(a) will be an immediate consequence of these formulas.

\begin{lemma} \label{standard Es related}
Unconditionally, there exists a positive constant $c$ such that
\begin{align*}
E^\psi(x)-E^\theta(x) &= 1 + O\bigl( \exp(-c\sqrt{\log x}) \bigr) \\
E^\Pi(x)-E^\pi(x) &= 1 + O\bigl( \exp(-c\sqrt{\log x}) \bigr).
\end{align*}
\end{lemma}

\begin{proof}
As $\psi(x) = \sum_{k=1}^\infty \theta(x^{1/k})$, the standard argument gives
\begin{align} \label{psi-theta}
\psi(x)-\theta(x) &= \sum_{k=2}^\infty \theta(x^{1/k})\\ &= \theta(x^{1/2}) + \theta(x^{1/3}) + \sum_{4\le k\ll \log x} \theta(x^{1/k})\notag \\&= \theta(x^{1/2}) + O(x^{1/3}).\notag
\end{align}
By the prime number theorem, there exists $c_0>0$ such that
\begin{multline*}
\psi(x)-\theta(x) = x^{1/2} + O\bigl( \exp(-c_0\sqrt{\log x^{1/2}}) \bigr) + O(x^{1/3}) \\ = x^{1/2} + O\bigl( \exp(-c\sqrt{\log x}) \bigr)
\end{multline*}
with $c=\frac{c_0}2$. The first claim follows upon dividing both sides by $x^{1/2}$.

Similarly, the identity $\Pi(x) = \sum_{k=1}^\infty \frac 1k \pi(x^{1/k})$ yields
\begin{align} \label{Pi-pi}
\Pi(x) - \pi(x) &= \sum_{k=2}^\infty \frac{\pi(x^{1/k})}k \\&= \frac{\pi(x^{1/2})}2 + \frac{\pi(x^{1/3})}3 + \sum_{4\le k\ll \log x} \frac{\pi(x^{1/k})}k\notag\\ &= \frac{\pi(x^{1/2})}2 + O\biggl( \frac{x^{1/3}}{\log x} \biggr).\notag
\end{align}
The second claim follows by invoking the prime number theorem for $\pi(x^{1/2})$ and dividing both sides by $x^{1/2}/\log x$.
\end{proof}

\begin{lemma} \label{reciprocal Es related 1}
Unconditionally, $E^{\psi_r}(x)-E^{\theta_r}(x) = -1 + O(1/\log x)$.
\end{lemma}

\begin{proof}
By Definition~\ref{defn:function_sets_zeta}, it suffices to prove $\sum_p \frac{\log p}{p(p-1)} + \theta_r(x) - \psi_r(x) = x^{-1/2} (1 + O(1/\log x))$, or equivalently
\begin{equation} \label{apr5a}
\sum_p \sum_{k=2}^\infty \frac{\log p}{p^k} + \sum_{p\le x} \frac{\log p}p - \sum_{k=1}^\infty \sum_{p\le x^{1/k}} \frac{\log p}{p^k} = x^{-1/2} (1 + O(1/\log x)).
\end{equation}
The $k=1$ term in the third sum precisely cancels the second sum. If we define $\Delta^{\theta_r}(x) = \theta_r(x) - (\log x - C_1)$, so that $\Delta^{\theta_r}(x) \ll 1/\log x$ by Mertens's formula~\cite[line~3]{dusart},
then the contribution to the left-hand side of $k=2$ equals
\begin{align*}
\sum_{p > \sqrt x} \frac{\log p}{p^2} & = \int_{\sqrt x}^\infty \frac1u \, d\theta_r(u) = \int_{\sqrt x}^\infty \frac1u \, d\bigl( \log u - C_1 + \Delta^{\theta_r}(u) \bigr) \\
&= \int_{\sqrt x}^\infty \frac1{u^2} \, du + \frac{\Delta^{\theta_r}(u)}{u} \bigg|_{\sqrt x}^\infty + \int_{\sqrt x}^\infty \frac{\Delta^{\theta_r}(u)}{u^2} \,du \\
&= -\frac1u \bigg|_{\sqrt x}^\infty + O\biggl( \frac{1/\log u}{u} \bigg|_{\sqrt x}^\infty + \int_{\sqrt x}^\infty \frac{1/\log u}{u^2} \,du \biggr) \\&= \frac1{\sqrt x} + O\biggl( \frac1{\sqrt x\log\sqrt x} \biggr).
\end{align*}
Finally, the contribution to the left-hand side of equation~\eqref{apr5a} from $k\ge3$ is, by Rankin's trick,
\begin{align*}
\sum_{k=3}^\infty \sum_{p>x^{1/k}} \frac{\log p}{p^k} \le \sum_p \sum_{k\ge3} \frac{\log p}{p^k} \biggl( \frac{p^k}x \biggr)^{3/5} &= x^{-3/5} \sum_p \sum_{k\ge3} \frac{\log p}{p^{2k/5}} \\
&= x^{-3/5} \sum_p \frac{\log p}{p^{6/5} - p^{4/5}}\\& \ll x^{-3/5} \ll \frac1{\sqrt x\log x},
\end{align*}
which concludes the proof of the lemma.
\end{proof}

\begin{lemma} \label{reciprocal Es related 2}
Unconditionally,
\begin{align*}
E^{\pi_\ell}(x)-E^{\Pi_r}(x) & = 1 + O(1/\log x) \\
E^{\pi_\ell}(x)-E^{\pi_r}(x) & \ll x^{-1/2}\log x \\
E^{\Pi_r}(x)-E^{\pi_r}(x) & = -1 + O(1/\log x).
\end{align*}
Furthermore, $E^{\pi_\ell}(x) < E^{\pi_r}(x)$ for $x > 1$.
\end{lemma}

\begin{proof}
The first assertion is precisely~\cite[Lemma 2.3]{lamzouri2016bias} translated into our notation. Similarly, the second assertion is precisely~\cite[Lemma 5]{lay2015sign}; then the third assertion follows from the other~two.
The final assertion follows from
\begin{align*}
\frac{E^{\pi_\ell}(x)-E^{\pi_r}(x)}{\sqrt x\log x} & = \bigl( \pi_\ell(x)- (\log\log x + C_0) \bigr) - \bigl( \pi_r(x)- (\log\log x - C_2) \bigr) \\
& = \pi_\ell(x)-\pi_r(x) - (C_0+C_2) \\
& =\sum_{p \le x} \sum_{k=2}^\infty \frac 1{kp^k} - \sum_{k=2}^\infty \sum_p \frac 1{kp^k} = -\sum_{k=2}^\infty \sum_{p > x} \frac 1{kp^k} < 0.
\qedhere
\end{align*}
\end{proof}

\begin{proof}[Proof of Theorem~\ref{RH theorem}(a)]
For the pairs of functions
$(f,g) \in \bigl\{ (\theta,\psi)$, $(\pi,\Pi)$, $(\psi_r,\theta_r)$, $(\Pi_r,\pi_r)$, $(\Pi_r,\pi_\ell) \bigr\}$, Lemmas~\ref{standard Es related}--\ref{reciprocal Es related 2} show that $E^g(x) - E^f(x) = 1 + o(1)$, which immediately implies that $E^f(x) < E^g(x)$ when~$x$ is sufficiently large.
\end{proof}

\subsection{Mellin transforms: preliminaries} \label{pre mellin section}

A familiar tool in analytic number theory is the Mellin transform $\int_2^\infty f(x) x^{-s-1}\,dx$ of a function $f(x)$, where the integral typically converges in some right half-plane. A famous theorem of Landau (see for example~\cite[Lemma~15.1]{montgomery_vaughan_2006}) allows us to connect oscillations of~$f(x)$ to the singularities of its Mellin transform. In preparation for carrying out this strategy in Section~\ref{mellin section} to facilitate the proof of Theorem~\ref{RH theorem}(b), we establish some preliminary integral formulas in this section. We consistently choose~$2$ as the lower endpoint of our integrals to avoid irrelevant technicalities when factors of $\frac1{\log x}$ or $\log\log x$ are present.

\begin{lemma}
\label{lemma:mellin_log}
For any complex number~$s$ with $\Re s > 1$,
\[
\int_2^\infty x^{-s}\log x \, dx = \frac{2^{1-s}}{(1-s)^2} - \frac{2^{1-s}\log 2}{1-s} .
\]
\end{lemma}

\begin{proof}
The identity follows from the substitution $v = \log x$ and integrating by parts:
\begin{align*}
\int_2^\infty x^{-s}\log x \, dx &= \int_{\log 2}^\infty v e^{(1-s)v} \, dv \\
&= v \frac{e^{(1-s)v}}{1-s} \bigg|_{\log 2}^\infty - \frac 1{1-s}\int_{\log 2}^\infty e^{(1-s)v} \, dv \\&= -(\log 2) \frac{2^{1-s}}{1-s} + \frac{2^{1-s}}{(1-s)^2}.
\qedhere
\end{align*}
\end{proof}

\begin{lemma}
\label{lemma:mellin_overlog}
For any complex number~$s$ with $\Re s > 1$,
\[
\int_2^\infty \frac{x^{-s}}{\log x} \, dx = -\log(s-1) + r_0(s)
\]
where $r_0(s)$ is an entire function with $r_0(1) = -C_0 - \log\log2$.
\end{lemma}

\begin{remark}
The proof is valid when $\Re s > 1$; what we mean by ``$r_0(s)$ is an entire function'' is that the function $r_0(s)$ implicitly defined by the formula has an analytic continuation to all of~$\C$. The analogous interpretation holds for similar statements in the rest of this section and the next.
\end{remark}

\begin{proof}
Using the substitution $u=(s-1)\log x$, we easily verify that
\begin{align*}
\int_2^\infty \frac{x^{-s}}{\log x} \, dx &= \frac{\li(2)}{s2^s} + \frac1s \int_{(s-1)\log2}^\infty \frac{e^{-u}}u \,du \\
&= \int_1^\infty \frac{e^{-u}}u \,du + \int_0^1 \frac{e^{-u}-1}u \,du \\ &\quad - \int_0^{(s-1)\log2} \frac{e^{-u}-1}u \,du + \int_{(s-1)\log2}^1 \frac{du}u \\
&= -C_0 - \int_0^{(s-1)\log2} \frac{e^{-u}-1}u \,du - \log\bigl( (s-1)\log2 \bigr),
\end{align*}
where the evalution of the first two integrals to~$-C_0$ is~\cite[equation~(7.31)]{montgomery_vaughan_2006}. This calculation establishes the lemma with
\[
r_0(s) = -C_0 - \int_0^{(s-1)\log2} \frac{e^{-u}-1}u \,du - \log\log2. \qedhere
\]
\end{proof}

\begin{lemma}
\label{lemma:mellin_li}
For any complex number~$s$ with $\Re s > 1$,
\begin{align*}
\int_2^\infty \li(x)x^{-s-1} \, dx &= -\frac 1s \log(s-1) + r_1(s) \\
\int_2^\infty \li(x)x^{-s-1}\log x \, dx &= \frac 1{s(s-1)} - \frac 1{s^2}\log(s-1) - r_\Pi(s),
\end{align*}
where $r_1(s)$ and $r_\Pi(s)$ are analytic for $\Re s>0$.
\end{lemma}

\begin{proof}
Integration by parts gives
\[
\int_2^\infty \li(x)x^{-s-1} \, dx = \frac{\li(2)}{s2^s} + \frac1s \int_2^\infty \frac{x^{-s}}{\log x} \,dx = \frac{\li(2)}{s2^s} + \frac{-\log(s-1) + r_0(s)}s
\]
by Lemma~\ref{lemma:mellin_overlog}, which establishes the first identity with
\[
r_1(s) = \frac{\li(2)}{s2^s} + \frac{r_0(s)}s.
\]
The second formula follows by differentiating both sides of the first equation with respect to~$s$, negating both sides, and setting $r_\Pi(s) = r_1'(s)$.
\end{proof}

\begin{remark}
Our proof of Lemma~\ref{lemma:mellin_li} is essentially the same as the one in~\cite[Theorem~15.2]{montgomery_vaughan_2006}. However, those authors use the convention that $\li(x) = \int_2^\infty \frac{du}{\log u}$ rather than our convention that $\li(x) = \int_0^\infty \frac{du}{\log u}$; thus we have chosen to include a complete proof (that yields an additional boundary term) to avoid confusion.
\end{remark}

\begin{lemma}
\label{lemma:mellin_loglog}
For any complex number~$s$ with $\Re s > 1$,\begin{align*}
\int_2^\infty x^{-s}\log\log x \, dx & = \frac 1{1-s} \bigl( \log(s-1) + C_0 \bigr) + r_2(s) \\
\int_2^\infty x^{-s} \log x \log\log x \, dx & = -\frac{\log (s-1)+C_0-1}{(s-1)^2} + r_3(s)
\end{align*}
where $r_2(s)$ and $r_3(s)$ are entire.
\end{lemma}

\begin{proof}
Integration by parts gives
\begin{align*}
\int_2^\infty x^{-s}\log\log x \, & dx = -\frac{2^{1-s}}{1-s} \log\log 2 - \frac 1{1-s} \int_2^\infty \frac{x^{1-s}}{x\log x} \, dx \\
&= \frac{2^{1-s}}{s-1} \log\log 2 + \frac{-\log(s-1) + r_0(s)}{s-1} \\
&= \frac{2^{1-s}-1}{s-1} \log\log 2 + \frac{-\log(s-1) - C_0}{s-1} + \frac{r_0(s) - r_0(1)}{s-1}
\end{align*}
by Lemma~\ref{lemma:mellin_overlog}, which establishes the first formula with
\[
r_2(s) = \frac{2^{1-s}-1}{s-1} \log\log 2 + \frac{r_0(s) - r_0(1)}{s-1}.
\]
The second formula follows by differentiating both sides of the first equation with respect to~$s$, negating both sides, and setting $r_3(s) = -r_2'(s)$.
\end{proof}

\subsection{Mellin transforms of number-theoretic functions} \label{mellin section}

In this section, we determine the Mellin transforms of all of the functions from Definition~\ref{defn:function_sets_zeta} and, more importantly, of all of the normalized error terms from Definition~\ref{defn:error_bias_zeta}. These Mellin transforms are needed for the proof of Theorem~\ref{RH theorem}(b) at the end of this section. One detail of note is that all of the upcoming Mellin transforms of normalized error terms have removable singularities at $s=1$; this is equivalent in essence to the fact that the main terms in Definition~\ref{defn:error_bias_zeta} have been chosen appropriately.

The usual Mellin transform identity (see~\cite[Theorem~1.3]{montgomery_vaughan_2006}, applied to any sequence~$(a_n)$ of complex numbers with $a_1=0$, is
\begin{equation} \label{usual Mellin}
\int_2^\infty \biggl( \sum_{n\le x} a_n \biggr) x^{-s-1}\,dx = \frac1s \sum_{n=1}^\infty \frac{a_n}{n^s}
\end{equation}
whenever $\Re s$ is positive and larger than the abscissa of convergence of the Dirichlet series on the right-hand side. From this formula, it is easy to derive that
\begin{equation} \label{reciprocal Mellin}
\int_2^\infty \biggl( \sum_{n\le x} \frac{a_n}n \biggr) x^{-s}\,dx = \frac1{s-1} \sum_{n=1}^\infty \frac{a_n}{n^s}.
\end{equation}

\begin{lemma} \label{psi mellin}
For any complex number~$s$ with $\Re s > 1$,
\begin{align*}
\int_2^\infty \psi(x) x^{-s-1} \, dx &= -\frac 1s \frac{\zeta'}\zeta(s) \\
\int_2^\infty E^\psi(x) x^{-s-1/2} \, dx &= -\frac 1s \frac{\zeta'}\zeta(s) + \frac{2^{1-s}}{1-s}.
\end{align*}
\end{lemma}

\begin{proof}
The first formula follows immediately from equation~\eqref{usual Mellin} with $a_n = \Lambda(n)$. We then calculate
\begin{align}
\int_2^\infty E^{\psi}(x)x^{-s-1/2} \, dx & = \int_2^\infty \frac{\psi(x)-x}{x^{1/2}} x^{-s-1/2} \, dx \label{psi to Epsi mellin} \\
& = \int_2^\infty \psi(x)x^{-s-1} \, dx - \int_2^\infty x^{-s} \, dx\notag \\&= -\frac 1s \frac{\zeta'}\zeta(s) + \frac{2^{1-s}}{1-s}. \notag
\qedhere
\end{align}
\end{proof}

\begin{lemma} \label{theta mellin}
For any complex number~$s$ with $\Re s > 1$,
\begin{align*}
\int_2^\infty \theta(x) x^{-s-1} \, dx &= -\frac 1s \frac{\zeta'}\zeta(s) + \frac1s \frac{\zeta'}\zeta(2s) + r_\theta(s) \\
\int_2^\infty E^\theta(x) x^{-s-1/2} \, dx &= -\frac 1s \frac{\zeta'}\zeta(s) + \frac1s \frac{\zeta'}\zeta(2s) + \frac{2^{1-s}}{1-s} + r_\theta(s)
\end{align*}
where $r_\theta(s)$ is analytic for $\Re s>\frac13$.
\end{lemma}

\begin{proof}
We write
\begin{equation} \label{theta mellin 1}
\int_2^\infty \theta(x) x^{-s-1} \, dx = \int_2^\infty \psi(x) x^{-s-1} \, dx - \int_2^\infty \psi(\sqrt x) x^{-s-1} \, dx + r_\theta(s)
\end{equation}
where, by equation~\eqref{theta psi 1/3},
\[
r_\theta(s) = \int_2^\infty \bigl( \theta(x) - \psi(x) + \psi(\sqrt x) \bigr) x^{-s-1} \, dx \ll \int_2^\infty x^{1/3} x^{-\sigma-1}\,dx
\]
converges to an analytic function for all $\Re s>\frac13$. The first integral on the right-hand side of equation~\eqref{theta mellin 1} is evaluated in Lemma~\ref{psi mellin}, and the second integral also follows from Lemma~\ref{psi mellin} since
\begin{equation} \label{psi sqrt}
\int_2^\infty \psi(\sqrt x) x^{-s-1} \, dx = \int_{\sqrt2}^\infty \psi(t) (t^2)^{-s-1} \cdot 2t\,dt = 2 \int_2^\infty \psi(t) t^{-(2s)-1}\,dt
\end{equation}
(again we may change the lower endpoint from~$\sqrt2$ to~$2$ since $\psi(t)=0$ for intermediate values).
The second formula follows from the first by the same calculation as in equation~\eqref{psi to Epsi mellin}.
\end{proof}

\begin{lemma} \label{Pi mellin}
For any complex number~$s$ with $\Re s > 1$,
\begin{align*}
\int_2^\infty \Pi(x) x^{-s-1} \, dx &= \frac 1s \log\zeta(s) \\
\int_2^\infty \Pi(x) x^{-s-1}\log x \, dx &= -\frac 1s \frac{\zeta'}\zeta(s) + \frac 1{s^2}\log\zeta(s) \\
\int_2^\infty E^{\Pi}(x)x^{-s-1/2} \, dx & = -\frac 1s \frac{\zeta'}\zeta(s) + \frac 1{s^2} \log\bigl( \zeta(s)(s-1) \bigr) \\ &\quad - \frac 1{s(s-1)} + r_\Pi(s)
\end{align*}
where $r_\Pi(s)$ is analytic for $\Re s>0$.
\end{lemma}

\begin{proof}
Since $\Pi(x) = \sum_{n \le x} \frac{\Lambda(n)}{\log n}$ and $\log\zeta(s) = \sum_{n=1}^\infty \frac{\Lambda(n)}{\log n} n^{-s}$ for $\Re s > 1$, the first formula follows from equation~\eqref{usual Mellin}. The second formula follows from the first by differentiating both sides with respect to~$s$ and negating both sides.
Then by Lemma~\ref{lemma:mellin_li} combined with the second formula,
\begin{align} \label{Pi to EPi mellin}
\int_2^\infty & E^{\Pi}(x)x^{-s-1/2}\, dx = \int_2^\infty \frac{\Pi(x)-\li(x)}{x^{1/2}/\log x} x^{-s-1/2} \, dx \\ 
& = \int_2^\infty \bigl(\Pi(x) x^{-s-1} \log x - \li(x) x^{-s-1} \log x \bigr) \, dx \notag \\
& = -\frac 1s \frac{\zeta'}\zeta(s) + \frac 1{s^2}\log\zeta(s) - \biggl(\frac 1{s(s-1)} - \frac 1{s^2}\log(s-1) - r_\Pi(s) \biggr) \notag \\
& = -\frac 1s \frac{\zeta'}\zeta(s) + \frac 1{s^2} \log\bigl( \zeta(s)(s-1) \bigr) - \frac 1{s(s-1)} + r_\Pi(s). \notag
\qedhere
\end{align}
\end{proof}

\begin{lemma} \label{pi mellin}
For any complex number~$s$ with $\Re s > 1$,
\begin{align*}
\int_2^\infty \pi(x) x^{-s-1} \, dx &= \frac 1s \log\zeta(s) - \frac 1{2s}\log\zeta(2s) + r_4(s) \\
\int_2^\infty \pi(x) x^{-s-1}\log x \, dx &= -\frac 1s \frac{\zeta'}\zeta(s) + \frac 1{s^2}\log\zeta(s) \\ &\qquad + \frac 1s \frac{\zeta'}\zeta(2s) - \frac 1{2s^2}\log\zeta(2s) + r_4'(s) \\
\int_2^\infty E^{\pi}(x)x^{-s-1/2} \, dx &= -\frac 1s \frac{\zeta'}\zeta(s) + \frac 1{s^2}\log\bigl(\zeta(s)(s-1)\bigr) \\
&\qquad{} + \frac 1s\frac{\zeta'}{\zeta}(2s) - \frac 1{2s^2}\log\zeta(2s) - \frac 1{s(s-1)} + r_\pi(s)
\end{align*}
where $r_4(s)$ and $r_\pi(s)$ are analytic for $\Re s > \frac13$.
\end{lemma}

\begin{proof}
We write
\begin{equation} \label{pi mellin 1}
\int_2^\infty \pi(x) x^{-s-1} \, dx = \int_2^\infty \Pi(x) x^{-s-1} \, dx - \int_2^\infty \tfrac12\Pi(\sqrt x) x^{-s-1} \, dx + r_4(s)
\end{equation}
where, by equation~\eqref{pi Pi 1/3},
\[
r_4(s) = \int_2^\infty \bigl( \pi(x) - \Pi(x) + \tfrac12\Pi(\sqrt x) \bigr) x^{-s-1} \, dx \ll \int_2^\infty \frac{x^{1/3}}{\log x} x^{-\sigma-1}\,dx
\]
converges to an analytic function for all $\Re s>\frac13$. The first integral on the right-hand side of equation~\eqref{pi mellin 1} is evaluated in Lemma~\ref{Pi mellin}, and the second integral also follows from Lemma~\ref{psi mellin} by a calculation similar to equation~\eqref{psi sqrt}. The second formula follows from the first by differentiating both sides with respect to~$s$ and negating both sides. Finally, the third formula follows from the second by the same calculation as in equation~\eqref{Pi to EPi mellin}, upon setting $r_\pi(s) = r_4'(s) - r_\Pi(s)$.
\end{proof}

\begin{lemma} \label{psi_r mellin}
For any complex number~$s$ with $\Re s > 1$,
\begin{align*}
\int_2^\infty \psi_r(x) x^{-s} \, dx &= \frac1{1-s} \frac{\zeta'}\zeta(s) \\
\int_2^\infty E^{\psi_r}(x)x^{-s-1/2} \, dx &= \frac 1{1-s} \frac{\zeta'}\zeta(s) - \frac{2^{1-s}}{(1-s)^2} + \frac{2^{1-s}(\log 2-C_0)}{1-s}.
\end{align*}
\end{lemma}

\begin{proof}
The first formula follows immediately from equation~\eqref{reciprocal Mellin} with $a_n = \Lambda(n)$. Then
\begin{align*}
\int_2^\infty E^{\psi_r}(x) x^{-s-1/2} & \, dx = \int_2^\infty (\psi_r(x) - \log x + C_0)\sqrt x\cdot x^{-s-1/2} \, dx \\
& = \int_2^\infty \psi_r(x)x^{-s} \,dx - \int_2^\infty x^{-s}\log x \,dx + \int_2^\infty C_0 x^{-s} \, dx \\
& = \frac 1{1-s} \frac{\zeta'}\zeta(s) - \biggl( \frac{2^{1-s}}{(1-s)^2} - \frac{2^{1-s}\log 2}{1-s} \biggr) - \frac{C_0 2^{1-s}}{1-s}
\end{align*}
by Lemma~\ref{lemma:mellin_log}.
\end{proof}

\begin{lemma} \label{theta_r mellin}
For any complex number~$s$ with $\Re s > 1$,
\begin{align}
\int_2^\infty \theta_r(x) x^{-s} \, dx &= \frac 1{1-s} \frac{\zeta'}\zeta(s) - \frac 1{1-s} \frac{\zeta'}\zeta(2s) - \frac{r_5(s)}{1-s} \notag \\
\int_2^\infty E^{\theta_r}(x)x^{-s-1/2} \, dx &= \frac 1{1-s} \frac{\zeta'}\zeta(s) - \frac 1{1-s} \biggl( \frac{\zeta'}\zeta(2s) - \frac{\zeta'}\zeta(2) \biggr) \notag \\
&\qquad{} - \frac{2^{1-s}}{(1-s)^2} + \frac{2^{1-s}(\log 2-C_0)}{1-s} + r_{\theta_r}(s),
\label{r theta r formula}
\end{align}
where $r_5(s)$ and $r_{\theta_r}(s)$ are analytic for $\Re s > \frac 13$.
\end{lemma}

\begin{proof}
We begin by defining
\begin{align} \label{r5 def}
r_5(s) &= \sum_p \frac{\log p}{p^s} - \biggl( -\frac{\zeta'}\zeta(s) + \frac{\zeta'}\zeta(2s) \biggr) \\
&= \sum_p \frac{\log p}{p^s} - \sum_{p^k} \frac{\log p}{p^{ks}} + \sum_{p^k} \frac{\log p}{p^{2ks}} = -\sum_{\substack{k\ge3 \\ k\text{ odd}}} \sum_p \frac{\log p}{p^{ks}}, \notag
\end{align}
which converges to an analytic function for $\Re s>\frac13$.
Equation~\eqref{reciprocal Mellin} then gives
\begin{equation}
\label{eqn:theta_r_partial}
\int_2^\infty \theta_r(x) x^{-s} \,dx = \frac 1{s-1} \sum_p \frac{\log p}{p^s} = \frac 1{1-s} \frac{\zeta'}\zeta(s) - \frac 1{1-s} \frac{\zeta'}\zeta(2s) - \frac {r_5(s)}{1-s}.
\end{equation}
By Lemma~\ref{lemma:mellin_log} it follows that
\begin{align*}
\int_2^\infty & E^{\theta_r}(x) x^{-s-1/2} \, dx
= \int_2^\infty \bigl(\theta_r(x) - (\log x - C_1) \bigr) \sqrt x \cdot x^{-s-1/2} \, dx \\
& = \int_2^\infty \theta_r(x) x^{-s} \, dx - \int_2^\infty x^{-s}\log x \, dx + \int_2^\infty C_1 x^{-s} \, dx \\
& = \frac 1{1-s} \frac{\zeta'}\zeta(s) - \frac 1{1-s} \frac{\zeta'}\zeta(2s) - \frac {r_5(s)}{1-s} - \biggl(\frac{2^{1-s}}{(1-s)^2} - \frac{2^{1-s}\log 2}{1-s}\biggr) \\ &\qquad{} - \frac{C_1 2^{1-s}}{1-s} \\
& = \frac 1{1-s} \frac{\zeta'}\zeta(s) - \frac 1{1-s} \frac{\zeta'}\zeta(2s) - \frac {r_5(s)}{1-s} - \frac{2^{1-s}}{(1-s)^2} + \frac{2^{1-s}\log 2}{1-s} \\
&\qquad{}- \frac{2^{1-s}}{1-s} \biggl( C_0 + \sum_p \frac{\log p}{p(p-1)} \biggr)
\end{align*}
by Definition~\ref{defn:error_bias_zeta}. If we now define
\[
r_{\theta_r}(s) = -\frac 1{1-s} \biggl( \frac{\zeta'}\zeta(2)+r_5(s)+2^{1-s}\sum_p \frac{\log p}{p(p-1)} \biggr),
\]
then the identity~\eqref{r theta r formula} holds; it remains to show that $r_{\theta_r}(s)$ is analytic for $\Re s>\frac13$, for which it suffices to show that the quantity in parentheses vanishes at $s=1$.
But by the definition~\eqref{r5 def} of $r_5(s)$,
\begin{align*}
\frac{\zeta'}\zeta(2s)+r_5(s) & = -\sum_{k=2}^\infty \sum_p \frac{\log p}{p^{ks}} = -\sum_p \log p \sum_{k=2}^\infty p^{-ks} = -\sum_p \frac {\log p}{p^s(p^s-1)},
\end{align*}
and therefore
\[
\frac{\zeta'}\zeta(2)+r_5(1)+2^{1-1}\sum_p \frac{\log p}{p(p-1)} = \lim_{s\to1} \biggl( \frac{\zeta'}\zeta(2s)+r_5(s) \biggr) + \sum_p \frac{\log p}{p(p-1)} = 0
\]
as required.
\end{proof}

\begin{lemma} \label{Pi_r mellin}
For any complex number~$s$ with $\Re s > 1$,
\begin{align*}
\int_2^\infty \Pi_r(x) x^{-s} \, dx &= \frac1{s-1}\log\zeta(s) \\
\int_2^\infty \Pi_r(x) x^{-s}\log x \, dx &= -\frac 1{s-1} \frac{\zeta'}\zeta(s) + \frac 1{(s-1)^2}\log\zeta(s) \\
\int_2^\infty E^{\Pi_r}(x)x^{-s-1/2} \, dx &= -\frac 1{s-1} \frac{\zeta'}\zeta(s) + \frac{\log\bigl( \zeta(s)(s-1) \bigr) - 1}{(s-1)^2} + r_{\Pi_r}(s)
\end{align*}
where $r_{\Pi_r}(s)$ is entire.
\end{lemma}

\begin{proof}
Since $\Pi(x) = \sum_{n \le x} \frac{\Lambda(n)}{\log n}$ and $\log\zeta(s) = \sum_{n=1}^\infty \frac{\Lambda(n)}{\log n} n^{-s}$ for $\Re s > 1$, the first formula follows from equation~\eqref{reciprocal Mellin}. The second formula follows from the first by differentiating both sides with respect to~$s$ and negating both sides. Then by Lemmas~\ref{lemma:mellin_loglog} and~\ref{lemma:mellin_log} combined with the second formula,
\begin{align*}
\int_2^\infty & E^{\Pi_r}(x) x^{-s-1/2} \, dx = \int_2^\infty \bigl( \Pi_r(x) - (\log\log x + C_0) \bigr) \sqrt x\log x \cdot x^{-s-1/2} \, dx \\
& = \int_2^\infty \Pi_r(x)x^{-s}\log x \, dx - \int_2^\infty x^{-s}\log x\log\log x \, dx \\ &\qquad{} - \int_2^\infty C_0 x^{-s}\log x \, dx \\
& = -\frac 1{s-1} \frac{\zeta'}\zeta(s) + \frac 1{(s-1)^2}\log\zeta(s) \\
&\qquad{}- \biggl( -\frac{\log (s-1)+C_0-1}{(s-1)^2} + r_3(s) \biggr) - C_0 \biggl( \frac{2^{1-s}}{(1-s)^2} - \frac{2^{1-s}\log 2}{1-s} \biggr) \\
& = -\frac 1{s-1} \frac{\zeta'}\zeta(s) + \frac{\log\bigl( \zeta(s)(s-1) \bigr) - 1}{(s-1)^2} + r_{\Pi_r}(s),
\end{align*}
where one can verify that
\[
r_{\Pi_r}(s) = C_0 \biggl( \frac{1-2^{1-s}}{(1-s)^2} + \frac{2^{1-s}\log 2}{1-s} \biggr) - r_3(s)
\]
has a removable singularity at $s=1$.
\end{proof}

\begin{lemma} \label{pi_r mellin}
For any complex number~$s$ with $\Re s > 1$,
\begin{align}
\int_2^\infty \pi_r(x) x^{-s}\log x \, dx &= \frac 1{(1-s)^2} \log\zeta(s) + \frac1{1-s} \frac{\zeta'}\zeta(s) \notag \\
&\qquad{} - \frac 1{2(1-s)^2} \log\zeta(2s) - \frac 1{1-s} \frac{\zeta'}\zeta(2s) \notag \\
&\qquad{} - \frac{C_0-\frac 12\log 2 + C_22^{1-s}}{(1-s)^2} + \frac{C_22^{1-s}\log 2}{1-s} \notag \\ &\qquad{} + r_{\pi_r}(s)+r_3(s). \notag \\
\int_2^\infty E^{\pi_r}(x)x^{-s-1/2} \, dx &= \frac 1{1-s}\frac{\zeta'}\zeta(s) + \frac {\log(\zeta(s)(s-1))-1}{(1-s)^2} \notag \\
&\qquad{}- \frac 1{2(1-s)^2} \log\frac{\zeta(2s)}{\zeta(2)} - \frac 1{1-s} \frac{\zeta'}\zeta(2s) + r_{\pi_r}(s) , \label{r pi r formula}
\end{align}
where $r_3(s)$ (from Lemma \ref{lemma:mellin_loglog}) is entire and $r_{\pi_r}(s)$ is analytic for $\Re s > \frac13$.
\end{lemma}

\begin{proof}
We begin by defining
\begin{align} \label{r6 def}
r_6(s) &= \sum_p \frac1{p^s} - \biggl( \log\zeta(s) - \frac12\log\zeta(2s) \biggr) \\
&= \sum_p \frac1{p^s} - \sum_{p^k} \frac1{kp^{ks}} + \frac12 \sum_{p^k} \frac1{kp^{2ks}} = -\sum_{\substack{k\ge3 \\ k\text{ odd}}} \sum_p \frac1{kp^{ks}}, \notag
\end{align}
which converges to an analytic function for $\Re s>\frac13$.
Equation~\eqref{reciprocal Mellin} then gives
\begin{equation*}
\int_2^\infty \pi_r(x) x^{-s} \,dx = \frac 1{s-1} \sum_p \frac1{p^s} = \frac 1{s-1} \log\zeta(s) - \frac 1{2(s-1)} \log\zeta(2s) + \frac {r_6(s)}{s-1},
\end{equation*}
and differentiating with respect to~$s$ and negating both sides yields
\begin{multline*}
\label{eqn:pi_r_partial}
\int_2^\infty \pi_r(x) x^{-s}\log x \,dx = \frac 1{(s-1)^2} \log\zeta(s) - \frac1{s-1} \frac{\zeta'}\zeta(s) \\
- \frac 1{2(s-1)^2} \log\zeta(2s) + \frac 1{s-1} \frac{\zeta'}\zeta(2s) + \frac {r_6(s)}{(s-1)^2} - \frac {r_6'(s)}{s-1}.
\end{multline*}
By Lemmas~\ref{lemma:mellin_loglog} and~\ref{lemma:mellin_log}, it follows that
\begin{align*}
&\int_2^\infty E^{\pi_r}(x)x^{-s-1/2} \, dx
= \int_2^\infty \bigl(\pi_r(x)-\log\log x+C_2\bigr)\sqrt x \cdot x^{-s-1/2} \, dx \\
& = \int_2^\infty \pi_r(x)x^{-s}\log x \, dx - \int_2^\infty x^{-s}\log x\log\log x \, dx + C_2\int_2^\infty x^{-s}\log x \, dx \\
&= \frac{\log\zeta(s)}{(1-s)^2} + \frac1{1-s} \frac{\zeta'}\zeta(s)
- \frac{\log\zeta(2s)}{2(1-s)^2} - \frac 1{1-s} \frac{\zeta'}\zeta(2s) + \frac {r_6(s)}{(1-s)^2} + \frac {r_6'(s)}{1-s} \\
&\quad{}- \biggl( -\frac{\log (s-1)+C_0-1}{(1-s)^2} + r_3(s) \biggr) + C_2 \biggl( \frac{2^{1-s}}{(1-s)^2} - \frac{2^{1-s}\log 2}{1-s} \biggr) \\
&= \frac{\log \bigl( \zeta(s)(s-1) \bigr) -1}{(1-s)^2} + \frac1{1-s} \frac{\zeta'}\zeta(s)
- \frac1{2(1-s)^2} \log \frac{\zeta(2s)}{\zeta(2)} - \frac 1{1-s} \frac{\zeta'}\zeta(2s) \\
&\quad{} + \frac {r_6(s)}{(1-s)^2} + \frac {r_6'(s)}{1-s}+ \frac{C_0-\frac12\log2}{(1-s)^2} - r_3(s) + C_2 \biggl( \frac{2^{1-s}}{(1-s)^2} - \frac{2^{1-s}\log 2}{1-s} \biggr) .
\end{align*}
If we now define
\begin{multline*}
r_{\pi_r}(s) = \frac1{(1-s)^2} \biggl( r_6(s) + C_0 - \frac12\log2 + C_2 2^{1-s} \\ + (1-s) \bigl( r_6'(s) - C_2 2^{1-s}\log 2 \bigr) \biggr) - r_3(s),
\end{multline*}
then the identity~\eqref{r pi r formula} holds; it remains to show that $r_{\pi_r}(s)$ is analytic for $\Re s>\frac13$, for which it suffices to show that the quantity in parentheses has a double zero at $s=1$. First, plugging in $s=1$ yields the value
\[
r_6(1) + C_0 - \frac12\log2 + C_2 = r_6(1) - \frac12\log2 + \sum_{k=2}^\infty \sum_p \frac1{kp^k},
\]
whose vanishing follows from equation~\eqref{r6 def} written in the form
\[
r_6(s) - \frac12\log\zeta(2s) = \sum_p \frac1{p^s} - \log\zeta(s) = -\sum_{k\ge2} \sum_p \frac1{kp^{ks}}.
\]
Second, taking the derivative of the expression in parentheses and plugging in $s=1$ yields
\[
r_6'(1) - C_2 2^{1-1}\log2 - r_6'(1) + C_2 2^{1-1}\log 2 = 0
\]
as required.
\end{proof}



\begin{lemma}
\label{lemma:pi_ell_bound_argument}
Suppose $f \in \B\cup\B_r$, and for any $X,L \in \R$, there
exists $x,y \in \R$ with $x \ge X$, $y \ge X$ such that $E^{\Pi_r}(x)-E^f(x) > L$ and $E^{\pi_r}(y)-E^f(y) < L$. Then it holds that $E^{\pi_\ell}(x)-E^f(x) > 0$ and $E^{\pi_\ell}(y)-E^f(y) < 0$.
\end{lemma}

\begin{proof}
By Lemma \ref{reciprocal Es related 2} there exists some $X_0$ such that for any real $z \ge X_0$,
$E^{\Pi_r}(z) < E^{\pi_\ell}(z) < E^{\pi_r}(z)$ and thus $E^{\Pi_r}(z)-E^f(z) < E^{\pi_\ell}(z)-E^f(z) < E^{\pi_r}(z)-E^f(z)$. By the assumption of this lemma, for any real $X$, one may select $x,y \ge \max\{X,X_0\}$ such that $E^{\Pi_r}(x)-E^f(x) > L$ and $E^{\pi_r}(y)-E^f(y) < L$.
As $x,y \ge X_0$, this implies $E^{\pi_\ell}(x)-E^f(x) > E^{\Pi_r}(x)-E^f(x) > L$ and $E^{\pi_\ell}(y) - E^f(y) < E^{\pi_r}(y)-E^f(y) < L$. This completes the proof.
\end{proof}

\begin{proof}[Proof of Theorem~\ref{RH theorem}(b)]
First, assume that RH is true, and choose any functions~$f$ and~$g$ in $\B\cup\B_r$ with $\beta_f < \beta_g$. Propositions~\ref{corollary:explicit_standards} and~\ref{corollary:explicit_reciprocals} show that $E^g(x) - E^f(x) = \beta_g - \beta_f + o(1) \ge 1 + o(1)$, which immediately implies that $E^f(x) < E^g(x)$ when~$x$ is sufficiently large.

Now, assume that RH is false. We first work out the special case $f=\theta$ and $g=\psi_r$, and afterwards we indicate how the proof extends to all other cases except those involving $g=\pi_\ell$.
By Lemmas~\ref{theta mellin} and~\ref{psi_r mellin}, for any real number~$L$ and any complex number~$s$ with $\Re s > 1$,
\begin{align} \label{first landau}
\int_2^\infty & \bigl( E^{\theta}(x) - E^{\psi_r}(x) + L \bigr) x^{-s-1/2} \, dx \notag \\ & = \biggl( -\frac 1s \frac{\zeta'}\zeta(s) + \frac1s \frac{\zeta'}\zeta(2s) + \frac{2^{1-s}}{1-s} + r_\theta(s) \biggr)
\\ &\qquad{} - \biggl( \frac 1{1-s} \frac{\zeta'}\zeta(s) - \frac{2^{1-s}}{(1-s)^2} + \frac{2^{1-s}(\log 2-C_0)}{1-s} \biggr) + \frac{2^{1/2-s}L}{s-\frac12}. \notag
\end{align}
We can verify that the right-hand side has a removable singularity at $s=1$ and no real singularities for $s>\frac12$. Moreover, the right-hand equals $\frac1{s(s-1)} \frac{\zeta'}\zeta(s)$ plus a function with no singularities in the half-plane $\Re s>\frac12$ except at $s=1$; in particular, since we assuming that RH is false, the right-hand side has non-real singularities in the half-plane $\Re s>\frac12$.

Since the analytic continuation of the left-hand side of equation~\eqref{first landau} has non-real singularities to the right of all its real singularities, we conclude from Landau's theorem~\cite[Lemma~15.1]{montgomery_vaughan_2006} that the integrand cannot eventually be of constant sign. In particular, the difference $E^{\theta}(x) - E^{\psi_r}(x)$ is neither bounded above nor bounded below.

To make this proof work, the key properties we needed the right-hand side to satisfy are:
\begin{itemize}
\item it has a meromorphic continuation to $\Re s>\frac13$,
\item it has a removable singularity at $s=1$ and no real singularities for $s>\frac12$;
\item it has non-removable singularities at the zeros of $\zeta(s)$.
\end{itemize}
If we choose any pair of distinct functions $f,g\in\B\cup\B_r\setminus\{\pi_\ell\}$, other than the five pairs in equation~\eqref{not these five}, we can use the lemmas in this section to verify that all three of these properties hold for $\int_2^\infty \bigl( E^{f}(x) - E^{g}(x) + L \bigr) x^{-s-1/2} \, dx$. (Indeed, the assumption $\beta_f < \beta_g$ is not even needed for this part of the proof.) The same Landau's theorem argument then finishes the proof of all cases of the theorem.

One further instructive example is $f=\psi_r$ and $g=\Pi$, in which case the relevant terms in the Mellin transform of $E^{\psi_r}(x) - E^{\Pi}(x) + L$ are
\begin{multline*}
\frac 1{1-s} \frac{\zeta'}\zeta(s) -  \biggl( -\frac 1s \frac{\zeta'}\zeta(s) + \frac 1{s^2} \log\bigl( \zeta(s)(s-1) \bigr) \biggr) \\ = \frac1{s(1-s)} \frac{\zeta'}\zeta(s) - \frac 1{s^2} \log\bigl( \zeta(s)(s-1) \bigr).
\end{multline*}
Both terms have singularities at the zeros of $\zeta(s)$; the fact that the singularity of the difference is non-removable follows from the observation that the first term has poles at each zero while the second term has a logarithmic singularity at each zero.

For $f,g\in\B\cup\B_r$ other than the five pairs in equation~\eqref{not these five}, the above proof can be extended to pairs that include $\pi^\ell$ by the observation it holds for all $f,g\in\B\cup\B_r\setminus\{\pi_\ell\}$ (except the aforementioned five pairs) and the application of Lemma \ref{lemma:pi_ell_bound_argument} to every $f \in \B\cup\B_r$ except for $\Pi_r$, $\pi_\ell$, and $\pi_r$.
\end{proof}

\section{Fourier transforms of distributions and densities} \label{joint dist section}

In this section, we record the Fourier transforms of the nine normalized error terms from Definition~\ref{defn:error_bias_zeta}, both individually and jointly, as well as the Fourier transforms (characteristic functions) of the random-variable distributions~$\mu_1$ and~$\mu_2$ from Definitions~\ref{eta1 def} and~\ref{eta2 def}. We then establish properties of these distributions that allow us to prove Theorems~\ref{theorem:log_density_zeta}--\ref{log_density part 2} and~\ref{theorem:joint_dist_zeta}.

\subsection{Fourier transforms}

Our first lemma allows us to use Fourier transforms to identify when a random variable on a high-dimensional space is really an affine image of a random variable on a lower-dimensional subspace.

\begin{lemma}
\label{lemma:simplify_fourier1}
Let $D \ge d \ge 1$ be integers. Let $\mathbf X = ( X_1,\ldots,X_D )$ be a random variable on~$\R^D$ whose distribution~$\mu$ has Fourier transform $\hat \mu(t_1,\ldots,t_D)$ and let~$\mathbf Y = ( Y_1,\ldots,Y_d )$ be a random variable on~$\R^d$ whose distribution~$\nu$ has Fourier transform $\hat \nu(t_1,\ldots,t_d)$.

Suppose there exists a partition $A_1,\ldots,A_d$ of $\{1,\ldots,D\}$, and real vectors $\vec c = (c_1,\ldots,c_D)$ and $\vec r = (r_1,\ldots,r_D)$ with $r_k \ne 0$ for all $1 \le k \le D$, such that the identity
\[
\hat\mu(t_1,\ldots,t_D) = \exp\biggl( -i\sum_{k=1}^D c_kt_k \biggr) \hat\nu\biggl(\sum_{k \in A_1} r_kt_k, \ldots, \sum_{k \in A_d} r_kt_k\biggr)
\]
holds for all $(t_1,\ldots,t_D) \in \R^D$. Then~$\mathbf X$ has the same distribution as the random variable
\begin{equation} \label{new rv}
\sum_{j=1}^d \sum_{k \in A_j} (r_kY_j+c_k) \mathbf e_k ,
\end{equation}
where $\{\mathbf e_1,\ldots,\mathbf e_D\}$ is the standard basis for~$\R^D$.
\end{lemma}

\begin{proof}
By definition, the Fourier transform of the random variable~\eqref{new rv} is
\begin{align*}
\Ex \exp\biggl(-i \sum_{j=1}^d\sum_{k \in A_j} & t_k (r_kY_j+c_k) \biggr) \\
&= \exp\biggl( -i\sum_{j=1}^D c_kt_k \biggr) \Ex \exp\biggl(-i \sum_{j=1}^d\sum_{k \in A_j} t_k r_k Y_j \biggr) \\
& = \exp\biggl( -i\sum_{j=1}^D c_kt_k \biggr) \hat\nu\biggl(\sum_{k \in A_1} r_kt_k, \ldots, \sum_{k \in A_d} r_kt_k\biggr) \\
&= \hat\mu(t_1,\ldots,t_D),
\end{align*}
and therefore the two random variables have the same distribution.
\end{proof}

Next we give a criterion for the absolute continuity of a probability measure based on its Fourier transform.

\begin{definition} \label{J0 definition}
The Bessel function of the first kind of order~$0$ is the even function from~$\R$ to~$\R$ defined by
\[
J_0(x) = \frac1\pi \int_0^\pi e^{ix\cos\theta} \,d\theta = \sum_{n=0}^\infty \frac{(-1)^n}{4^n} \frac{x^{2n}}{n!^2}
\]
(see~\cite[item~9.1.21]{AS} and~\cite[equation~(1.71.1)]{Szego}). The defining integral shows that $|J_0(x)| \le 1$, and in fact the stronger bound
\begin{equation} \label{Bessel bound}
|J_0(x)| \le \min\biggl\{1, \sqrt{\frac 2{\pi|x|}}\biggr\}
\end{equation}
holds for all $x\in\R$ (see~\cite[Theorem~7.31.2]{Szego}).
\end{definition}

\begin{lemma}
\label{lemma:bessel_product_continuity}
Let $\mu$ be a probability measure on $\R^d$ with Fourier transform
\[ \hat\mu(\vec t\,\,) = \prod_{m=1}^\infty J_0(|\vec r_m \cdot \vec t\,\,|) \]
for some sequence~$(\vec r_m)$ of vectors in~$\C^d$.
The measure~$\mu$ is absolutely continuous with respect to Lebesgue measure on~$\R^d$ if there exist $2d+1$ disjoint sets $S_1,\ldots,S_{2d+1}$ of indices such that for $1 \le k \le 2d+1$,
\[
\mathrm{Span}(\{ \Re \vec r_m \colon m \in S_k \} \cup \{ \Im \vec r_m \colon m \in S_k \}) = \R^d.
\]
In particular,~$\mu$ is absolutely continuous with respect to Lebesgue measure on~$\R^d$ if either
\begin{enumerate}[label={\rm(\alph*)}]
\item $d = 1$ and at least three elements of the sequence $\vec r_m$ are nonzero, or
\item $d = 2$ and $\{\Re \vec r_m, \Im \vec r_m\}$ is a basis for~$\R^2$ for at least five elements of the sequence~$(\vec r_m)$.
\end{enumerate}
\end{lemma}

\begin{proof}
Parts (a) and (b) are special cases of the general statement (simply take the $S_k$ to be singleton sets containing the relevant elements~$r_m$). We prove the general statement using the fact (see~\cite[Lemma~A.6(b)]{IPNR}) that a probability measure~$\mu$ is absolutely continuous with respect to Lebesgue measure on~$\R^d$ (which we denote by~$\lambda_d$) if $\int_{\R^d} |\hat\mu(\vec t\,)| \, d\lambda(\vec t\,)$ converges.

By assumption, for each $k$ the set $\{ \Re \vec r_m \colon m \in S_k \} \cup \{ \Im \vec r_m \colon m \in S_k \}$ contains a basis~$B_k$ for~$\R^d$. Let $\mathbf M_k$ be the $d \times d$ matrix whose rows are the elements of~$B_k$, so that~$\mathbf M_k$ is invertible.
On one hand, the elements of $\mathbf M_k \vec t\,$ are $\{\vec b \cdot \vec t\,\colon \vec b\in B_k\}$, and so $|\mathbf M_k\vec t\,|^2 = \sum_{\vec b \in B_k} (\vec b \cdot \vec t\,)^2$; on the other hand, $|\mathbf M_k \vec t\,| \ge K_k |\vec t\,|$ for some constant~$K_k$ (namely the reciprocal of the operator norm of~$\mathbf M_k^{-1}$). It follows that there must be at least one $\vec b \in B_k$ such that $(\vec b \cdot \vec t\,)^2 \ge K_k^2|\vec t\,|^2/d$, and thus at least one $m \in S_k$ such that $|\vec r_m \cdot \vec t\,| \ge K_k|\vec t\,|/\sqrt d$. Consequently, by the inequality~\eqref{Bessel bound},
\begin{align*}
|\hat\mu(\vec t\,)| = \prod_{m=1}^\infty |J_0(|\vec r_m \cdot \vec t\,|)| &\le \prod_{k=1}^{2d+1} \prod_{m \in S_k} |J_0(\vec r_m \cdot \vec t\,)|\\&\le \prod_{k=1}^{2d+1} \min\biggl\{ 1, \sqrt{\frac {2\sqrt d}{\pi K_k|\vec t\,|}} \biggr\}\\& \le \min\{1,K|\vec t\,|^{-d-1/2}\}
\end{align*}
with $K=(\frac2\pi)^{d+1/2} d^{d/1/4} \min\{K_1,\dots,K_{2d+1}\}^{-d-1/2}$.
Finally, if $\mathcal B$ is the unit ball in~$\R^d$ centered at the origin and $\mathcal B^C$ its complement, then
\[ \int_{\R^d} |\hat\mu(\vec t\,)| \, d\lambda(\vec t\,) \le \int_{\mathcal B} 1 \, d\lambda(\vec t\,) + \int_{\mathcal B^C} K|\vec t\,|^{-d-1/2} \, d\lambda(\vec t\,) < \infty . \]
which confirms that~$\mu$ is absolutely continuous with respect to Lebesgue measure on~$\R^d$.
\end{proof}




\subsection{Properties of the distribution~\texorpdfstring{$\bm{\mu_2}$}{mu2}} \label{mu2 properties section}

In this section, we formally connect the normalized error terms in Definition~\ref{defn:error_bias_zeta} to the distributions~$\mu_1$ and~$\mu_2$ from Definitions~\ref{eta1 def} and~\ref{eta2 def}, as well as establishing some useful properties of these distributions, starting with their Fourier transforms.

\begin{lemma}
\label{note:common_measures}
The Fourier transform of $\mu_1$ is
\[
\hat\mu_1(t) = \prod_{m=1}^\infty J_0\Biggl(\frac {2t}{\sqrt{\frac 14+\gamma_m^2}}\Biggr).
\]
The Fourier transform of $\mu_2$ is
\[
\hat\mu_2(t_1,t_2) = \prod_{m=1}^\infty J_0\left(\frac {|2\gamma_m(t_1+t_2)+i(t_1-t_2)|}{\frac 14+\gamma_m^2}\right).
\]
\end{lemma}

\begin{proof}
If $Z$ is a random variable uniformly distributed on the complex unit circle and $\vec \alpha \in \C^k$, then the Fourier transform of $2\Re(Z\vec\alpha)$, for any real vector~$\vec t\,$, is by definition.
\[
\Ex e^{-i\vec t\, \cdot 2\Re(Z \vec \alpha)} = \Ex e^{-i\Re(2(\vec t\, \cdot \vec \alpha)Z)} = \int_0^1 \exp\Bigl( -i\Re\bigl( 2(\vec t\, \cdot \vec \alpha)e^{2\pi i\theta} \bigr) \Bigr) \,d\theta = J_0(2|\vec t\, \cdot \vec \alpha|),
\]
where the last equality follows from a standard calculation (see for example \cite[starting from equation~(5.1)]{akbary2013limiting}). Given an independent collection of many such random variables, the Fourier transform of their sum is simply the product of their individual Fourier transforms. This establishes the lemma both for~$\mu_1$, for which each~$\vec \alpha$ is of the form $1/(\frac12+i\gamma_m) \in \C^1$, and for~$\mu_2$, for which each~$\vec \alpha$ is of the form $\bigl( 1/(\frac12+i\gamma_m), 1/(-\frac12+i\gamma_m) \bigr) \in \C^2$. (In both cases, the square-summability of the series of coefficients implies both the almost-sure convergence of the overall random variable and the convergence of the infinite product that forms the Fourier transform.)
\end{proof}


\begin{corollary}\label{nine mu1s cor}
Assume RH and LI.
For every $f\in\B\cup\B_r$, the normalized error term $E^f(x)$ has a limiting logarithmic distribution equal to~$\mu_1+\beta_f$.
\end{corollary}

\begin{proof}
The claim follows immediately by applying~\cite[Theorems~1.4 and~1.9]{akbary2013limiting} to the nine explicit formulas given in Propositions~\ref{corollary:explicit_standards} and~\ref{corollary:explicit_reciprocals}, and confirming that in all cases the result matches the Fourier transform of~$\mu_1$ computed in Lemma~\ref{note:common_measures}.
\end{proof}

\begin{remark} \label{coincidence remark}
Each of those nine explicit formulas is a sum of terms of the form $2x^{i\gamma}$ weighted by a coefficient. For the four standard functions in Proposition~\ref{corollary:explicit_standards} that coefficient is $1/\rho = 1/(\frac12+i\gamma)$, while for the five reciprocal functions in Proposition~\ref{corollary:explicit_reciprocals} that coefficient is $1/(\rho-1) = 1/(-\frac12+i\gamma)$. The coincidence that $\bigl| 1/(\frac12+i\gamma) \bigr| = \bigl| 1/(-\frac12+i\gamma) \bigr|$ is what causes these nine limiting logarithmic distributions to be identical up to the bias factors. Furthermore, the fact that $1/(\frac12+i\gamma)$ and $1/(-\frac12+i\gamma)$ are very close (when~$\gamma$ is large) but not equal is what causes the two coordinates of~$\mu_2$ to be highly but not perfectly correlated.
\end{remark}

\begin{lemma}
\label{lemma:mu_one_continuity}
The distribution~$\mu_2$ is absolutely continuous with respect to Lebesgue measure on~$\R^2$.
\end{lemma}

\begin{proof}
Lemma~\ref{note:common_measures} can be written as
\[
\hat\mu_2(\vec t\,) = \prod_{m=1}^\infty J_0\bigl( |\vec r_m \cdot \vec t\,| \bigr)
\quad\text{where}\quad
\vec r_m = \frac {2\gamma_m}{1/4+\gamma_m^2}(1,1) + \frac i{1/4+\gamma_m^2}(1,-1) ;
\]
since the real and imaginary parts of each~$\vec r_m$ form a basis for~$\R^2$, the absolute continuity of~$\mu_2$ follows from Lemma~\ref{lemma:bessel_product_continuity}.
\end{proof}

Our last goal of this section is to determine the precise support of~$\mu_2$, for which we need a preliminary lemma.

\begin{lemma} \label{greedy lemma}
Assume RH. For any real number $0 \le v < \frac w2$, there exists a set $M\subset\N$ of positive integers, containing only finitely many pairs of consecutive integers, such that $\sum_{m\in M} 1/(\frac14+\gamma_m^2) = v$.
\end{lemma}

\begin{proof}
We prove that the greedy algorithm produces such a subset $M=\{m_1,m_2,\dots\}$: given $m_1,\dots,m_{k-1}$, let $m_k$ be the smallest integer such that $\sum_{j=1}^k 1/(\frac14+\gamma_{m_j}^2) \le v$. Since $\sum_{m\in\N} 1/(\frac14+\gamma_m^2) = \frac w2$ by equation~\eqref{same comp}, this greedy algorithm produces a finite or infinite set~$M$ such that $\sum_{m\in M} 1/(\frac14+\gamma_m^2) = v$; the fact that equality holds when~$M$ is infinite follows in a standard way from the fact that each $1/(\frac14+\gamma_m^2)$ is smaller than $\sum_{j=m+1}^\infty 1/(\frac14+\gamma_m^2)$. The same fact shows that~$M$ cannot contain every integer from some point onwards, since the greedy algorithm would have actually selected the last integer not in~$M$.

To argue that~$M$ contains only finitely many pairs of consecutive integers, it thus suffices to show that the set of integers~$m$ with $m-1\notin M$ and $m,m+1\in M$ is bounded above. Indeed, for any such~$m$, we must have
\[
\frac1{\frac14+\gamma_m^2} + \frac1{\frac14+\gamma_{m+1}^2} \le t - \sum_{j\in M\cap\{1,\dots,m-1\}} \frac1{\frac14+\gamma_j^2} < \frac1{\frac14+\gamma_{m-1}^2}.
\]
But the classic counting function for the zeros of $\zeta(s)$ implies that $\gamma_m = \frac{2\pi m}{\log m} \bigl( 1 + O(\frac1{\log m}) \bigr)$, and therefore we would have
\begin{multline*}
\frac{\log^2m}{4\pi^2m^2} \biggl( 1 + O\Bigl( \frac1{\log m} \Bigr) \biggr) + \frac{\log^2(m+1)}{4\pi^2(m+1)^2} \biggl( 1 + O\Bigl( \frac1{\log m} \Bigr) \biggr) \\ < \frac{\log^2(m-1)}{4\pi^2(m-1)^2} \biggl( 1 + O\Bigl( \frac1{\log m} \Bigr) \biggr),
\end{multline*}
which is false when~$m$ is sufficiently large.
\end{proof}

\begin{lemma}
\label{lemma:mu_one_max_support}
Assume RH. The support of the distribution~$\mu_2$ is equal to the diagonal strip $\{ (x,y) \in \R^2 \colon |x-y| \le w \}$, where~$w$ is as in equation~\eqref{w def}.
\end{lemma}

\begin{proof}
By Definition~\ref{eta2 def}, the absolute difference between the two components of~$V_2$ is
\begin{align*}
\biggl| 2\Re\sum_{m=1}^\infty Z_m \biggl( \frac 1{\frac12+i\gamma_m}-\frac 1{-\frac12+i\gamma_m} \biggr) \biggr| &= \biggl| 2\Re\sum_{m=1}^\infty \frac{Z_m}{\frac14+\gamma_m^2} \biggr|\\& \le 2 \sum_{m=1}^\infty \frac1{\frac14+\gamma_m^2} = \sum_\rho \frac 1{\rho(1-\rho)} = w
\end{align*}
by Lemma~\ref{lemma:rho_squared} (this is the same computation as in equation~\eqref{same comp}). Thus, the support of~$\mu_2$ is contained in the given diagonal strip.

To prove the converse, it suffices to show that the open diagonal strip is contained in the support, since the support is always closed. Any point $(x,y)$ in the open diagonal strip can be written as $(x,y) = (1,1)u+(1,-1)v$ where $|v| < \frac w2$. By the symmetry of~$\mu_2$ it suffices to consider $0\le v<\frac w2$. By Lemma~\ref{greedy lemma} we may choose $M\subset\N$, containing only finitely many pairs of consecutive integers, such that $\sum_{m\in M} 1/(\frac14+\gamma_m^2) = v$.

We are now ready to assign values to the independent random variables~$Z_m$ in Definition~\ref{eta2 def} that force $V_2=(x,y)$. Note that
\begin{align*}
V_2 &= \sum_{m=1}^\infty \biggl( 2\Re \frac{Z_m}{\frac12+i\gamma_m}, 2\Re \frac{Z_m}{-\frac12+i\gamma_m} \biggr)\\ &= \Re \sum_{m=1}^\infty \biggl( -iZ_m\frac{2\gamma_m}{\frac14+\gamma_m^2} (1,1) + Z_m\frac1{\frac14+\gamma_m^2}(1,-1) \biggr).
\end{align*}
Set $Z_m = 1$ for each $m\in M$, and set $Z_m = \pm i$ for each $m\in\N\setminus M$, so that
\[
V_2 = (1,1) \sum_{m\in\N\setminus M} \pm\frac{2\gamma_m}{\frac14+\gamma_m^2} + (1,-1) \sum_{m\in M} \frac1{\frac14+\gamma_m^2}.
\]
The second sum equals~$v$ by construction. Since~$M$ contains only finitely many pairs of consecutive integers, $\N\setminus M$ is a set of integers of density at least~$\frac12$; therefore, the first sum diverges when all terms are positive. Consequently, by the Riemann rearrangement theorem, there is a choice of signs that makes that sum equal~$u$, as required.
\end{proof}

\subsection{Joint distribution and densities} \label{joint dist subsection}

We can now establish the joint limiting logarithmic distribution of all nine normalized error terms from Definition~\ref{defn:error_bias_zeta}, from which we can deduce Theorems~\ref{theorem:log_density_zeta}--\ref{log_density part 2} and~\ref{theorem:joint_dist_zeta}.

\begin{proposition} \label{nine prop}
Assume RH and LI. Define the $\R^9$-valued function
\begin{multline*}
\mathbf E(x) = \bigl( E^\psi(x),E^\theta(x),E^\Pi(x),E^\pi(x),\\ E^{\psi_r}(x),E^{\theta_r}(x),E^{\Pi_r}(x),E^{\pi_r}(x),E^{\pi_\ell}(x) \bigr).
\end{multline*}
Then the limiting logarithmic distribution of $\mathbf E(x)$ exists and is the same as the distribution of the random variable
\[
(0,{-1},0,{-1},0,1,0,1,1) + Y_1(1,1,1,1,0,0,0,0,0) + Y_2(0,0,0,0,1,1,1,1,1)
\]
where $Y_1$ and $Y_2$ are the components of the $\R^2$-valued random variable~$V_2$ from Definition~\ref{eta2 def}.
\end{proposition}

\noindent Note that the constant vector is exactly the list of bias terms for the nine functions in~order.

\begin{proof}
Given Propositions~\ref{corollary:explicit_standards} and~\ref{corollary:explicit_reciprocals}, we may apply~\cite[Theorems~1.4 and~1.9]{akbary2013limiting} to conclude that the limiting logarithmic distribution of $\mathbf E(x)$ exists and has Fourier transform
\begin{multline*}
\exp\bigl( -i(-\xi_2-\xi_4+\xi_6+\xi_8+\xi_9) \bigr) \\ \cdot
\prod_{m=1}^\infty J_0\biggl(2\biggl|
\frac{\xi_1+\xi_2+\xi_3+\xi_4}{\frac 12+i\gamma_m}+
\frac{\xi_5+\xi_6+\xi_7+\xi_8+\xi_9}{-\frac 12+i\gamma_m}
\biggr|\biggr) .
\end{multline*}
In light of the formula for~$\hat\mu_2$ from Lemma~\ref{note:common_measures},
the proposition now follows from an application of Lemma~\ref{lemma:simplify_fourier1} with $D=9$ and $d=2$, $A_1=\{1,2,3,4\}$, $A_2=\{5,6,7,8,9\}$, $(c_1,\dots,c_9) = (0,{-1},0,{-1},0,1,0,1,1)$, and $(r_1,\dots,r_9) = (1,\dots,1)$.
\end{proof}

\begin{proof}[Proof of Theorem~\ref{theorem:joint_dist_zeta}]
All cases of the theorem follow immediately from Proposition~\ref{nine prop}.

As examples of part~(a), if~$f$ and~$g$ are the standard functions~$\psi$ and~$\theta$, then the limiting logarithmic distribution of $\bigl( E^\psi(x), E^\pi(x) \bigr)$ is given by the first two coordinates of the limiting logarithmic distribution of $\mathbf E(x)$, which yields $(0,-1) + Y_1(1,1)$ where~$Y_1$ is a random variable whose distribution is~$\mu_1$. Similarly, if~$f$ and~$g$ are the reciprocal functions~$\pi_r$ and~$\pi_\ell$, then the limiting logarithmic distribution of $\bigl( E^{\pi_r}(x), E^{\pi_\ell}(x) \bigr)$ is given by the last two coordinates of the limiting logarithmic distribution of $\mathbf E(x)$, which yields $(1,1) + Y_2(1,1)$ where~$Y_2$ is also a random variable whose distribution is~$\mu_1$. In both cases, the statement of the theorem follows from the properties of~$\mu_1$ noted in Definition~\ref{eta1 def}.

On the other hand, an example of part~(b) is~$f = \psi$ and~$g = \psi_r$, for which the limiting logarithmic distribution of $\bigl( E^\psi(x), E^{\psi_r}(x) \bigr)$ is given by the first and fifth coordinates of the limiting logarithmic distribution of $\mathbf E(x)$; this yields $(0,0) + Y_1(1,0) + Y_2(0,1) = (Y_1,Y_2)$, which by definition has distribution~$\mu_2$. In this case, the statement of the theorem follows from Lemmas~\ref{lemma:mu_one_continuity} and~\ref{lemma:mu_one_max_support}.
\end{proof}


\begin{proof}[Proof of Theorem~\ref{theorem:log_density_zeta}]
Assuming RH, note that Theorem~\ref{RH theorem} tells us that the hypothesis $\beta_f < \beta_g$ (which holds for all parts of this theorem) implies that $E^f(x) < E^g(x)$ for all sufficiently large~$x$. In particular, Corollary~\ref{nine mu1s cor} implies that
\begin{align*}
\delta(0<E^f<E^g) = \delta(0<E^f) &= \Pr(0<V_1+\beta_f) = \Pr(-\beta_f<V_1) \\
\delta(E^f<E^g<0) = \delta(E^g<0) &= \Pr(V_1+\beta_g<0) = \Pr(V_1<-\beta_g)
\end{align*}
in the notation of Definition~\ref{eta1 def}.
\begin{enumerate}[label={\rm(\alph*)}]
\item When $\beta_f=-1$ and $\beta_g=0$, we have $\Pr(-\beta_f<V_1) = \eta_1$ and $\Pr(V_1<-\beta_g) = \frac12$.
\item When $\beta_f=0$ and $\beta_g=1$, we have $\Pr(-\beta_f<V_1) = \frac12$ and $\Pr(V_1<-\beta_g) = \eta_1$.
\item When $\beta_f=-1$ and $\beta_g=1$, we have $\Pr(-\beta_f<V_1) = \Pr(V_1<-\beta_g) = \eta_1$.
\end{enumerate}
In all cases, we then have $\delta(E^f<0<E^g) = 1 - \delta(0<E^f) - \delta(E^g<0)$ by the absolute continuity of~$\mu_1$.
\end{proof}

\begin{proof}[Proof of Theorem~\ref{theorem:log_density_zeta 1.5}]
Consider the pairs of functions treated by the theorem: let
\[
(f,g) \in \bigl\{ (\pi,\theta),\, (\theta_r,\pi_r),\, (\theta_r,\pi_\ell),\, (\pi_r,\pi_\ell),\, (\psi,\Pi),\, (\psi_r,\Pi_r) \bigr\},
\]
so that $\beta_f=\beta_g$, and~$f$ and~$g$ are either both standard functions or both reciprocal functions. In these cases, Theorem~\ref{o(1) theorem}(a) implies that $f(x) = g(x) + o(1)$. Consequently for every $\ep>0$,
\[
\delta(E^f \le 0 \le E^g) < \delta(\ep < E^f \le 0) = \mu_1\bigl( (\ep,0] \bigr),
\]
and therefore $\delta(E^f \le 0 \le E^g) = 0$ by the absolute continuity of~$\mu_1$ (Lemma~\ref{lemma:mu_one_continuity}). By the same argument, $\delta(E^g \le 0 \le E^f) = 0$ as well. It follows that
\begin{align*}
\delta(0 < E^f \text{ and } 0 < E^g) = \delta(0<E^f) &= \Pr(0<V_1+\beta_f) = \Pr(-\beta_f<V_1) \\
\delta(E^f < 0 \text{ and } E^g < 0) = \delta(E^f<0) &= \Pr(V_1+\beta_f<0) = \Pr(V_1<-\beta_f)
\end{align*}
in the notation of Definition~\ref{eta1 def}.
\begin{enumerate}[label={\rm(\alph*)}]
\item When $(f,g) = (\pi,\theta)$, we have $\beta_f=\beta_g=-1$ and therefore $\Pr(-\beta_f<V_1) = \eta_1$ and $\Pr(V_1<-\beta_f)$ by definition.
\item When $(f,g) \in \bigl\{ (\psi,\Pi),\, (\psi_r,\Pi_r) \bigr\}$, we have $\beta_f=\beta_g=0$ and therefore $\Pr(-\beta_f<V_1) = \Pr(V_1<-\beta_f) = \frac12$ by symmetry.
\item Finally, when $(f,g) \in \bigl\{ (\theta_r,\pi_r),\, (\theta_r,\pi_\ell),\, (\pi_r,\pi_\ell) \bigr\}$, we have $\beta_f=\beta_g=1$ and therefore $\Pr(-\beta_f<V_1) = 1-\eta_1$ and $\Pr(V_1<-\beta_f) = \eta_1$. \qedhere
\end{enumerate}
\end{proof}

\begin{proof}[Proof of Theorem~\ref{log_density part 2}]
Since~$f$ is a standard function and~$g$ is a reciprocal function with $\beta_f=\beta_g=0$, Theorem~\ref{theorem:joint_dist_zeta} gives that the limiting logarithmic distribution of~$\bigl( E^f(x),E^g(x) \bigr)$ is~$\mu_2$. Definition~\ref{eta2 def} tells us that $\eta_2 = \delta(E^f < 0 < E^g) + \delta(E^g < 0 < E^f)$; but the two densities on the right-hand side are equal by symmetry, and thus each individual density equals $\frac{\eta_2}2$. Similarly,
\begin{multline*}
\delta(E^f < E^g < 0) + \delta(E^g < E^f < 0) + \delta(0 < E^f < E^g) + \delta(0 < E^g < E^f) \\ = 1-\eta_2
\end{multline*}
by the absolute continuity of~$\mu_2$;
but the four densities on the left-hand side are all equal by the symmetries of~$\mu_2$, and therefore each individual density equals~$\frac{1-\eta_2}4$.
\end{proof}

\noindent
We note for future use that this proof implies that the $\mu_2$-measure of the first quadrant is
\begin{multline} \label{relating mu1(Q1) and eta1}
\mu_2\bigl( \{(x,y)\in\R^2\colon x>0,\, y>0 \bigr) \\ = \delta(0 < E^f < E^g) + \delta(0 < E^g < E^f) = \frac{1-\eta_2}2.
\end{multline}

\section{A formula for densities related to~\texorpdfstring{${\eta_2}$}{eta2}} \label{density calc section}

It turns out that the calculation of the quantity~$\eta_2$ from Definition~\ref{eta2 def} is equivalent to determining the mass that~$\mu_2$ assigns to the first quadrant $Q_1 = \{ (x,y)\in\R^2\colon x>0,\,y>0\}$ (the precise relationship is given in equation~\eqref{relating mu1(Q1) and eta1} below). In this section, we follow the method of Feuerverger and the second author~\cite{doi:10.1080/10586458.2000.10504659} to write this mass in terms of a principal value of a two-dimensional integral involving the Fourier transform~$\hat\mu_2$.

\subsection{Well-behaved products of Bessel functions}

The concept of a ``well-behaved function'' appeared in~\cite{doi:10.1080/10586458.2000.10504659} in connection with principal values of multivariate integrals. In this paper, we adopt the following slightly more general definition:

\begin{definition}
\label{defn:well_behaved}
A function $f$ defined on $\R^n$ is {\em well-behaved} if it has continuous partial derivatives of all orders and there exist positive constants $\beta_1,\beta_2,\beta_3$ such that for every subset $\{j_1,\ldots,j_k\} \subseteq \{1,\ldots,n\}$ we have
\[ \bigg| \frac{\partial f}{\partial x_{j_1}\ldots\partial x_{j_k}} f(\vec x) \bigg| \le \beta_1 \exp(-\beta_2|\vec x|^{\beta_3}) \]
where $|\cdot|_2$ is the Euclidean norm on $\vec x$ in $\R^n$.
\end{definition}

\noindent
(The definition in~\cite{doi:10.1080/10586458.2000.10504659} is identical except that it fixes $\beta_3 = 1$.) The goal of this section is to show that the function~$\hat\mu_2$ given in Lemma~\ref{note:common_measures}, and indeed a large class of functions of the same shape, are well-behaved.

\begin{lemma}
\label{lemma:bessel_product_analytic}
Let $\vec r_m$ be a sequence of nonzero vectors in $\C^d$ with $\sum_{m=1}^\infty |\vec r_m|^2$ finite, and let $f \colon \R^d \to \R$ be the function
\begin{equation} \label{product of J0 with rm}
f(\vec t\,) = \prod_{m=1}^\infty J_0(|\vec r_m \cdot \vec t\,|).
\end{equation}
Then~$f$ is analytic on~$\R^d$ and, in particular, has continuous partial derivatives of all orders.
\end{lemma}

\begin{proof}
In this proof, $\vec t\, = (t_1,\ldots,t_d)$ and $\vec r_m = (r_{m,a},\ldots,r_{m,b})$. Define
\[
s(\vec t\,)^2 = \sum_{a=1}^d \sum_{b=1}^d t_at_b
\quad\text{and}\quad
g(\vec t\,) = \prod_{m=1}^\infty \exp\biggl( \frac 14|\vec r_m|^2 s(\vec t\,)^2\biggr) .
\]
Since $\sum_{m=1}^\infty |\vec r_m|^2$ converges, say to $R^2$, the infinite product defining $\vec g(t)$ can be written as
\begin{align*}
g(\vec t\,) &= \exp\biggl(\frac 14 \sum_{m=1}^\infty |\vec r_m|^2s(\vec t\,)^2\biggr)\\ &= \exp\biggl(\frac 14 R^2s(\vec t\,)^2 \biggr) = \sum_{n=0}^\infty \frac 1{4^n} \frac{R^{2n}(\sum_{a=1}^d\sum_{b=1}^d t_at_b)^n}{n!}.
\end{align*}
The power series on the right-hand side converges absolutely for all $\vec t\, \in \R^d$, and hence the inner sum can be expanded by the multinomial theorem and rearranged into the form
\[ g(t_1,\ldots,t_d) = \sum_{n_1=0}^\infty \cdots \sum_{n_d=0}^\infty G_{n_1,\ldots,n_d} t_1^{n_1} \cdots t_d^{n_d} \]
for some nonnegative constants $G_{n_1,\ldots,n_d}$. In particular, $g(\vec t\,)$ is an analytic function on~$\R^d$.

For any positive integer~$B$, define
\begin{align*}
f_B(\vec t\,) &= \prod_{m=1}^B J_0(|\vec r_m \cdot \vec t\,|) = \prod_{m=1}^B \sum_{n=0}^\infty \frac{(-1)^n}{4^n} \frac{(|\vec r_m \cdot \vec t\,|^2)^n}{n!^2} \\
g_B(\vec t\,) &= \prod_{m=1}^B \exp\biggl( \frac 14|\vec r_m|^2 s(\vec t\,)^2\biggr) = \prod_{m=1}^B \sum_{n=0}^\infty \frac 1{4^n} \frac{(|\vec r_m|^2s(\vec t\,)^2)^n}{n!}
\end{align*}
by the Maclaurin series for $J_0(x)$ in Definition~\ref{J0 definition} and the exponential function.
Both $f_B(\vec t\,)$ and $g_B(\vec t\,)$ are certainly analytic on $\R^d$, and thus there exist real constants $F_{B;n_1,\ldots,n_d}$ and $G_{B,n_1,\ldots,n_d}$ such that
\begin{align*}
f_B(\vec t\,) &= \sum_{n_1=0}^\infty \ldots \sum_{n_d=0}^\infty F_{B;n_1,\ldots,n_d} t_1^{n_1} \cdots t_d^{n_d} \\
g_B(\vec t\,) &= \sum_{n_1=0}^\infty \ldots \sum_{n_d=0}^\infty G_{B;n_1,\ldots,n_d} t_1^{n_1} \cdots t_d^{n_d} .
\end{align*}
Since $g_B(\vec t\,)$ is a product of exponential functions, it is again clear that the $G_{B;n_1,\ldots,n_d}$ are nonnegative. Moreover, since
\[
|\vec r_m \cdot \vec t\,|^2 = \sum_{a=1}^d \sum_{b=1}^d r_{m,a}\bar r_{m,b} t_at_b
\quad\text{and}\quad
|\vec r_m|^2 s(\vec t\,)^2 = \sum_{a=1}^d \sum_{b=1}^d |\vec r_m|^2 t_at_b,
\]
the inequality $|r_{m,a}\bar r_{m,b}| \le |\vec r_m|^2$ implies that $|F_{B;n_1,\ldots,n_d}| \le G_{B;n_1,\ldots,n_d}$ for all nonnegative integers $n_1,\ldots,n_d$. Since $\lim_{B\to\infty} g_B(\vec t\,) = g(\vec t\,)$, the dominated convergence theorem implies that $\lim_{B\to\infty} f_B(\vec t\,) = f(\vec t\,)$ as well, which shows that $f(\vec t\,)$ is analytic on~$\R^d$ as claimed.
\end{proof}

\begin{lemma}
\label{lemma:bessel_product_bound}
Let $\vec r_m$ be a sequence of nonzero vectors in $\C^d$ with $\sum_{m=1}^\infty |\vec r_m|^2$ finite.
For any finite set $M \subseteq \N$, define
\[ f(\vec t\,, M) = \prod_{m \in \N \setminus M} J_0(|\vec r_m \cdot \vec t\,|). \]
Suppose there exist constants $\alpha_1, \alpha_2, R>0$ and $\alpha_3>1$ such that for all $\vec t\, \in \R$ with $|\vec t\,| > R$, there exists a set $S(\vec t\,) \subseteq \N$ with $\#S(\vec t\,) \ge \alpha_1 |\vec t\,|^{\alpha_2}$ such that $\frac \pi 2|\vec r_m \cdot \vec t\,| \ge \alpha_3$ for all $m \in S(\vec t\,)$. Then there exist positive constants $\beta_1, \beta_2, \beta_3$, depending only on $\alpha_1, \alpha_2, \alpha_3, R$, and~$\#M$, such that $|f(\vec t\,,M)| \le \beta_1 \exp(-\beta_2|\vec t\,|^{-\beta_3})$ for all $\vec t\, \in \R^d$.
\end{lemma}

\begin{proof}
Since $\sum_{m \in \N \setminus M}^\infty |\vec r_m|^2 \le \sum_{m=1}^\infty |\vec r_m|^2 < \infty$, Lemma~\ref{lemma:bessel_product_analytic} implies that $f(\vec t\,, M)$ is well defined. Set $N=\#M$.
By the inequality~\eqref{Bessel bound}, when $|\vec t\,| > R$ we have
\[ |f(\vec t\,)| \le \prod_{m \in S(\vec t\,) \setminus M} \sqrt{\frac 2{\pi|\vec r_m \cdot \vec t\,|}} \le \prod_{m \in S(\vec t\,) \setminus M}\alpha_3^{-1/2} \le \alpha_3^{-\alpha_1(|\vec t\,|^{\alpha_2})/2+N/2} . \]
Hence when $|\vec t\,| > \max\{R,(\frac{2N}{\alpha_1})^{1/\alpha_2}$\}, we have $\frac N2 < \frac{\alpha_1}4 |\vec t\,|^{\alpha_2}$ and thus
\[ |f(\vec t\,, M)| \le \alpha_3^{-\alpha_1|\vec t\,|^{\alpha_2}/4} = \exp(-\beta_2|\vec t\,|^{\beta_3}) \]
where $\beta_2 = \frac{\alpha_1}4\log\alpha_3$ and $\beta_3 = \alpha_2$.
On the other hand $|f(\vec t\,, M)| \le 1$ for all $\vec t\, \in \R^d$, and so when $|\vec t\,| \le \max\{R,(\frac{2N}{\alpha_1})^{1/\alpha_2}\}$ we have $|f(\vec t\,, M)| \le \exp(-\beta_2|\vec t\,|^{\beta_3}) \exp(\beta_2\max\{R^{\beta_3}, \frac{2N}{\alpha_1}\})$. Together these inequalities show that
\[ |f(\vec t\,, M)| \le \beta_1 \exp(-\beta_2|\vec t\,|^{\beta_3}) \]
for all $\vec t\, \in \R^d$, where $\beta_1 = 1 + \exp(\beta_2\max\{R^{\beta_3}, \frac{2N}{\alpha_1}\})$.
\end{proof}

\begin{definition} \label{tilde J0 definition}
Define the entire function
\[
\tilde J_0(x) = \sum_{n=0}^\infty \frac{(-1)^n}{4^n} \frac{x^n}{n!^2}.
\]
Note that $\tilde J_0(x) = J_0(\sqrt x)$ for $x\ge0$.
\end{definition}

\begin{lemma} \label{derivatives of J0 lemma}
For every integer $n\ge1$, we have $\tilde J_0^{(n)}(x) \ll_n 1$ for all $x\ge0$.
\end{lemma}

\begin{proof}
The bound $\tilde J_0^{(n)}(x) \ll_n 1$ holds for $0\le x\le1$ simply by continuity, so we may assume that $x>1$.
The inequalities $|J_0'(x)| \le \frac x2$ and $|J_0^{(n)}(x)| \le 1$ for $n\ge2$ follow easily by differentiating under the integral sign in Definition~\ref{J0 definition} (as pointed out in~\cite[page~540]{doi:10.1080/10586458.2000.10504659}). It is easy to prove by induction that there exist constants $\alpha_{n,k}$ such that
\begin{align*}
\tilde J_0^{(n)}(x) &= \sum_{k=1}^n \alpha_{n,k} x^{-n+k/2} J_0^{(k)}(\sqrt x)\\ &= \alpha_{n,1} x^{-n+1/2} J_0'(\sqrt x) + \sum_{k=2}^n \alpha_{n,k} x^{-n+k/2} J_0^{(k)}(\sqrt x).
\end{align*}
The lemma follows from the fact that each term is bounded for $x>1$ and $n\ge1$.
\end{proof}

\begin{lemma} \label{derivatives of J0 of dot product lemma}
Let $\ell\ge1$ and $i_1,\dots,i_\ell\in\{1,\dots,d\}$ be integers. For any $R>0$,
\[
\frac{\partial^\ell J_0\bigl( |\vec r\cdot\vec t\,| \bigr)}{\partial t_{i_1} \cdots \partial t_{i_\ell}} \ll_{\ell,R} |\vec r\,|^2 \max\{1,|\vec t\,|^\ell\}
\]
uniformly for $\vec t\, = (t_1,\ldots,t_d)\in\R^d$ and $\vec r\in\C^d$ with $|\vec r\,| \le R$.
\end{lemma}

\begin{proof}
Our strategy uses the identity $J_0\bigl( |\vec r\cdot\vec t\,| \bigr) = \tilde J_0\bigl( |\vec r\cdot\vec t\,|^2 \bigr)$. (This identity reflects the observation that~$J_0$ is an even function and therefore $J_0\bigl( |\vec r\cdot\vec t\,| \bigr)$ is still smooth despite the square root implicit in $|\vec r\cdot\vec t\,|$.)
Write $\vec r = (r_1,\ldots,r_d)$. To guide our intuition, we compute explicitly that for any $i,j\in\{1,\dots,d\}$,
\begin{align*}
\frac{\partial}{\partial t_i} \tilde J_0\bigl( |\vec r\cdot\vec t\,|^2 \bigr) &= \tilde J_0'\bigl( |\vec r\cdot\vec t\,|^2 \bigr) \frac{\partial}{\partial t_i} |\vec r\cdot\vec t\,|^2 \\
&= \tilde J_0'\bigl( |\vec r\cdot\vec t\,|^2 \bigr) \biggl( 2\Re r_i \sum_{k=1}^d \Re r_kt_k + 2\Im r_i \sum_{k=1}^d \Im r_kt_k \biggr) \\
\frac{\partial^2}{\partial t_i\partial t_j} \tilde J_0\bigl( |\vec r\cdot\vec t\,|^2 \bigr) &= \tilde J_0''\bigl( |\vec r\cdot\vec t\,|^2 \bigr) \biggl( 2\Re r_i \sum_{k=1}^d \Re r_kt_k + 2\Im r_i \sum_{k=1}^d \Im r_kt_k \biggr)^2 \\
&\qquad{}+ \tilde J_0'\bigl( |\vec r\cdot\vec t\,|^2 \bigr) (2\Re r_i\Re r_j + 2\Im r_i\Im r_j).
\end{align*}
In general, one can prove by induction on~$\ell\ge1$ that
\[
\frac{\partial^\ell \tilde J_0\bigl( |\vec r\cdot\vec t\,|^2 \bigr)}{\partial t_{i_1} \cdots \partial t_{i_\ell}} = \sum_{\ell/2 \le h \le \ell} P_{2h,2h-\ell}(\vec r,\vec t\,) \tilde J_0^{(h)}\bigl( |\vec r\cdot\vec t\,|^2 \bigr)
\]
where $P_{a,b}(\vec r,\vec t\,)$ is a polynomial that is homogeneous of degree~$a$ in the variables~$\Re r_k$ and~$\Im r_k$ and also homogeneous of degree~$b$ in the variables~$t_k$. It follows from Lemma~\ref{derivatives of J0 lemma} that
\[
\frac{\partial^\ell \tilde J_0\bigl( |\vec r\cdot\vec t\,|^2 \bigr)}{\partial t_{i_1} \cdots \partial t_{i_\ell}} \ll_\ell \sum_{\ell/2 \le h \le \ell} |\vec r\,|^{2h} |\vec t\,|^{2h-\ell} \cdot 1 \ll_{\ell,R} \sum_{\ell/2 \le h \le \ell} |\vec r\,|^2 |\vec t\,|^{2h-\ell},
\]
which implies the statement of the lemma.
\end{proof}

\begin{lemma}
\label{lemma:bessel_product_well_behaved}
Let $\vec r_m$ be a sequence of nonzero vectors in $\C^d$ with $\sum_{m=1}^\infty |\vec r_m|^2$ finite, and let $f(\vec t\,)$ be as in equation~\eqref{product of J0 with rm}.
Suppose there exist constants $\alpha_1, \alpha_2>0$ and $\alpha_3,\tau>1$ such that for all $\vec t\, \in \R$ with $|\vec t\,| > \tau$, there exists a set $S(\vec t\,) \subseteq \N$ with $\#S(\vec t\,) \ge \alpha_1 |\vec t\,|^{\alpha_2}$ such that $\frac \pi 2|\vec r_m \cdot \vec t\,| \ge \alpha_3$ for all $m \in S(\vec t\,)$. Then $f(\vec t\,)$ is well-behaved.
\end{lemma}

\begin{proof}
By Lemma~\ref{lemma:bessel_product_analytic}, $f(\vec t\,)$ is analytic on $\R^d$ with continuous partial derivatives of all orders. It therefore suffices to establish the bound in Definition~\ref{defn:well_behaved} on partial derivatives of~$f(\vec t\,)$; indeed, we will establish such a bound even when multiple derivatives with respect to the same variable are present. We note that the convergence of $\sum_{m=1}^\infty |\vec r_m|^2$ implies the existence of a real number $R>0$ such that $|\vec r_m| \le R$ for all $m\in\N$.

By repeated application of the product rule,
\[
\frac{\partial^n f(\vec t\,)}{\partial t_{j_1} \cdots \partial t_{j_n}} = \sum_{M = (m_1,\dots,m_n) \in \N^n} \prod_{m\in\N\setminus M} J_0\bigl( |\vec r_m\cdot\vec t\,| \bigr) \prod_{m\in M} D_{M,m} J_0\bigl( |\vec r_m\cdot\vec t\,| \bigr),
\]
where $D_{M,m}$ is the composition of the $\ell(m)$ operators $\frac{\partial}{\partial t_{j_i}}$ over all $1\le i\le n$ such that $m_i=m$. (We are abusing notation slightly: $m\in M = (m_1,\dots,m_n)$ means $m\in\{m_1,\dots,m_n\}$.) By Lemma~\ref{lemma:bessel_product_bound},
there exist positive constants $\beta_1, \beta_2, \beta_3$ (depending only on $\alpha_1, \alpha_2, \alpha_3, R$, and~$\#M \le n$) such that the first product is at most $\beta_1 \exp(-\beta_2|\vec t\,|^{-\beta_3})$ in absolute value for all $\vec t\, \in \R^d$. Therefore by Lemma~\ref{derivatives of J0 of dot product lemma}, applied~$\#M$ times with the sum of the relevant values of~$\ell(m)$ equaling~$n$,
\begin{align*}
\frac{\partial^n f(\vec t\,)}{\partial t_{j_1} \cdots \partial t_{j_n}} &\ll_{n,R} \sum_{M = (m_1,\dots,m_n) \in \N^n} \exp(-\beta_2|\vec t\,|^{-\beta_3}) \prod_{m\in M} |\vec r\,|^2  \max\{1,|\vec t\,|^{\ell(m)}\} \\
&= \exp(-\beta_2|\vec t\,|^{-\beta_3}) \max\{1,|\vec t\,|^n\} \sum_{M = (m_1,\dots,m_n) \in \N^n} \prod_{m\in M} |\vec r\,|^2 \\
&\le \exp(-\beta_2|\vec t\,|^{-\beta_3}) \max\{1,|\vec t\,|^n\} \sum_{L = 1}^n \genfrac{\{}{\}}{0pt}{}{n}{L} \biggl( \prod_{m=1}^\infty \bigl( 1 + |\vec r\,|^2 \bigr) \biggr)^L,
\end{align*}
where the Stirling number $\genfrac{\{}{\}}{0pt}{}{n}{k}$ is the number of partitions of $\{1,\dots,n\}$ into~$L$ subsets. Since $\sum_{m=1}^\infty |\vec r_m|^2 < \infty$, each of these infinite products converges, and therefore
\begin{align*}
\frac{\partial^n f(\vec t\,)}{\partial t_{j_1} \cdots \partial t_{j_n}} &\ll_{n,\{r_m\}}  \exp(-\beta_2|\vec t\,|^{-\beta_3}) \max\{1,|\vec t\,|^n\} \\ & \ll_{n,\beta_2,\beta_3,\{r_m\}}  \exp(-\beta_2|\vec t\,|^{-\beta_3/2}),
\end{align*}
which completes the proof of the lemma.
\end{proof}

\begin{proposition}
\label{prop:mu_one_well_behaved}
The function $\hat\mu_2$ is well-behaved.
\end{proposition}

\begin{proof}
By Lemma~\ref{note:common_measures}, $\hat\mu_2(\vec t\,) = \prod_{m=1}^\infty J_0(|\vec r_m \cdot \vec t\,|)$ with
\[ \vec r_m = \biggl(\frac{2\gamma_m+i}{\frac14+\gamma_m^2}, \frac{2\gamma_m-i}{\frac 14+\gamma_m^2} \biggr) , \]
where $\gamma_1,\gamma_2,\ldots$ is the ascending sequence of the positive ordinates of the zeros of the Riemann zeta function. We use the fact that $N(T)$, the number of indices~$m$ such that $\gamma_m \le T$, is $\asymp T\log T$. In particular, the convergence $\sum_{m=1}^\infty |\vec r_m|^2$ follows by partial summation from
\[
|\vec r_m|^2 = 2\frac{4\gamma_m^2+1}{(\frac 14+\gamma_m^2)^2} \ll \frac1{\gamma_m^2}.
\]
Since $\gamma_1 \ge 1$, we see that for any $\vec t\,\in\R^d$,
\begin{align*}
|\vec r_m \cdot \vec t\,| &= \frac{\sqrt{4\gamma_m^2(t_1+t_2)^2+(t_1-t_2)^2}}{\frac 14+\gamma_m^2}\\ &\ge \frac{\sqrt{(t_1+t_2)^2+(t_1-t_2)^2}}{\frac 14+\gamma_m^2} \\&= \frac{\sqrt2|\vec t\,|}{\frac 14+\gamma_m^2} \ge \frac{|\vec t\,|}{\gamma_m^2} .
\end{align*}
In particular, if we define $S(\vec t\,) = \{ m \in \N \colon \gamma_m \le (\pi|\vec t\,|)^{1/2} \}$ for $\vec t\, \ne \vec 0$, then $\frac\pi 2|\vec r_m \cdot \vec t\,| \ge \frac12$ for all $m \in S(\vec t\,)$. The number of elements in $S(\vec t\,)$ is $N((\pi|\vec t\,|)^{1/2}) \gg |\vec t\,|^{1/2} \log |\vec t\,| \gg |\vec t\,|^{1/2}$ for $|\vec t\,|\ge2$. Therefore the proposition follows from Lemma~\ref{lemma:bessel_product_well_behaved} applied with $\alpha_2 = \alpha_3 = \frac12$ and $\tau=2$.
\end{proof}

\subsection{Integral formula for the densities}

We are interested in the mass that~$\mu_2$ from Definition~\ref{eta2 def} assigns to the first quadrant. The following formula identifies this quantity in terms of the principal value of a two-dimensional integral (see~\cite[equation~(2-30)]{doi:10.1080/10586458.2000.10504659} for a precise definition).

\begin{proposition}
\label{prop:mu_one_first_quadrant}
If $Q_1 \subseteq \R^2$ is the first quadrant, then
\[ \mu_2(Q_1) = \frac 14 - \frac 1{4\pi^2} \PV \iint_{\R^2} \frac{\hat\mu_2(u,v)}{uv} \, du \, dv . \]
\end{proposition}

\begin{proof}
By the Fourier inversion formula and the monotone
convergence theorem,
\[ \mu_2(Q_1) = \lim_{c\to 0^+} \frac 1{4\pi^2} \int_0^\infty \int_0^\infty e^{-c(x+y)} \iint_{\R^2} e^{i(ux+vy)} \hat\mu_2(u,v) \, du \, dv \, dx \, dy . \]
Hence, by Fubini's Theorem,
\begin{align*}
\mu_2(Q_1)
& = \lim_{c\to 0^+} \frac 1{4\pi^2} \iint_{\R^2} \hat\mu_2(u,v) \int_0^\infty \int_0^\infty e^{(uxi-cx)+(vyi-cy)} \, dx \, dy \, du \, dv \\
& = \lim_{c\to 0^+} \frac 1{4\pi^2} \iint_{\R^2} \frac{\hat\mu_2(u,v)}{(c-iu)(c-iv)} \, du \, dv \\
& = \lim_{c\to 0^+} \frac 1{4\pi^2} \iint_{\R^2} \frac{\hat\mu_2(u,v)(c+iu)(c+iv)}{(u^2+c^2)(v^2+c^2)} \, du \, dv \\
& = \lim_{c\to 0^+} \frac 1{4\pi^2} \iint_{\R^2} \frac{\hat\mu_2(u,v)(c^2-uv+ic(u+v))}{(u^2+c^2)(v^2+c^2)} \, du \, dv . \\
\end{align*}
Therefore
\begin{align*}
\mu_2(Q_1) &=
\lim_{c\to 0^+} \frac{c^2}{4\pi^2} \iint_{\R^2} \frac{\hat\mu_2(u,v)}{(u^2+c^2)(v^2+c^2)} \, du \, dv \\
&\qquad{}+
\lim_{c\to 0^+} \frac {ic}{4\pi^2} \iint_{\R^2} \frac{\hat\mu_2(u,v)(u+v)}{(u^2+c^2)(v^2+c^2)} \, du \, dv \\
&\qquad{}-
\lim_{c\to 0^+} \frac 1{4\pi^2} \iint_{\R^2} \frac{\hat\mu_2(u,v)uv}{(u^2+c^2)(v^2+c^2)} \, du \, dv .
\end{align*}
The second integral vanishes since $\mu_2(u,v) = \mu_2(-u,-v)$.
Hence, by the change of variables $\alpha = \frac uc$, $\beta = \frac vc$ in the first integral,
\begin{multline*}
\mu_2(Q_1) =
\lim_{c\to 0^+} \frac 1{4\pi^2} \iint_{\R^2} \frac{\hat\mu_2(\alpha c,\beta c)}{(1+\alpha^2)(1+\beta^2)} \, d\alpha \, d\beta \\ -
\lim_{c\to 0^+} \frac 1{4\pi^2} \iint_{\R^2} \frac{\hat\mu_2(u,v)uv}{(u^2+c^2)(v^2+c^2)} \, du \, dv .
\end{multline*}
Since $|\mu_2(\alpha c,\beta c)| \le 1 = |\mu_2(0,0)|$, it follows by the dominated convergence theorem that
\begin{align*}
\mu_2(Q_1) &=
\frac 1{4\pi^2} \iint_{\R^2} \frac{d\alpha \, d\beta}{(1+\alpha^2)(1+\beta^2)} \\ &\qquad{} -
\lim_{c\to 0^+} \frac 1{4\pi^2} \iint_{\R^2} \frac{\hat\mu_2(u,v)uv}{(u^2+c^2)(v^2+c^2)} \, du \, dv \\
&=
\frac 14 -
\lim_{c\to 0^+} \frac 1{4\pi^2} \iint_{\R^2} \frac{\hat\mu_2(u,v)uv}{(u^2+c^2)(v^2+c^2)} \, du \, dv .
\end{align*}
Since $\hat\mu_2$ is well-behaved by Lemma~\ref{lemma:bessel_product_well_behaved}, the evaluation 
\[ \mu_2(Q_1) =
\frac 14 -
\frac 1{4\pi^2} \, \PV \iint_{\R^2} \frac{\hat\mu_2(u,v)}{uv} \, du \, dv .
\]
now follows from~\cite[Lemma~2.6]{doi:10.1080/10586458.2000.10504659} with $n=k=2$. (That lemma was derived under the more restrictive definition of ``well-behaved'' where $\beta_3=1$, but the proof is essentially identical for any~$\beta_3>0$.)
\end{proof}

\section{Calculation and error analysis} \label{error section}

We carry out the numerical calculation of~$\eta_2$ in this section, with a rigorous error term. Both the numerical approximation itself and the error analysis used in this section closely follow the techniques used by Feuerverger and the second author in~\cite{doi:10.1080/10586458.2000.10504659}, which itself was inspired by the methods of Rubinstein and Sarnak~\cite{Rubinstein1994ChebyshevsB}.
To summarize, in Section~\ref{dti section} we rewrite the relevant principal value as a discrete sum over a lattice in~$\R^2$. That sum is then truncated in Section~\ref{ttros section} to a bounded portion of the lattice, and then the summand involving~$\hat\mu_2$ is approximated in Section~\ref{ttip section} using only the zeros of $\zeta(s)$ up to a fixed height. Each of these steps introduces errors, and the errors are explicitly bounded in the next three sections.

\subsection{Discretizing the integral} \label{dti section}

Our first approximation uses the definitions
\begin{multline} \label{S ep and Err1 def}
S(\ep) = \ep^2 \mathop{\sum\sum}_{\substack{m,n \in \Z \\ m,n \text{ odd}}} \frac{\hat\mu_2(n\ep/2,m\ep/2)}{(m\ep/2)(n\ep/2)}
\\ \text{and}\quad
\Err_1 = \PV \iint_{\R^2} \frac{\hat\mu_2(u,v)}{uv} \, du \, dv - S(\ep) .
\end{multline}
Note that the double integral appears in Proposition~\ref{prop:mu_one_first_quadrant} for the mass~$\mu_2$ assigned to the first quadrant.

Bounding~$\Err_1$ requires tail bounds for the distribution~$\mu_2$. As an initial step, we establish such a bound for the one-dimensional distribution~$\mu_1$ that equals each component of~$\mu_2$. We use the same method as Rubinstein and Sarnak~\cite{Rubinstein1994ChebyshevsB}, who actually treat~$\mu_1$ itself as well as other related distributions; the fact that we record better numerical constants for~$\mu_1$ specifically is because they were bounding several distributions simultaneously.

\begin{lemma}
\label{lemma:mu_zero_exponential_bound}
Assume RH. Let $X$ be a random variable with distribution function~$\mu_1$.
For any real number $x \ge 0.14$,
\[ \Pr(X \ge x) \le \exp(-3.75(x-0.14)^2) . \]
\end{lemma}

\begin{proof}
By Definition~\ref{eta1 def},
\[
X = \sum_{k=1}^\infty r_k \sin(2\pi U_k)
\quad\text{where}\quad
r_k = \frac 2{\sqrt{1/4+\gamma_k^2}} .
\]
Since $r_k \gg \frac1{\gamma_k} \gg \frac{\log k}k$ by the counting function for nontrivial zeros of $\zeta(s)$, the series $\sum_{k=1}^\infty r_k$ diverges; on the other hand, by Lemma~\ref{lemma:rho_squared}, $\sum_{k=1}^\infty r_k^2 = \sum_\rho 1/\rho(1-\rho) = w$ assuming RH.
It follows from~\cite[Lemma~3.1]{doi:10.1080/10586458.2000.10504659} that for $x \ge 2r_1$,
\[ \Pr(X \ge x) \le \exp\biggl(\frac{-3(x-2r_1)^2}{16w}\biggr) . \]
Since $w < 0.05$ and $r_1 < 0.07$, it follows that
\[ \Pr(X \ge x) \le \exp\biggl(\frac{-3(x-0.14)^2}{16 \cdot 0.05}\biggr) \]
for $x \ge 0.14$, as claimed.
\end{proof}

We will bootstrap this one-dimensional tail bound into a bound for the two-dimensional distribution~$\mu_2$, using the following notation.

\begin{definition}
For random variables $X,Y$ with joint distribution~$\mu_2$, define $\bar P(u,v) = \Pr(X > u,\, Y > v)$, and set $\bar P_1(u) = \bar P(u,-\infty)$ and $\bar P_2(v) = \bar P(-\infty, v)$. Also define
\begin{align*}
P^\star(u,v) &= \bar P(2\pi u, 2\pi v)-\bar P_1(2\pi u)\bar P_2(2\pi v) \\
Q(u,v) &= P^\star(u,v)+P^\star(-u,v)+P^\star(u,-v)+P^\star(-u,-v).
\end{align*}
\end{definition}

\begin{lemma} \label{Q formula lemma}
For all $u,v\in\R$,
\begin{align*}
&Q(u,v) \\ &= 2\Pr(X > 2\pi u, Y > 2\pi v) + 2\Pr(X > -2\pi u, Y > 2\pi v) - 2\Pr(Y > 2\pi v) \\
&= 2\Pr(X > 2\pi u, Y > 2\pi v) + 2\Pr(X > 2\pi u, Y > -2\pi v) - 2\Pr(X > 2\pi u).
\end{align*}
\end{lemma}

\begin{proof}
By definition,
\begin{align*}
Q(u,v)
& = \Pr(X > 2\pi u, Y > 2\pi v) - \Pr(X > 2\pi u)\Pr(Y > 2\pi v) \\
& \qquad{} + \Pr(X > -2\pi u, Y > 2\pi v) - \Pr(X > -2\pi u)\Pr(Y > 2\pi v) \\
& \qquad{} + \Pr(X > 2\pi u, Y > -2\pi v) - \Pr(X > 2\pi u)\Pr(Y > -2\pi v) \\
& \qquad{} + \Pr(X > -2\pi u, Y > -2\pi v) - \Pr(X > -2\pi u)\Pr(Y > -2\pi v) .
\end{align*}
By the symmetry and absolute continuity of~$\mu_1$,
\begin{align*}
\Pr(X > -2\pi u) & = 1-\Pr(X < -2\pi u) = 1-\Pr(X > 2\pi u) \\
\Pr(Y > -2\pi v) & = 1-\Pr(Y < -2\pi v) = 1-\Pr(Y > 2\pi v) .
\end{align*}
Thus
\begin{align*}
Q(u,v)
& = \Pr(X > 2\pi u, Y > 2\pi v) + \Pr(X > -2\pi u, Y > 2\pi v) \\
& \qquad{} + \Pr(X > 2\pi u, Y > -2\pi v) + \Pr(X > -2\pi u, Y > -2\pi v) - 1 .
\end{align*}
Then by Lemmas~\ref{lemma:mu_one_continuity} and Definition~\ref{eta2 def},
\begin{align*}
\Pr(&X > -2\pi u, Y > 2\pi v) \\ & = 1-\Pr(X < -2\pi u)-\Pr(Y < 2\pi v)+\Pr(X < -2\pi u, Y < 2\pi v) \\
& = 1-\Pr(X > 2\pi u)-\Pr(Y < 2\pi v)+\Pr(X > 2\pi u, Y > -2\pi v) \\
\Pr(&X > 2\pi u, Y > -2\pi v) \\ & = 1-\Pr(X < 2\pi u)-\Pr(Y < -2\pi v)+\Pr(X < 2\pi u, Y < -2\pi v) \\
& = 1-\Pr(X < 2\pi u)-\Pr(Y > 2\pi v)+\Pr(X > -2\pi u, Y > 2\pi v) \\
\Pr(&X > -2\pi u, Y > -2\pi v) \\ & = 1-\Pr(X < -2\pi u)-\Pr(Y < -2\pi v)+\Pr(X < -2\pi u, Y < 2\pi v) \\
& = 1-\Pr(X > 2\pi u)-\Pr(Y > 2\pi v)+\Pr(X > 2\pi u, Y > 2\pi v).
\end{align*}
Thus, $Q(u,v)$ can be simplified even further as
\begin{align*}
Q(u,v) & = 2\Pr(X > 2\pi u, Y > 2\pi v) + 2\Pr(X > -2\pi u, Y > 2\pi v) \\
& \qquad{} + 1 - \Pr(X < 2\pi u) - \Pr(Y > 2\pi v) \\ & \qquad{} + 1 - \Pr(X > 2\pi u) - \Pr(Y > 2\pi v) - 1 \\
& = 2\Pr(X > 2\pi u, Y > 2\pi v) + 2\Pr(X > -2\pi u, Y > 2\pi v) \\ & \qquad{} - 2\Pr(Y > 2\pi v) . 
\end{align*}
or alternatively as
\begin{align*}
Q(u,v) & = 2\Pr(X > 2\pi u, Y > 2\pi v) + 2\Pr(X > 2\pi u, Y > -2\pi v) \\
& \qquad{} + 1 - \Pr(X > 2\pi u) - \Pr(Y < 2\pi v) \\ & \qquad{} + 1 - \Pr(X > 2\pi u) - \Pr(Y > 2\pi v) - 1 \\
& = 2\Pr(X > 2\pi u, Y > 2\pi v) + 2\Pr(X > 2\pi u, Y > -2\pi v) \\ & \qquad{} - 2\Pr(X > 2\pi u) .
\qedhere
\end{align*}
\end{proof}

\begin{lemma} \label{Q bound lemma}
For all $u,v\in\R$,
\[
|Q(u,v)| \le 2\min\bigl\{ \Pr(X > 2\pi u), \Pr(Y > 2\pi v) \bigr\}.
\]
\end{lemma}

\begin{proof}
The trivial inequalities $0 \le \Pr(X > \pm2\pi u, Y > 2\pi v) \le \Pr( Y>2\pi v)$ imply that
\begin{align*}
- 2\Pr(Y > 2\pi v) &= 0 + 0 - 2\Pr(Y > 2\pi v) \\
&\le 2\Pr(X > 2\pi u, Y > 2\pi v) + 2\Pr(X > -2\pi u, Y > 2\pi v) \\ &\qquad{} - 2\Pr(Y > 2\pi v) \\
&\le 2\Pr(Y > 2\pi v) + 2\Pr(Y > 2\pi v) - 2\Pr(Y > 2\pi v) \\ & \quad{} = 2\Pr(Y > 2\pi v);
\end{align*}
similarly, the trivial inequalities $0 \le \Pr(X > 2\pi u, Y > \pm2\pi v) \le \Pr( Y>2\pi u)$ imply that
\begin{align*}
- 2\Pr(Y > 2\pi u) &= 0 + 0 - 2\Pr(Y > 2\pi u) \\
&\le 2\Pr(X > 2\pi u, Y > 2\pi v) + 2\Pr(X > 2\pi u, Y > -2\pi v) \\ &\qquad{} - 2\Pr(Y > 2\pi u) \\
&\le 2\Pr(Y > 2\pi u) + 2\Pr(Y > 2\pi u) - 2\Pr(Y > 2\pi u) \\ & \quad{} = 2\Pr(Y > 2\pi u);
\end{align*}
The claim now follows directly from Lemma~\ref{Q formula lemma}.
\end{proof}

\begin{lemma} \label{error_1_inequality}
We have
\[
|\Err_1| \le 4\pi^2 \sum_{\substack{(\kappa,\lambda) \in \Z_{\ge 0} \\ (\kappa,\lambda) \ne (0,0)}} \bigg|Q\biggl(\frac\kappa\ep, \frac\lambda\ep\biggr)\bigg| .
\]
\end{lemma}

\begin{proof}
This assertion is established by the calculations leading to~\cite[equation~(3-7)]{doi:10.1080/10586458.2000.10504659}. Those calculations involved a different two-dimensional distribution, but the calculations apply to any (sufficiently nice) two-dimensional distribution, including~$\mu_2$.
\end{proof}



\begin{lemma}
\label{lemma:error_discrete}
Assume RH. If $0 < \ep \le 13$ then
\[ |\Err_1| \le 96\pi^2\exp\biggl(-3.75\biggl(\frac{2\pi}\ep-0.14\biggr)^2\biggr). \]
\end{lemma}

\begin{proof}
By Lemmas~\ref{error_1_inequality}, \ref{Q bound lemma}, and~\ref{lemma:mu_zero_exponential_bound},
\begin{align*}
|\Err_1| &\le 4\pi^2 \sum_{\substack{(\kappa,\lambda) \in \Z_{\ge 0} \\ (\kappa,\lambda) \ne (0,0)}} \bigg|Q\biggl(\frac\kappa\ep, \frac\lambda\ep\biggr)\bigg| \\
&\le 4\pi^2 \sum_{\substack{(\kappa,\lambda) \in \Z_{\ge 0} \\ (\kappa,\lambda) \ne (0,0)}} 2\min\bigg\{\Pr\Bigl(X > 2\pi \frac\kappa\ep\Bigr),\Pr\Bigl(Y > 2\pi \frac\lambda\ep\Bigr)\bigg\} \\
&\le 8\pi^2 \sum_{\substack{(\kappa,\lambda) \in \Z_{\ge 0} \\ (\kappa,\lambda) \ne (0,0)}} \min\bigg\{ \exp\biggl(-3.75\Bigl(\frac{2\pi\kappa}\ep-0.14\Bigr)^2\biggr), \\ & \qquad\qquad\qquad\qquad\qquad\qquad \exp\biggl(-3.75\Bigl(\frac{2\pi\lambda}\ep-0.14\Bigr)^2\biggr) \bigg\}.
\end{align*}
The inequalities $0<\ep\le 13$ imply $\frac{2\pi}\ep - 0.14 > 0.14$, which means that $(\frac{2\pi k}\ep - 0.14)^2$ is an increasing function of the nonnegative integer variable~$k$. Consequently,
\begin{align*}
|\Err_1|
& \le 8\pi^2 \sum_{\substack{(\kappa,\lambda) \in \Z_{\ge 0} \\ (\kappa,\lambda) \ne (0,0)}} \exp\biggl(-3.75\Bigl(\frac{2\pi}\ep\max\{\kappa,\lambda\}-0.14\Bigr)^2\biggr) \\
& \le 16\pi^2 \sum_{\kappa=0}^\infty \sum_{\lambda=\max\{1,\kappa\}}^\infty \exp\biggl(-3.75\Bigl(\frac{2\pi}\ep\lambda-0.14\Bigr)^2\biggr) .
\end{align*}
One can check that each term in the inner sum is at most half the previous term when $0<\ep\le13$, and therefore
\begin{align*}
|\Err_1|
& \le 16\pi^2 \biggl( 2\exp\biggl(-3.75\Bigl(\frac{2\pi}\ep-0.14\Bigr)^2 \biggr) \\ &\qquad{} + \sum_{\kappa=1}^\infty 2\exp\biggl(-3.75\Bigl(\frac{2\pi}\ep\kappa-0.14\Bigr)^2\biggr) \\
& \le 96\pi^2 \exp\biggl(-3.75\Bigl(\frac{2\pi}\ep-0.14\Bigr)^2\biggr) 
\end{align*}
as claimed.
\end{proof}

\subsection{Truncating the range of summation} \label{ttros section}

The next step, as in~\cite{doi:10.1080/10586458.2000.10504659}, is to compute the error caused
by truncating the infinite sum $S(\ep)$ from equation~\eqref{S ep and Err1 def}
to a finite region. In that work, they truncated their sum to run over a bounded square aligned with the coordinate axes.
For us, however,
the extremely different decay rates of~$\mu_2$ in the two diagonal directions means that the same method would necessarily include many points in which the summand is very small, slowing the numerical computations significantly.
Instead, we will truncate~$S(\ep)$ to run over the rectangular region
\begin{equation} \label{RCC def}
\mathcal R(C_x,C_y) = \bigl\{ (x,y) \colon |(x,y)\cdot(\tfrac1{\sqrt 2},\tfrac1{\sqrt 2})| \le C_x,\, |(x,y)\cdot(-\tfrac1{\sqrt 2},\tfrac1{\sqrt 2})| \le C_y \bigr\}
\end{equation}
which is rotated with respect to the $xy$-axes. More precisely, for any $C_x,C_y > 0$ we define
\begin{align}
S(\ep,C_x,C_y) &= \ep^2 \mathop{\sum\sum}_{\substack{m,n \text{ odd} \\ (n\ep/2,m\ep/2) \in \mathcal R(C_x,C_y)}} \frac{\hat\mu_2(n\ep/2,m\ep/2)}{(m\ep/2)(n\ep/2)} \notag \\
\Err_2 &= S(\ep)-S(\ep,C_x,C_y) .
\label{S ep CC and Err2 def}
\end{align}
By partitioning the square lattice into two complementary diagonal lattices, and exploiting symmetries of~$\hat\mu_2(x,y)$, we can establish more convenient expressions for~$S(\ep)$ and this new truncation.

\begin{definition} \label{ae be tilde C definition}
For any $\ep > 0$, define
\begin{align*}
\la_\ep(k,\ell) & = \tfrac\ep 2 (1,1) + \ep k(1,1) + \ep\ell(-1,1) = \ep(\tfrac12+k-\ell,\tfrac12+k+\ell) , \\
\lb_\ep(k,\ell) & = \tfrac\ep 2 (-1,1) + \ep k(1,1) + \ep\ell(-1,1) = \ep(-\tfrac12+k-\ell,\tfrac12+k+\ell) .
\end{align*}
Further, for any $\ep,C_x,C_y>0$ define $\tilde C_x = {C_x}/{\ep\sqrt 2}$ and $\tilde C_y = {C_y}/{\ep\sqrt 2}$.
\end{definition}

\begin{lemma}
\label{lemma:split_lattice}
For any $\ep > 0$,
\begin{align*}
S(\ep)
& = \frac2{\ep^2} \sum_{k=0}^\infty
\frac{\hat\mu_2(\la_\ep(k,0))}{(k+1/2)^2}
- \frac2{\ep^2} \sum_{\ell=0}^\infty
\frac{\hat\mu_2(\lb_\ep(0,\ell))}{(\ell+1/2)^2} \\
& \qquad{} + \frac4{\ep^2} \sum_{k=0}^\infty \sum_{\ell=1}^\infty
\frac{\hat\mu_2(\la_\ep(k,\ell))}{(k+1/2)^2-\ell^2}
+ \frac4{\ep^2} \sum_{k=1}^\infty \sum_{\ell=0}^\infty \frac{\hat\mu_2(\lb_\ep(k,\ell))}{k^2-(\ell+1/2)^2} .
\end{align*}
Further, for any $\ep,C_x,C_y > 0$,
\begin{align*}
S(\ep,C_x,C_y) & =
\frac2{\ep^2} \sum_{0 \le k \le \tilde C_x-\frac 12}
\frac{\hat\mu_2(\la_\ep(k,0))}{(k+1/2)^2}
- \frac2{\ep^2} \sum_{0 \le \ell \le \tilde C_y-\frac 12}
\frac{\hat\mu_2(\lb_\ep(0,\ell))}{(\ell+1/2)^2} \\
& \qquad{} + \frac4{\ep^2} \sum_{0 \le k \le \tilde C_x -\frac 12}\sum_{1 \le \ell \le \tilde C_y} \frac{\hat\mu_2(\la_\ep(k,\ell))}{(k+1/2)^2-\ell^2} \\ & \qquad + \frac4{\ep^2} \sum_{1 \le k \le \tilde C_x}\sum_{0 \le \ell \le \tilde C_y-\frac 12} \frac{\hat\mu_2(\lb_\ep(k,\ell))}{k^2-(\ell+1/2)^2} .
\end{align*}
\end{lemma}

\begin{proof}
Define the function $s(x,y) = \frac{\hat\mu_2(x,y)}{xy}$ and the lattice \[ \lL = \{ (n\ep/2,m\ep/2) \colon n,m \text{ are odd integers}\}, \] so that $S(\ep)
= \sum_{(x,y) \in \lL} s(x,y)$. From Lemma~\ref{note:common_measures} we see that $s(x,y) = s(y,x) = s(-y,-x) = s(-x,-y)$, and thus
\begin{align*}
S(\ep) & = \sum_{(x,y) \in \lL} s(x,y) \\ & = \sum_{\substack{(x,y) \in \lL \\ x = y > 0}} \bigl(s(x,y)+s(-x,-y)\bigr)
+ \sum_{\substack{(x,y) \in \lL \\ -x = y > 0}} \bigl(s(x,y)+s(y,x)\bigr) \\
& \quad{} + \sum_{\substack{(x,y) \in \lL \\ y > |x|}} \bigl(s(x,y)+s(-y,-x)+s(y,x)+s(-x,-y)\bigr) \\
& = 2\sum_{\substack{(x,x) \in \lL \\ x > 0}} s(x,x)
+ 2\sum_{\substack{(x,-x) \in \lL \\ y > 0}} s(x,-x)
+ 4\sum_{\substack{(x,y) \in \lL \\ y > |x|}}s(x,y) .
\end{align*}
It is easy to check that~$\lL$ is the disjoint union of the two sublattices
\begin{align*}
\lA &= \{ \la_\ep(k,\ell)\colon k,\ell\in\Z \} \\
&= \biggl\{ \la_\ep \biggl( \frac{m+n}4-\frac12, \frac{m-n}4 \biggr) \colon n,m \text{ are odd integers, }n\equiv m\mod 4 \biggr\} \\
&= \{ (n\ep/2,m\ep/2) \colon n,m \text{ are odd integers, }n\equiv m\mod 4 \} \\
\lB &= \{ \lb_\ep(k,\ell)\colon k,\ell\in\Z \} \\
&= \biggl\{ \lb_\ep \biggl( \frac{m+n}4, \frac{m-n}4-\frac12 \biggr) \colon n,m \text{ are odd integers, }n\not\equiv m\mod 4 \biggr\} \\
&= \{ (n\ep/2,m\ep/2)) \colon n,m \text{ are odd integers, }n\not\equiv m\mod 4 \}.
\end{align*}
Consequently, in this notation, we have
\begin{align*}
S(\ep) &= 2\sum_{\substack{(x,x) \in \lL \\ x > 0}} s(x,x) + 2\sum_{\substack{(x,-x) \in \lL \\ y > 0}} s(x,-x) + 4\sum_{\substack{(x,y) \in \lL \\ y > |x|}}s(x,y) \\
&= 2\sum_{k=0}^\infty s\bigl( \la_\ep(k,0) \bigr) - 2\sum_{\ell=0}^\infty s\bigl( \la_\ep(0,\ell) \bigr) \\ &\qquad{} + 4 \biggl( \sum_{k=0}^\infty \sum_{\ell=1}^\infty s\bigl( \la_\ep(k,\ell) \bigr) + \sum_{k=1}^\infty \sum_{\ell=0}^\infty s\bigl( \la_\ep(k,\ell) \bigr) \biggr),
\end{align*}
which is equivalent to the first assertion of the lemma.
The bounds for the sums in the second assertion follow immediately from the definitions of~$\la_\ep$, $\lb_\ep$, and~$\mathcal R(C_x,C_y)$.
\end{proof}

\begin{lemma}
\label{lemma:error_2_part_1}
For any $\ep,C_x,C_y > 0$,
\begin{align*}
|&\Err_2| \le
2\sum_{k > \tilde C_x-\frac 12}
\bigg|\frac{\hat\mu_2(\la_\ep(k,0))}{(k+1/2)^2}\bigg|
+ 2\sum_{\ell > \tilde C_y-\frac 12}
\bigg|\frac{\hat\mu_2(\lb_\ep(0,\ell))}{(\ell+1/2)^2}\bigg| \\
& \quad{} + 4\sum_{k > \tilde C_x -\frac 12}\sum_{\ell=1}^\infty \bigg|\frac{\hat\mu_2(\la_\ep(k,\ell))}{(k+1/2)^2-\ell^2}\bigg| + 4\sum_{k > \tilde C_x}\sum_{\ell=0}^\infty \bigg|\frac{\hat\mu_2(\lb_\ep(k,\ell))}{k^2-(\ell+1/2)^2}\bigg| \\
& \quad{} + 4\sum_{0 \le k \le \tilde C_x - \frac 12}\sum_{\ell > \tilde C_y} \bigg|\frac{\hat\mu_2(\la_\ep(k,\ell))}{(k+1/2)^2-\ell^2}\bigg| + 4\sum_{1 \le k \le \tilde C_x}\sum_{\ell > \tilde C_y-\frac 12} \bigg|\frac{\hat\mu_2(\lb_\ep(k,\ell))}{k^2-(\ell+1/2)^2}\bigg|
\end{align*}
in the notation of Definition~\ref{ae be tilde C definition}.
\end{lemma}

\begin{proof}
Since $\Err_2 = S(\ep)-S(\ep,C_1,C_2)$, this bound follows immediately from
Lemma~\ref{lemma:split_lattice} and the triangle inequality.
\end{proof}

\begin{definition} \label{Xi Upsilon def}
For any $J,K \in \N$, define
\[
\Xi_J = \prod_{j=1}^J \sqrt{\frac{1/4+\gamma_j^2}{\gamma_j}}
\quad\text{and}\quad
\Upsilon_K = \prod_{j=1}^K \sqrt{1/4+\gamma_j^2} .
\]
\end{definition}

The following lemma encodes a key observation that will be used repeatedly.

\begin{lemma}
\label{lemma:error_2_part_2}
For any $\ep, C_x, C_y > 0$ and $J,K \in \N$,
\begin{align*}
2\sum_{k > \tilde C_x-\frac 12}
\bigg|\frac{\hat\mu_2(\la_\ep(k,0))}{(k+1/2)^2}\bigg|
&\le 2\Xi_J2^{-J/4}\pi^{-J/2}C_x^{-J/2}\biggl(\frac{\ep\sqrt 2}{C_x}\biggr)\biggl(\frac{\ep\sqrt 2}{C_x}+\frac 2{2+J}\biggr) \\
2\sum_{\ell > \tilde C_y-\frac 12}
\bigg|\frac{\hat\mu_2(\lb_\ep(0,\ell))}{(\ell+1/2)^2}\bigg|
&\le 2\Upsilon_K2^{K/4}\pi^{-K/2}C_y^{-K/2}\biggl(\frac{\ep\sqrt 2}{C_y}\biggr)\biggl(\frac{\ep\sqrt 2}{C_y}+\frac 2{2+K}\biggr) .
\end{align*}
\end{lemma}

\begin{proof}
For any $k\ge0$, we have $\la_\ep(k,0) = \ep(k+\tfrac12,k+\tfrac12)$ by Definition~\ref{ae be tilde C definition}. Lemma~\ref{note:common_measures} then gives
\[
\hat\mu_2(\la_\ep(k,0)) = \prod_{m=1}^\infty J_0\biggl(\frac {2\gamma_m\ep(2k+1)}{1/4+\gamma_m^2}\biggr) \le \prod_{m=1}^J J_0\biggl( \frac {2\gamma_m\ep(2k+1)}{1/4+\gamma_m^2} \biggr)
\]
by inequality~\eqref{Bessel bound}, which further implies
\begin{equation} \label{Xi bound}
\bigl| \hat\mu_2(\la_\ep(k,0)) \bigr| \le \prod_{m=1}^J \sqrt{ \frac1\pi \frac{1/4+\gamma_m^2}{\gamma_m\ep(2k+1)} } = \Xi_J(\pi\ep(2k+1))^{-J/2}.
\end{equation}
Therefore
\begin{align*}
2\sum_{k > \tilde C_x-\frac 12}
\bigg|&\frac{\hat\mu_2(\la_\ep(k,0))}{(k+1/2)^2}\bigg| \\
& \le 2\Xi_J2^{-J/2}(\pi\ep)^{-J/2} \sum_{k > \tilde C_x-\frac 12} (k+1/2)^{-2-J/2} \\
& \le 2\Xi_J2^{-J/2}(\pi\ep)^{-J/2} \biggl(\biggl(\frac{\ep\sqrt 2}{C_x}\biggr)^{2+J/2} + \int_{\tilde C_x}^\infty u^{-2-J/2} \, du \biggr) \\
& \le 2\Xi_J2^{-J/2}(\pi\ep)^{-J/2} \biggl(\biggl(\frac{\ep\sqrt 2}{C_x}\biggr)^{2+J/2} + \frac 2{2+J} \biggl(\frac{\ep\sqrt 2}{C_x}\biggr)^{1+J/2}\biggr) \\
& \le 2\Xi_J2^{-J/4}(\pi)^{-J/2}C_x^{-J/2} \biggl(\biggl(\frac{\ep\sqrt 2}{C_x}\biggr)^2 + \frac 2{2+J} \biggl(\frac{\ep\sqrt 2}{C_x}\biggr)\biggr) .
\end{align*}
This establishes the first inequality, and the second follows from a similar calculation.
\end{proof}

\begin{definition} \label{tilde A tilde B def}
For any $J \in \N$, define
\begin{align*}
\lAb_J(k,\ell) & = \prod_{j=1}^J \frac{\sqrt{1/4+\gamma_j^2}}{\bigl( \ell^2+4\gamma_j^2(k+1/2)^2 \bigr)^{1/4}} \\
\lBb_J(k,\ell) & = \prod_{j=1}^J \frac{\sqrt{1/4+\gamma_j^2}}{\bigl( (\ell+1/2)^2+4\gamma_j^2k^2 \bigr)^{1/4}}.
\end{align*}
The derivation of the inequality~\eqref{Xi bound} generalizes immediately to show that
\begin{align*}
\hat\mu_2(\la_\ep(k,\ell)) & \le (\pi\ep)^{-J/2}\lAb_J(k,\ell) , \\
\hat\mu_2(\lb_\ep(k,\ell)) & \le (\pi\ep)^{-J/2}\lBb_J(k,\ell) .
\end{align*}
for any integers $k,\ell\ge0$.
\end{definition}

\begin{lemma}
\label{lemma:error_2_part_4}
For any $\ep > 0$ and $C_x \ge 5\ep\sqrt 2$ and $J \in \N$,
\begin{align*}
4 & \sum_{k > \tilde C_x - \frac 12}\sum_{1 \le \ell \le k} \bigg|\frac{\hat\mu_2(\la_\ep(k,\ell))}{(k+1/2)^2-\ell^2}\bigg| \\
&\qquad{}\le 4\Xi_J2^{-J/4}\pi^{-J/2}C_x^{-J/2}\biggl(\biggl(\frac{\ep\sqrt 2}{C_x}\biggr)\biggl(1+\frac 1{4\tilde C_x-1} + \frac 12 \log\Bigl(\frac{4C_x}{\ep\sqrt 2} - 1\Bigr)\biggr) \\
&\qquad\qquad{}+ \frac 2J\biggl(1+\frac 1J\biggr)\biggl(1+\frac 1{4\tilde C_x-1}\biggr) + \frac 1J\log\Bigl(\frac{4C_x}{\ep\sqrt 2}-1\Bigr)\biggr) \\
4 & \sum_{k > \tilde C_x}\sum_{0 \le \ell \le k-1} \bigg|\frac{\hat\mu_2(\lb_\ep(k,\ell))}{k^2-(\ell+1/2)^2}\bigg| \\
&\qquad{}\le 4\Xi_J2^{-J/4}\pi^{-J/2}C_x^{-J/2}\biggl(\biggl(\frac{\ep\sqrt 2}{C_x}\biggr)\biggl(1+\frac 1{4\tilde C_x-1} + \frac 12 \log\Bigl(\frac{4C_x}{\ep\sqrt 2} - 1\Bigr)\biggr) \\
&\qquad\qquad{}+ \frac 2J\biggl(1+\frac 1J\biggr)\biggl(1+\frac 1{4\tilde C_x-1}\biggr) + \frac 1J\log\Bigl(\frac{4C_x}{\ep\sqrt 2}-1\Bigr)\biggr).
\end{align*}
\end{lemma}

\begin{proof}
Observe for all $k \ge 0$ and $J \in \N$,
\[ \lAb_J(k,\ell) = \Xi_J(k+1/2)^{-J/2} \prod_{j=1}^J \sqrt{\frac{\gamma_j}{\sqrt{(\frac\ell{k+\frac 12})^2+4\gamma_j^2}}} \le \Xi_J2^{-J/2}(k+1/2)^{-J/2} . \]
Consequently,
\begin{align*}
4 & \sum_{k > \tilde C_x - \frac 12}\sum_{1 \le \ell \le k} \bigg|\frac{\hat\mu_2(\la_\ep(k,\ell))}{(k+1/2)^2-\ell^2}\bigg| \\
& \le 4\Xi_J2^{-J/2}(\pi\ep)^{-J/2} \sum_{k > \tilde C_x - \frac 12}\sum_{1 \le \ell \le k} \frac{(k+1/2)^{-J/2}}{(k+1/2)^2-\ell^2} \\
& \le 4\Xi_J2^{-J/2}(\pi\ep)^{-J/2} \\ &\quad{}\cdot \sum_{k > \tilde C_x - \frac 12} (k+1/2)^{-J/2}\biggl(\frac 1{k+\frac 14}+\int_1^k \frac{du}{(k+1/2)^2-u^2}\biggr) \\
& \le 4\Xi_J2^{-J/2}(\pi\ep)^{-J/2} \\ &\quad{}\cdot \sum_{k > \tilde C_x - \frac 12} (k+1/2)^{-J/2}\biggl(\frac 1{k+\frac 14}+\frac 1{2k+1}\int_1^k \frac 1{k+\frac 12+u} + \frac 1{k+\frac 12-u} \, du\biggr) \\
& \le 4\Xi_J2^{-J/2}(\pi\ep)^{-J/2} \\ &\quad{}\cdot \sum_{k > \tilde C_x - \frac 12} (k+1/2)^{-J/2}\biggl(\frac 1{k+\frac 14}+\frac 1{2k+1}\log\biggl(2(2k+1/2)\frac{k-\frac 12}{k+\frac 32}\biggr)\biggr) \\
& \le 4\Xi_J2^{-J/2}(\pi\ep)^{-J/2} \\ &\quad{}\cdot \sum_{k > \tilde C_x - \frac 12} (k+1/2)^{-J/2}\biggl(\frac 1{k+1/4}+\frac 1{2k+1}\log(4k+1)\biggr) .
\end{align*}
As $k+1/2 > \tilde C_x$ and $u/(u-1/4)$ is a decreasing function of $u$,
\begin{multline*}
\biggl(\frac 1{k+1/4}+\frac 1{2k+1}\log(4k+1)\biggr) \\
\le \frac 1{k+1/2}\biggl(\frac{\tilde C_x}{\tilde C_x-1/4}\Bigl(1+\frac 1J\Bigr)+\frac 12\log(4k+1) - \frac{k+1/2}{J(k+1/4)}\biggr)
\end{multline*}
for $C_x > \ep\sqrt 2$. Hence
\begin{align*}
& 4\sum_{k > \tilde C_x - \frac 12}\sum_{1 \le \ell \le k} \bigg|\frac{\hat\mu_2(\la_\ep(k,\ell))}{(k+1/2)^2-\ell^2}\bigg|
\le 4\Xi_J2^{-J/2}(\pi\ep)^{-J/2} \\ &\quad{}\cdot \sum_{k > \tilde C_x - \frac 12} (k+1/2)^{-1-J/2}\biggl(\frac{\tilde C_x}{\tilde C_x-\frac 14}\Bigl(1+\frac 1J\Bigr)+\frac 12\log(4k+1) - \frac{k+\frac 12}{J(k+\frac 14)}\biggr) .
\end{align*}
By observing that $\int u^{-1-J/2}\big(\frac 12\log(4u-1)-\frac u{J(u-1/4)}\big) \, du = - \frac 1J u^{-J/2} \log(4u-1)+C$, and that the integrand can be easily verified to be decreasing when $u \ge \tilde C_x \ge 5$ and $J \in \N$, we see that
\begin{align*}
4&\sum_{k > \tilde C_x - \frac 12}\sum_{1 \le \ell \le k} \bigg|\frac{\hat\mu_2(\la_\ep(k,\ell))}{(k+1/2)^2-\ell^2}\bigg| \\
& \le 4\Xi_J2^{-J/2}(\pi\ep)^{-J/2} \biggl(\biggl(\frac{\ep\sqrt 2}{C_x}\biggr)^{1+J/2}\biggl(\frac{\tilde C_x}{\tilde C_x-\frac 14}+\frac 12\log\biggl(4\tilde C_x-1\biggr)\biggr) \\
& \quad{} + \int_{\tilde C_x}^\infty u^{-1-J/2}\frac{\tilde C_x}{\tilde C_x-\frac 14}\Bigl(1+\frac 1J\Bigr)+\frac 12\log(4u-1)-\frac u{J(u-1/4)} \, du \biggr) \\
& \le 4\Xi_J2^{-J/2}(\pi\ep)^{-J/2} \biggl(\biggl(\frac{\ep\sqrt 2}{C_x}\biggr)^{1+J/2}\biggl(1+\frac 1{4\tilde C_x-1}+\frac 12\log\biggl(4\tilde C_x-1\biggr)\biggr) \\
& \quad{} + \frac 2J\biggl(1+\frac 1{4\tilde C_x-1}\biggr)\biggl(1+\frac 1J\biggr)\biggl(\frac{\ep\sqrt 2}{C_x}\biggr)^{J/2} \\
& \quad{} + \frac 1J\biggl(\frac{\ep\sqrt 2}{C_x}\biggr)^{J/2}\log\biggl(4\tilde C_x-1\biggr)\biggr) \\
& \le 4\Xi_J2^{-J/4}\pi^{-J/2}C_x^{-J/2}\biggl(\biggl(\frac{\ep\sqrt 2}{C_x}\biggr)\biggl(1+\frac 1{4\tilde C_x-1}+\frac 12\log\biggl(4\tilde C_x-1\biggr)\biggr) \\
& \quad{} + \frac 2J\biggl(1+\frac 1{4\tilde C_x-1}\biggr)\biggl(1+\frac 1J\biggr) + \frac 1J\log\biggl(4\tilde C_x-1\biggr)\biggr) .
\end{align*}
This establishes the first inequality, and the second follows from a similar calculation.
\end{proof}

\begin{lemma}
\label{lemma:error_2_part_3}
For any $\ep > 0$ and $C_y > C_x > 0$ and $K \in \N$,
\begin{align*}
4 & \sum_{0 \le k \le \tilde C_x - \frac 12}\sum_{\ell > \tilde C_y} \bigg|\frac{\hat\mu_2(\la_\ep(k,\ell))}{(k+1/2)^2-\ell^2}\bigg| \\
&\le 4\Upsilon_K2^{K/4}\pi^{-K/2}C_y^{-K/2}\biggl(\frac{\ep\sqrt 2}{C_y}\biggr)\biggl(\frac{\ep\sqrt 2}{C_y} + \frac 2{2+K}\biggr) \\ &\qquad{}\cdot \biggl(1+\frac{C_x^2}{C_y^2-C_x^2}+\tilde C_x\biggl(1+\frac{C_x}{C_y-C_x}\biggr)\biggr) \\
4 & \sum_{1 \le k \le \tilde C_x}\sum_{\ell > \tilde C_y-\frac 12} \bigg|\frac{\hat\mu_2(\lb_\ep(k,\ell))}{k^2-(\ell+1/2)^2}\bigg| \\
&\le 4\Upsilon_K2^{K/4}\pi^{-K/2}C_y^{-K/2}\biggl(\frac{\ep\sqrt 2}{C_y}\biggr)\biggl(\frac{\ep\sqrt 2}{C_y} + \frac 2{2+K}\biggr) \\ &\qquad{}\cdot \biggl(1+\frac{C_x^2}{C_y^2-C_x^2}+\tilde C_x\biggl(1+\frac{C_x}{C_y-C_x}\biggr)\biggr) .
\end{align*}
\end{lemma}

\begin{proof}
By the inequalities in Definition~\ref{tilde A tilde B def},
\[ \bigg|\frac{\hat\mu_2(\la_\ep(k,\ell))}{((k+1/2)^2+\ell)^2}\bigg| \le (\pi\ep)^{-K/2}\frac{\lAb_K(k,\ell)}{|(k+1/2)^2-\ell^2|} . \]
Observe that for $\ell > 0$,
\begin{align*}
\lAb_K(k,\ell) &= \Upsilon_K \prod_{j=1}^K (\ell^2+4\gamma_j^2(k+1/2)^2)^{-1/4} \\
&= \Upsilon_K\ell^{-K/2} \prod_{j=1}^K (1+4\gamma_j^2(k+1/2)^2\ell^{-2})^{-1/4} \le \Upsilon_K\ell^{-K/2} .
\end{align*}
Hence, for $C_y > C_x > 0$,
\begin{align*}
4&\sum_{0 \le k \le \tilde C_x-\frac 12}\sum_{\ell > \tilde C_y} \bigg|\frac{\hat\mu_2(\la_\ep(k,\ell))}{(k+1/2)^2-\ell^2}\bigg| \\
& \le 4\Upsilon_K(\pi\ep)^{-K/2} \sum_{0 \le k \le \tilde C_x-\frac 12}\sum_{\ell > \tilde C_y} \ell^{-K/2}(\ell^2-(k+1/2)^2)^{-1} \\
& \le 4\Upsilon_K(\pi\ep)^{-K/2} \sum_{\ell > \tilde C_y} \ell^{-K/2}\biggl(\frac 1{\ell^2-\frac{C_x^2}{2\ep^2}} + \int_{1/2}^{\frac{C_x}{\ep\sqrt 2}} \frac{du}{\ell^2-u^2} \biggr) \\
& \le 4\Upsilon_K(\pi\ep)^{-K/2} \sum_{\ell > \tilde C_y} \ell^{-K/2}\biggl(\frac 1{\ell^2-\frac{C_x^2}{2\ep^2}} + \frac 1{2\ell} \int_{1/2}^{\frac{C_x}{\ep\sqrt 2}} \frac 1{\ell+u} + \frac 1{\ell-u} \, du \biggr) \\
& \le 4\Upsilon_K(\pi\ep)^{-K/2} \sum_{\ell > \tilde C_y} \ell^{-K/2}\biggl(\frac 1{\ell^2-\frac{C_x^2}{2\ep^2}} + \frac 1{2\ell} \log\biggl(\frac{(\ell+\tilde C_x)(\ell-\frac 12)}{(\ell - \tilde C_x)(\ell+\frac 12)}\biggr)\biggr)
\end{align*}
By the Maclaurin series expansion
of $\log(1+x)$ as an alternating series,
\[
\log\biggl(\frac{(\ell+\tilde C_x)(\ell-\frac 12)}{(\ell - \tilde C_x)(\ell+\frac 12)}\biggr)
\le \log\Bigl(\frac{\ell+\tilde C_x}{\ell-\tilde C_x}\Bigr)
= \log\Bigl(1+\frac{2\tilde C_x}{\ell-\tilde C_x}\Bigr) \le \frac{2\tilde C_x}{\ell-\tilde C_x} ,
\]
and hence
\begin{multline*}
4\sum_{0 \le k \le \tilde C_x-\frac 12}\sum_{\ell > \tilde C_y} \bigg|\frac{\hat\mu_2(\la_\ep(k,\ell))}{(k+1/2)^2-\ell^2}\bigg| \\
\le 4\Upsilon_K(\pi\ep)^{-K/2} \sum_{\ell > \tilde C_y} \ell^{-K/2}\biggl(\frac 1{\ell^2-\frac{C_x^2}{2\ep^2}} + \frac{\tilde C_x}{\ell^2-(\tilde C_x)\ell}\biggr)\biggr).
\end{multline*}
For $\ell > \tilde C_y$ we observe that
\[ \frac 1{\ell^2}\biggl(\frac{\ell^2}{\ell^2-\frac{C_x^2}{2\ep^2}}\biggr) < \frac 1{\ell^2}\biggl(\frac{C_y^2}{C_y^2-C_x^2}\biggr) = \frac 1{\ell^2}\biggl(1+\frac{C_x^2}{C_y^2-C_x^2}\biggr) ,  \]
and likewise
\[ \frac 1{\ell^2}\biggl(\frac{\tilde C_x\ell^2}{\ell^2-(\tilde C_x)\ell}\biggr) < \frac 1{\ell^2}\tilde C_x\biggl(\frac{C_y^2}{C_y^2-C_xC_y}\biggr) = \frac 1{\ell^2}\tilde C_x\biggl(1+\frac{C_x}{C_y-C_x}\biggr) . \]
Therefore
\begin{align*}
4 & \sum_{0 \le k \le \tilde C_x-\frac 12}\sum_{\ell > \tilde C_y} \bigg|\frac{\hat\mu_2(\la_\ep(k,\ell))}{(k+1/2)^2-\ell^2}\bigg| \\
& \le 4\Upsilon_K(\pi\ep)^{-K/2} \sum_{\ell > \tilde C_y} \ell^{-2-K/2}\biggl(1+\frac {C_x^2}{C_y^2-C_x^2} + \tilde C_x\biggl(1+\frac{C_x}{C_y-C_x}\biggr)\biggr) \\
& \le 4\Upsilon_K(\pi\ep)^{-K/2} \biggl(\biggl(\frac{\ep\sqrt 2}{C_y}\biggr)^{2+K/2} + \int_{\tilde C_y}^\infty u^{-2-K/2} \, du\biggr) \\ & \qquad{} \cdot \biggl(1+\frac {C_x^2}{C_y^2-C_x^2} + \tilde C_x\biggl(1+\frac{C_x}{C_y-C_x}\biggr)\biggr) \\
& \le 4\Upsilon_K(\pi\ep)^{-K/2} \biggl(\biggl(\frac{\ep\sqrt 2}{C_y}\biggr)^{2+K/2} + \frac 2{2+K}\biggl(\frac{\ep\sqrt 2}{C_y}\biggr)^{1+K/2}\biggr) \\ & \qquad{} \cdot \biggl(1+\frac {C_x^2}{C_y^2-C_x^2} + \tilde C_x\biggl(1+\frac{C_x}{C_y-C_x}\biggr)\biggr) \\
& \le 4\Upsilon_K2^{K/4}\pi^{-K/2}C_y^{-K/2} \biggl(\biggl(\frac{\ep\sqrt 2}{C_y}\biggr)^2 + \frac 2{2+K}\biggl(\frac{\ep\sqrt 2}{C_y}\biggr)\biggr) \\ & \qquad{} \cdot \biggl(1+\frac {C_x^2}{C_y^2-C_x^2} + \tilde C_x\biggl(1+\frac{C_x}{C_y-C_x}\biggr)\biggr) .
\end{align*}
This completes the proof of the first inequality, and the second is established analogously.
\end{proof}

\begin{lemma}
\label{lemma:error_2_part_5}
For any real numbers $\ep, C_x \ge 5\ep\sqrt 2$ and any $J \in \N$,
\begin{align*}
4 & \sum_{k > \tilde C_x - \frac 12}\sum_{\ell \ge k+1} \bigg|\frac{\hat\mu_2(\la_\ep(k,\ell))}{(k+1/2)^2-\ell^2}\bigg| \\
&\qquad{}\le 4\Xi_J2^{-J/4}\pi^{-J/2}C_x^{-J/2}\biggl\{\biggl(\frac{\ep\sqrt 2}{C_x}\biggr)\biggl(1+\frac 12\log\Bigl(\frac{4C_x}{\ep\sqrt 2}+1\Bigr)+\frac 1{4J} \frac{\ep\sqrt 2}{C_x}\biggr) \\
&\qquad\qquad{}+ \frac 2J\biggl(1+\frac 1J\biggr) + \frac 1J\log\biggl(\frac{4C_x}{\ep\sqrt 2}+1\biggr)\biggr\} \\
4 & \sum_{k > \tilde C_x}\sum_{\ell \ge k} \bigg|\frac{\hat\mu_2(\lb_\ep(k,\ell))}{k^2-(\ell+1/2)^2}\bigg| \\
&\qquad{}\le 4\Xi_J2^{-J/4}\pi^{-J/2}C_x^{-J/2}\biggl\{\biggl(\frac{\ep\sqrt 2}{C_x}\biggr)\biggl(1+\frac 12\log\Bigl(\frac{4C_x}{\ep\sqrt 2}+1\Bigr)+\frac 1{4J} \frac{\ep\sqrt 2}{C_x}\biggr) \\
&\qquad\qquad{}+ \frac 2J\biggl(1+\frac 1J\biggr) + \frac 1J\log\biggl(\frac{4C_x}{\ep\sqrt 2}+1\biggr)\biggr\} .
\end{align*}
\end{lemma}

\begin{proof}
For all $k \ge 0$ and $J \in \N$,
\[ \lAb_J(k,\ell) = \Xi_J(k+1/2)^{-J/2} \prod_{j=1}^J \sqrt{\frac{\gamma_j}{\sqrt{(\frac\ell{k+\frac 12})^2+4\gamma_j^2}}} \le \Xi_J2^{-J/2}(k+1/2)^{-J/2} . \]
Consequently,
\begin{align*}
&4\sum_{k > \tilde C_x - \frac 12}\sum_{\ell \ge k+1} \bigg|\frac{\hat\mu_2(\la_\ep(k,\ell))}{(k+1/2)^2-\ell^2}\bigg| \\
& \le 4\Xi_J2^{-J/2}(\pi\ep)^{-J/2} \sum_{k > \tilde C_x - \frac 12}\sum_{\ell \ge k+1} \frac{(k+1/2)^{-J/2}}{\ell^2-(k+1/2)^2} \\
& \le 4\Xi_J2^{-J/2}(\pi\ep)^{-J/2} \\ &\quad{}\cdot \sum_{k > \tilde C_x - \frac 12} (k+1/2)^{-J/2}\biggl(\frac 1{(k+1)^2-(k+1/2)^2}+\int_{k+1}^\infty \frac{du}{u^2-(k+1/2)^2}\biggr) \\
& \le 4\Xi_J2^{-J/2}(\pi\ep)^{-J/2} \\ &\quad{}\cdot \sum_{k > \tilde C_x - \frac 12} (k+1/2)^{-J/2}\biggl(\frac 1{k+\frac 34}+\frac 1{2k+1}\int_{k+1}^\infty \frac 1{u-k-\frac 12} - \frac 1{u+k+\frac 12} \, du\biggr) \\
& \le 4\Xi_J2^{-J/2}(\pi\ep)^{-J/2} \sum_{k > \tilde C_x - \frac 12} (k+1/2)^{-J/2}\biggl(\frac 1{k+\frac 34}+\frac{\log(4k+3)}{2k+1}\biggr) \\
& \le 4\Xi_J2^{-J/2}(\pi\ep)^{-J/2} \\ &\quad{}\cdot \sum_{k > \tilde C_x - \frac 12} (k+1/2)^{-1-J/2}\biggl(\frac{k+\frac 12}{k+\frac 34}(1+1/J-1/J)+\frac 12\log(4k+3)\biggr) \\
& \le 4\Xi_J2^{-J/2}(\pi\ep)^{-J/2} \\ &\quad{}\cdot \sum_{k > \tilde C_x - \frac 12} (k+1/2)^{-1-J/2}\biggl(1+\frac 1J+\frac 12\log(4k+3)-\frac{k+\frac 12}{J(k+\frac 34)}\biggr) .
\end{align*}
Observe that $\int u^{-1-J/2} \bigl( \frac 12 \log(4u+1) - \frac u{J(u+1/4)} \bigr) \, du = - \frac 1J u^{-J/2}\log(4u+1) + C$, and it is easy to verify that the integrand is decreasing for any $u \ge \tilde C_x \ge 5$ and $J \in \N$, so
\begin{align*}
4 & \sum_{k > \tilde C_x - \frac 12}\sum_{\ell \ge k+1} \bigg|\frac{\hat\mu_2(\la_\ep(k,\ell))}{(k+1/2)^2-\ell^2}\bigg| \\
& \le 4\Xi_J2^{-J/2}(\pi\ep)^{-J/2} \\ &\quad{}\cdot \biggl(\biggl(\frac{\ep\sqrt 2}{C_x}\biggr)^{1+J/2}\biggl(1+\frac 1J+\frac 12\log\biggl(\frac{4C_x}{\ep\sqrt 2}+1\biggr)-\frac{\tilde C_x}{J(\tilde C_x+\frac 14)}\biggr) \\
& \qquad{} + \int_{\tilde C_x}^\infty u^{-1-J/2}\biggl(1+\frac 1J+\frac 12 \log(4u+1)-\frac u{J(u+\frac 14)}\biggr) \, du\biggr) \\
& \le 4\Xi_J2^{-J/2}(\pi\ep)^{-J/2} \\ &\quad{}\cdot \biggl(\biggl(\frac{\ep\sqrt 2}{C_x}\biggr)^{1+J/2}\biggl(1+\frac 1J+\frac 12\log\biggl(\frac{4C_x}{\ep\sqrt 2}+1\biggr)-\frac{\tilde C_x}{J(\tilde C_x+\frac 14)}\biggr) \\
& \qquad{} + \frac 2J\biggl(1+\frac 1J\biggr)\biggl(\frac{\ep\sqrt 2}{C_x}\biggr)^{J/2}+\frac 1J\biggl(\frac{\ep\sqrt 2}{C_x}\biggr)^{J/2}\log\biggl(\frac{4Cx}{\ep\sqrt 2}+1\biggr)\biggr) \\
& \le 4\Xi_J2^{-J/4}\pi^{-J/2}C_x^{-J/2}\biggl(\biggl(\frac{\ep\sqrt 2}{C_x}\biggr)\biggl(1+\frac 12\log\biggl(\frac{4C_x}{\ep\sqrt 2}+1\biggr)+\frac 1{J(\frac{4C_x}{\ep\sqrt 2}+1)}\biggr) \\
& \quad{} + \frac 2J\biggl(1+\frac 1J\biggr)+\frac 1J\log\biggl(\frac{4Cx}{\ep\sqrt 2}+1\biggr)\biggr) .
\end{align*}
The first inequality of the lemma follows upon noting that
$1/{J(\frac{4C_x}{\ep\sqrt 2}+1)} \le {\ep\sqrt 2}/{4JC_x}$.
The second inequality is established analogously.
\end{proof}

\begin{lemma}
\label{lemma:error_2_conclusion}
For any $\ep > 0$ and $C_y > C_x \ge 5\ep\sqrt 2$ and $J,K \in \N$,
\begin{align*}
|\Err_2| 
& \le \Xi_J2^{-J/4}\pi^{-J/2}C_x^{-J/2}\Bigg[2\biggl(\frac{\ep\sqrt 2}{C_x}\biggr)\biggl(\frac{\ep\sqrt 2}{C_x}+\frac 2{2+J}\biggr) \\ 
& \quad{}\quad{} + 8\biggl\{\biggl(\frac{\ep\sqrt 2}{C_x}\biggr)\biggl(1+\frac 1{4\tilde C_x-1} + \frac 12 \log\Bigl(\frac{4C_x}{\ep\sqrt 2} - 1\Bigr)\biggr) \\ 
& \quad{}\quad{}\quad{} + \frac 2J\biggl(1+\frac 1J\biggr)\biggl(1+\frac 1{4\tilde C_x-1}\biggr) + \frac 1J\log\Bigl(\frac{4C_x}{\ep\sqrt 2}-1\Bigr)\biggr\} \\ 
& \quad{}\quad{} + 8\biggl\{\biggl(\frac{\ep\sqrt 2}{C_x}\biggr)\biggl(1+\frac 12\log\Bigl(\frac{4C_x}{\ep\sqrt 2}+1\Bigr)+\frac 1{4J} \frac{\ep\sqrt 2}{C_x}\biggr) \\ 
& \quad{}\quad{}\quad{}\quad{} + \frac 2J\biggl(1+\frac 1J\biggr) + \frac 1J\log\biggl(\frac{4C_x}{\ep\sqrt 2}+1\biggr)\biggr\}\Bigg] \\ 
& \quad{} + \Upsilon_K2^{K/4}\pi^{-K/2}C_y^{-K/2}\Bigg[2\biggl(\frac{\ep\sqrt 2}{C_y}\biggr)\biggl(\frac{\ep\sqrt 2}{C_y}+\frac 2{2+K}\biggr) \\ 
& \quad{}\quad{}\quad{} + 8\biggl(\frac{\ep\sqrt 2}{C_y}\biggr)\biggl(\frac{\ep\sqrt 2}{C_y} + \frac 2{2+K}\biggr) \\
& \quad{}\quad{}\quad{}\quad{} \cdot \biggl(1+\frac{C_x^2}{C_y^2-C_x^2}+\tilde C_x\biggl(1+\frac{C_x}{C_y-C_x}\biggr)\biggr)\Bigg] . \\ 
\end{align*}
\end{lemma}

\begin{proof}
Combining Lemmas~\ref{lemma:error_2_part_1}, \ref{lemma:error_2_part_2}, and~\ref{lemma:error_2_part_4}--\ref{lemma:error_2_part_5} yields
\begin{align*}
|& \Err_2| 
\le 2\Xi_J2^{-J/4}\pi^{-J/2}C_x^{-J/2}\biggl(\frac{\ep\sqrt 2}{C_x}\biggr)\biggl(\frac{\ep\sqrt 2}{C_x}+\frac 2{2+J}\biggr) \\ 
& \quad{} + 2\Upsilon_K2^{K/4}\pi^{-K/2}C_y^{-K/2}\biggl(\frac{\ep\sqrt 2}{C_y}\biggr)\biggl(\frac{\ep\sqrt 2}{C_y}+\frac 2{2+K}\biggr) \\ 
& \quad{} + 8\Upsilon_K2^{K/4}\pi^{-K/2}C_y^{-K/2}\biggl(\frac{\ep\sqrt 2}{C_y}\biggr)\biggl(\frac{\ep\sqrt 2}{C_y} + \frac 2{2+K}\biggr) \\ & \qquad{} \cdot \biggl(1+\frac{C_x^2}{C_y^2-C_x^2}+\tilde C_x\biggl(1+\frac{C_x}{C_y-C_x}\biggr)\biggr) \\ 
& \quad{} + 8\Xi_J2^{-J/4}\pi^{-J/2}C_x^{-J/2}\biggl(\biggl(\frac{\ep\sqrt 2}{C_x}\biggr)\biggl(1+\frac 1{4\tilde C_x-1} + \frac 12 \log\Bigl(\frac{4C_x}{\ep\sqrt 2} - 1\Bigr)\biggr) \\ 
& \quad{}\quad{}\quad{} + \frac 2J\biggl(1+\frac 1J\biggr)\biggl(1+\frac 1{4\tilde C_x-1}\biggr) + \frac 1J\log\Bigl(\frac{4C_x}{\ep\sqrt 2}-1\Bigr)\biggr) \\ 
& \quad{} + 8\Xi_J2^{-J/4}\pi^{-J/2}C_x^{-J/2}\biggl(\biggl(\frac{\ep\sqrt 2}{C_x}\biggr)\biggl(1+\frac 12\log\Bigl(\frac{4C_x}{\ep\sqrt 2}+1\Bigr)+\frac 1{4J} \frac{\ep\sqrt 2}{C_x}\biggr) \\ 
& \quad{}\quad{}\quad{} + \frac 2J\biggl(1+\frac 1J\biggr) + \frac 1J\log\biggl(\frac{4C_x}{\ep\sqrt 2}+1\biggr)\biggr) 
\end{align*}
which implies the statement of the lemma.
\end{proof}

In the interest of making the optimization step tractable,
we will only minimize the exponential terms $\Xi_J2^{-J/4}\pi^{-J/2}C_x^{-J/2}$ and $\Upsilon_K2^{K/4}\pi^{-K/2}C_y^{-K/2}$ for a given~$C_x$ or~$C_y$; in light of Definition~\ref{Xi Upsilon def}, this means choosing the maximal~$J$ and~$K$ such that
\begin{equation} \label{J and K choice}
\sqrt{\frac{1/4+\gamma_J^2}{\gamma_J}} \le 2^{1/4}\pi^{1/2}C_x^{1/2}
\quad\text{and}\quad
\sqrt{1/4+\gamma_K^2} \le 2^{-1/4}\pi^{1/2}C_y^{1/2}.
\end{equation}

\subsection{Truncating the infinite product} \label{ttip section}

For the final step, the function $\hat\mu_2$ will
be approximated by a finite product of Bessel functions and a correction factor.

In the calculations of the error term, certain constants must be known to arbitrary precision. The following lemmas provide the formulas used to calculate them.

\begin{lemma}
\label{lemma:rho_power}
For any integer $n \ge 2$,
\[
\sum_\rho \frac 1{\rho^n} = 1 - \frac 1{(n-1)!} \frac{d^n}{ds^n} \log \zeta(s) \bigg|_{s=0} - \biggl({-}\frac 12\biggr)^n \zeta(n) .
\]
In particular, $\sum_\rho \frac 1{\rho^2} = B_2$ and $\sum_\rho \frac 1{\rho^4} = B_4$,
where we define
\begin{align*}
B_2 &= 1-\frac{\pi^2}{24}+2\zeta''(0)+\log^2(2\pi) \\
B_4 &= 1 - \frac{\pi^4}{1440} - \frac 16\Bigl(-2\zeta^{(4)}(0)+8\zeta^{(3)}(0)\log(2\pi) \\ &\quad{} -12\zeta^{(2)}(0)(\zeta^{(2)}(0)+2\log^2(2\pi))-6\log^4(2\pi)\Bigr).
\end{align*}
\end{lemma}

\begin{proof}
By~\cite[Corollary 10.14]{montgomery_vaughan_2006}, there exists a constant~$B$ such that
\[ \frac d{ds} \log\zeta(s) = B + \frac 12\log\pi - \frac 1{s-1} - \frac 12 \frac{\Gamma'}{\Gamma}(s/2+1) + \sum_\rho \biggl(\frac 1{s-\rho} + \frac 1\rho\biggr) . \]
Consequently, for $n \ge 2$,
\[
\frac {d^n}{ds^n} \log\zeta(s) = \frac{(-1)^n(n-1)!}{(s-1)^n} - \frac 12 \frac{d^{n-1}}{ds^{n-1}} \frac{\Gamma'}{\Gamma}(s/2+1) + \sum_\rho \biggl(\frac{(-1)^{n-1}(n-1)!}{(s-\rho)^n}\biggr) .
\]
Furthermore, by~\cite[Equation C.11, pg. 522]{montgomery_vaughan_2006},
\[ \frac{\Gamma'}\Gamma(a) = - \frac 1a - C_0 - \sum_{m=1}^\infty \biggl(\frac 1{a+m}-\frac 1m\biggr), \]
and therefore
\[ \frac d{da} \frac{\Gamma'}{\Gamma}(a) = \frac 1{a^2} + \sum_{m=1}^\infty \frac 1{(a+m)^2} = \zeta(2,a) , \]
where $\zeta(z,a)$ is the Hurwitz zeta-function. Since $\frac d{da} \zeta(z,a) = -z\zeta(z+1,a)$, we see that for $n \ge 2$,
\[ \frac{d^{n-1}}{da^{n-1}} \frac{\Gamma'}\Gamma(a) = (-1)^n(n-1)!\zeta(n,a) . \]
Consequently, by the chain rule,
\begin{multline*} \frac{d^n}{ds^n} \log\zeta(s) = \frac{(-1)^n(n-1)!}{(s-1)^n} - \biggl(-\frac 12\biggr)^n(n-1)!\zeta(n,s/2+1) \\ + \sum_\rho\frac{(-1)^{n-1}(n-1)!}{(s-\rho)^n} . \end{multline*}
The first assertion of the lemma follows by evaluating at $s = 0$ (noting that $\zeta(n,1)=\zeta(n)$) and rearranging.

The specific formulas for $\sum_\rho \frac1{\rho^2}$ and $\sum_\rho \frac1{\rho^4}$ can 
be obtained by repeated differentiation of $\frac{\zeta'}{\zeta}(s)$ and the known values
\[ \zeta(0) = -\frac 12 \text{ and } \zeta'(0) = -\frac 12\log(2\pi) \text{ \cite[equations~(10.11) and~(10.14)]{montgomery_vaughan_2006}}. \qedhere \]
\end{proof}

\begin{lemma}
\label{lemma:square_sum_convergence}
Assuming RH,
\begin{align*}
\sum_{\gamma > 0} \frac{\gamma^2}{(1/4+\gamma^2)^4} & = \frac{B_1}2 + \frac{B_2}2  - \frac {B_4}4 \approx 1.43512\ldots \cdot 10^{-7} , \\
\sum_{\gamma > 0} \frac{\gamma^2}{(1/4+\gamma^2)^2} & = \frac{B_1-B_2}4 \approx 0.0230864\ldots , \\
\sum_{\gamma > 0} \frac 1{(1/4+\gamma^2)^2} & = B_1+B_2 \approx 3.71006\ldots \cdot 10^{-5} ,
\end{align*}
where $B_1 = C_0+2-\log(4\pi)$ and $B_2$ and $B_4$ are defined in Lemma~\ref{lemma:rho_power}.
\end{lemma}


\begin{proof}
For any complex number $\rho\ne0,1$,
\begin{align*}
\frac{-(\rho-1/2)^2}{\rho^4(1-\rho)^4} & = \frac{1}{\rho(1-\rho)}+\frac{1/2}{\rho^2}+\frac{1/2}{(1-\rho)^2}+\frac{-1/4}{\rho^4}+\frac{-1/4}{(1-\rho)^4} , \\
\frac{-(\rho-1/2)^2}{\rho^2(1-\rho)^2} & = \frac{1/2}{\rho(1-\rho)} + \frac {-1/4}{\rho^2} + \frac {-1/4}{(1-\rho)^2} , \\
\frac 1{\rho^2(1-\rho)^2} & = \frac 2{\rho(1-\rho)} + \frac 1{\rho^2} + \frac 1{(1-\rho)^2} .
\end{align*}
The functional equation for $\zeta(s)$ ensures that its nontrivial zeros come in pairs $\rho$, $1-\rho$.  and $\frac 1\rho + \frac 1{1-\rho} = \frac 1{\rho(1-\rho)}$, and therefore
\begin{align*}
\sum_\rho \frac{-(\rho-1/2)^2}{\rho^4(1-\rho^4)} & = \sum_\rho \frac 1{\rho(1-\rho)} + \sum_\rho \frac 1{\rho^2} - \frac 12 \sum_\rho \frac 1{\rho^4} = B_1 + B_2 - \frac{B_4}2, \\
\sum_\rho \frac{-(\rho-1/2)^2}{\rho^2(1-\rho^2)} & = \frac 12 \sum_\rho \frac 1{\rho(1-\rho)} - \frac 12 \sum_\rho \frac 1{\rho^2} = \frac{B_1}2 - \frac{B_2}2 , \\
\sum_\rho \frac 1{\rho^2(1-\rho^2)} & = 2\sum_\rho \frac 1{\rho(1-\rho)} + 2\sum_\rho \frac 1{\rho^2} = 2B_1 + 2B_2
\end{align*}
unconditionally
by~\cite[Corollary 10.14]{montgomery_vaughan_2006} and Lemma~\ref{lemma:rho_power}. If we now assume RH, we see that
\begin{align*}
\sum_{\gamma > 0} \frac{\gamma^2}{(1/4+\gamma^2)^4} & = \frac 12 \sum_\rho \frac{-(\rho-1/2)^2}{\rho^4(1-\rho)^4} = \frac{B_1}2 + \frac{B_2}2  - \frac {B_4}4 \\
\sum_{\gamma > 0} \frac{\gamma^2}{(1/4+\gamma^2)^2} & = \frac 12 \sum_{\rho} \frac{-(\rho-1/2)^2}{\rho^2(1-\rho)^2} = \frac{B_1-B_2}4 \\
\sum_{\gamma > 0} \frac 1{(1/4+\gamma^2)^2} & = \frac 12 \sum_\rho \frac 1{\rho^2(1-\rho)^2} = B_1+B_2
\end{align*}
as claimed.
\end{proof}

\begin{definition}
\label{defn:truncated_product}
For $T>0$, define the truncated products
\[ F_T(u,v) = \prod_{0 < \gamma \le T} J_0\biggl(\frac{\sqrt{\gamma^2u^2+v^2}}{1/4+\gamma^2}\biggr) \text{ and } \tilde F_T(u,v) = \prod_{\gamma > T} J_0\biggl(\frac{\sqrt{\gamma^2u^2+v^2}}{1/4+\gamma^2}\biggr), \]
and define
\[ G_T(u,v) = F_T(u,v)\bigl(1-P_Tu^2-Q_Tv^2+(P_TQ_T+R_T)u^2v^2\bigr) \]
where
\begin{multline*}
P_T = \frac 14 \sum_{\gamma > T} \frac{\gamma^2}{(1/4+\gamma^2)^2} , \quad Q_T = \frac 14\sum_{\gamma > T} \frac 1{(1/4+\gamma^2)^2} , \\ \text{and } R_T = \frac 1{32} \sum_{\gamma > T} \frac{\gamma^2}{(1/4+\gamma^2)^4}
\end{multline*}
are the tails of the infinite sums from Lemma~\ref{lemma:square_sum_convergence}.
We then define
\begin{align}
S(\ep,C_x,C_y,T) & =
2\sum_{0 \le k \le \tilde C_x-\frac 12}
\frac{G_T(4\ep(k+1/2),0)}{(k+1/2)^2} \notag \\
& \qquad{} - 2\sum_{0 \le \ell \le \tilde C_y-\frac 12}
\frac{G_T(0,2\ep(\ell+1/2))}{(\ell+1/2)^2} \\
& \qquad{} + 4\sum_{0 \le k \le \tilde C_x -\frac 12}\sum_{1 \le \ell \le \tilde C_y} \frac{G_T(4\ep(k+1/2),2\ep\ell)}{(k+1/2)^2-\ell^2} \label{SeCCT def} \notag \\
& \qquad{} + 4\sum_{1 \le k \le \tilde C_x}\sum_{0 \le \ell \le \tilde C_y-\frac 12} \frac{G_T(4\ep k, 2\ep(\ell+1/2))}{k^2-(\ell+1/2)^2} \notag
\end{align}
and the third error term
\begin{equation} \label{err3 def}
\Err_3 = S(\ep,C_x,C_y)-S(\ep,C_x,C_y,T) .
\end{equation}
For use in our upper bounds, we also define
\begin{multline*}
\Delta_T(u,v) = \frac 1{2(1-P_Tu^2)(1-Q_Tv^2)} \\ \cdot \biggl(\Bigl(1-\frac{P_Tu^2}2\Bigr)Q_T^2v^4+P_T^2u^4\Bigl(1-\frac{Q_Tv^2}2\Bigr)-\frac{P_T^2u^4Q_T^2v^4}2 \biggr) .
\end{multline*}
\end{definition}

Recall the quantities $\la_\ep(k,\ell)$, $\lb_\ep(k,\ell)$, $\tilde C_x$, and~$\tilde C_y$ from Definition~\ref{ae be tilde C definition}.

\begin{lemma}
\label{lemma:truncated_product_error}
Let $T>0$, and choose any $0 < C_x < 1/\sqrt{8P_T}$ and $0 < C_y < 1/\sqrt{2Q_T}$.
\begin{enumerate}[label={\rm(\alph*)}]
\item For any $0 < k \le \tilde C_x -\frac 12$ and $0 \le \ell \le \tilde C_y$,
\begin{multline*}
\bigl|\hat\mu_2(\la_\ep(k,\ell)) - G_T(4\ep(k+1),2\ep\ell)\bigr| \\ \le \bigl|F_T(4\ep(k+1/2),2\ep\ell)\bigr|\Delta_T\bigl(4\ep(k+1/2),2\ep\ell\bigr) .
\end{multline*}
\item For any $0 \le k \le \tilde C_x$ and $0 < \ell \le \tilde C_y - \frac12$,
\begin{multline*}
\bigl|\hat\mu_2(\lb_\ep(k,\ell)) - G_T(4\ep k, 2\ep(\ell+1/2))\bigr| \\ \le \bigl|F_T(4\ep k, 2\ep(\ell+1/2))\bigr|\Delta_T\bigl(4\ep k, 2\ep(\ell+1/2)\bigr) .
\end{multline*}
\end{enumerate}
\end{lemma}

\begin{proof}
By the Maclaurin expansion in Definition~\ref{J0 definition},
\[ \tilde F_T(u,v) = \prod_{\gamma > T} \sum_{m=0}^\infty \frac{(-1)^m}{(m!)^2} \biggl(\frac 14 \cdot \frac{\gamma^2u^2+v^2}{(1/4+\gamma^2)^2}\biggr)^m = \sum_{m=0}^\infty \sum_{n=0}^\infty b_{m,n} u^{2m} v^{2n} \]
where
\[ b_{0,0} = 1, b_{1,0} = -P_T, b_{0,1} = -Q_T, \text{ and } b_{1,1} = P_TQ_T+R_T. \]
Furthermore, since
\begin{align*}
\prod_{\gamma > T} \sum_{m=0}^\infty \frac 1{m!} \biggl(\frac 14 \cdot \frac{\gamma^2u^2+v^2}{(1/4+\gamma^2)^2}\biggr)^m &= \prod_{\gamma > T} \exp\biggl(\frac 14 \cdot \frac{\gamma^2u^2+v^2}{(1/4+\gamma^2)^2}\biggr) \\
&= \exp(P_Tu^2)\exp(Q_Tv^2) \\
& = \sum_{m=0}^\infty \sum_{n=0}^\infty \frac{P_T^mQ_T^n}{m!n!} u^{2m}v^{2n},
\end{align*}
we see that $|b_{m,n}| \le {P_T^mQ_T^n}/{m!n!}$. It follows that whenever $P_Tu^2 < 1$ and $Q_Tv^2 < 1$,
\begin{multline*}
\bigg|\sum_{m=0}^1 \sum_{n=2}^\infty b_{m,n}u^{2m}v^{2n}\bigg|
\le \sum_{m=0}^1 \sum_{n=2}^\infty \frac{P_T^mQ_T^n}{m!n!} u^{2m}v^{2n} \\ 
\le \sum_{m=0}^1 \frac{P_T^m}{m!} u^{2m}\frac{Q_T^2v^4}{(2!)(1-Q_Tv^2)}
= \biggl(1+\frac{P_Tu^2}2\biggr) \cdot \frac{Q_T^2v^4}{2(1-Q_Tv^2)} \\ \le \frac{1-\frac{P_Tu^2}2-\frac{P_T^2u^4}2}{1-P_Tu^2} \cdot \frac{Q_T^2v^4}{2(1-Q_Tv^2)} .
\end{multline*}
Similarly, under the same assumptions,
\begin{align*}
\bigg|\sum_{m=2}^\infty \sum_{n=0}^1 b_{m,n}u^{2m}v^{2n}\bigg|
& \le \sum_{m=2}^\infty \sum_{n=0}^1 \frac{P_T^mQ_T^n}{m!n!} u^{2m}v^{2n} \\
& \le \sum_{m=2}^\infty \frac{P_T^m}{m!} u^{2m}\biggl(1+\frac{Q_Tu^2}2\biggr) \\
& \le \frac{P_T^2u^4}{2(1-P_Tu^2)} \cdot \frac{1-\frac{Q_Tv^2}2-\frac{Q_T^2v^4}2}{1-Q_Tv^2} \\
\bigg|\sum_{m=2}^\infty \sum_{n=2}^\infty b_{m,n}u^{2m}v^{2n}\bigg|
& \le \sum_{m=2}^\infty \sum_{n=2}^\infty \frac{P_T^mQ_T^n}{m!n!} u^{2m}v^{2n} \\
& \le \sum_{m=2}^\infty \frac{P_T^m}{m!} u^{2m}\frac{Q_T^2v^4}{2(1-Q_Tv^2)} \\
& \le \frac{P_T^2u^4}{2(1-P_Tu^2)} \cdot \frac{Q_T^2v^4}{2(1-Q_Tv^2)} .
\end{align*}
These inequalities combine to show that
\[
\bigl| \tilde F_T(u,v) - (1-P_Tu^2-Q_Tv^2+(P_TQ_T+R_T)u^2v^2) \bigr| \le \Delta_T(u,v)
\]
whenever $P_T < u^{-2}$ and $Q_T < v^{-2}$. Finally, since
\begin{align*}
\hat\mu_2(\la_\ep(k,\ell)) &= F_T(4\ep(k+1/2),2\ep\ell)\tilde F_T(4\ep(k+1/2),2\ep\ell) \\
\hat\mu_2(\lb_\ep(k,\ell)) &= F_T(4\ep k,2(\ell+1/2))\tilde F_T(4\ep k,2\ep(\ell+1/2)) ,
\end{align*}
the lemma follows by setting $(u,v) = (4\ep(k+1/2),2\ep\ell)$ and $(u,v) = (4\ep k,2\ep(\ell+1/2))$.
\end{proof}

\begin{lemma}
\label{lemma:error_3}
If $P_T < \frac 1{8C_x^2}$ and $Q_T < \frac 1{2C_y^2}$, then
\begin{align*}
|\Err_3| & \le
2\sum_{0 \le k \le \tilde C_x-\frac 12}
\frac{|F_T(4\ep(k+1/2),0)|}{(k+1/2)^2}\Delta_T(4\ep(k+1/2),0) \\
& \qquad{} + 2\sum_{0 \le \ell \le \tilde C_y-\frac 12}
\frac{|F_T(0,2\ep(\ell+1/2))|}{(\ell+1/2)^2}\Delta_T(0,2\ep(\ell+1/2)) \\
& \qquad{} + 4\sum_{0 \le k \le \tilde C_x -\frac 12}\sum_{1 \le \ell \le \tilde C_y} \frac{|F_T(4\ep(k+1/2),2\ep\ell)|}{|(k+1/2)^2-\ell^2|}\Delta_T(4\ep(k+1/2),2\ep\ell) \\
& \qquad{} + 4\sum_{1 \le k \le \tilde C_x}\sum_{0 \le \ell \le \tilde C_y-\frac 12} \frac{|F_T(4\ep k, 2\ep(\ell+1/2))|}{|k^2-(\ell+1/2)^2|}\Delta_T(4\ep k, 2\ep(\ell+1/2)) ,
\end{align*}
with $\Delta_T(u,v)$ as in Definition~\ref{defn:truncated_product}.
\end{lemma}

\begin{proof}
By Lemma~\ref{lemma:split_lattice},
\begin{align*}
\Err_3 & = S(\ep,C_x,C_y)-S(\ep,C_x,C_y,T) \\
& = 2\sum_{0 \le k \le \tilde C_x-\frac 12} \frac{\hat\mu_2(\la_\ep(k,0))-G_T(4\ep(k+1/2),0)}{(k+1/2)^2} \\
& \qquad{} - 2\sum_{0 \le \ell \le \tilde C_y-\frac 12}
\frac{\hat\mu_2(\lb_\ep(0,\ell))-G_T(0,2\ep(\ell+1/2))}{(\ell+1/2)^2} \\
& \qquad{} + 4\sum_{0 \le k \le \tilde C_x -\frac 12}\sum_{1 \le \ell \le \tilde C_y} \frac{\hat\mu_2(\la_\ep(k,\ell)-G_T(4\ep(k+1/2),2\ep\ell)}{(k+1/2)^2-\ell^2} \\
& \qquad{} + 4\sum_{1 \le k \le \tilde C_x}\sum_{0 \le \ell \le \tilde C_y-\frac 12} \frac{\hat\mu_2(\lb_\ep(k,\ell))-G_T(4\ep k, 2\ep(\ell+1/2))}{k^2-(\ell+1/2)^2} .
\end{align*}
Since $P_T < \frac 1{8C_x^2}$ and $Q_T < \frac 1{2C_y^2}$ the rest follows by the triangle inequality and Lemma~\ref{lemma:truncated_product_error}.
\end{proof}

\subsection{Result of calculation and error term estimation}

Assuming RH, Proposition~\ref{prop:mu_one_first_quadrant} and equations~\eqref{S ep and Err1 def}, \eqref{S ep CC and Err2 def}, and~\eqref{err3 def} tell us that
\[
\mu_2(Q_1) = \frac 14 - \frac 1{4\pi^2} \bigl( S(\ep,C_x,C_y,T) + \Err_1 + \Err_2 + \Err_3 \bigr).
\]
We coded the definition of $S(\ep,C_x,C_y,T)$ and of these three error terms in~C using the PARI library~\cite{PARI2}. Smaller values of~$\ep$, and larger values of~$C_x$, $C_y$, and~$T$, result in smaller error terms at the expense of longer calculation time. We ended up choosing
\begin{equation} \label{param choices}
\ep = 1.5,\, C_x = 30,\, C_y = 2500,\, \text{and } T = 7500.
\end{equation}
With these parameters we found that $S(\ep,C_x,C_y,T) \approx -9.60218$, and therefore
\[
\bigl| \mu_2(Q_1) - 0.493226 \bigr| \le \frac1{4\pi^2} \bigl( |\Err_1| + |\Err_2| + |\Err_3| \bigr).
\]
Under the same parameters, Lemma~\ref{lemma:error_discrete} gives $|\Err_1| \le 1.9026\cdot10^{-24}$; also, the inequalities~\eqref{J and K choice} result in the choices $J = 44$ and $K = 18$, in which case Lemma~\ref{lemma:error_2_conclusion} gives $|\Err_2| \le 8.1525\cdot10^{-6}$. Finally, our calculations show that $P_{7500} \approx 4.2891\cdot10^{-5}$, $Q_{7500} \approx 2.3321\cdot10^{-13}$, and $R_{7500} \approx 3.0537\cdot10^{-22}$, and therefore Lemma~\ref{lemma:error_3} gives $|\Err_3| \le 3.5443\cdot10^{-4}$. These error bounds combine to yield
\[
\bigl| \mu_2(Q_1) - 0.493226 \bigr| \le 0.00001.
\]
Since $\mu_2(Q_1) = \frac{1-\eta_2}2$ by equation~\eqref{relating mu1(Q1) and eta1}, we deduce that
\[
\bigl| \eta_2 - 0.013548 \bigr| \le 0.00002.
\]

\section*{Acknowledgments}

The authors thank the anonymous referees for their helpful comments.
The second author was supported in part by a Natural Sciences and Research Council of Canada Discovery Grant.

\bibliography{comparative}
\bibliographystyle{abbrv}

\addresseshere

\newpage  
\appendix

\section{Graphs illustrating the main theorems} \label{graph appendix}

We exhibit graphs showing every possible case of the densities from Theorems~\ref{theorem:log_density_zeta}--\ref{log_density part 2} and of the 
limiting logarithmic distributions from Theorem~\ref{theorem:joint_dist_zeta}. The graph is determined solely by whether each function is standard or reciprocal and by the bias factors of the two functions.

\begin{figure}[ht]
\includegraphics[width=3in]{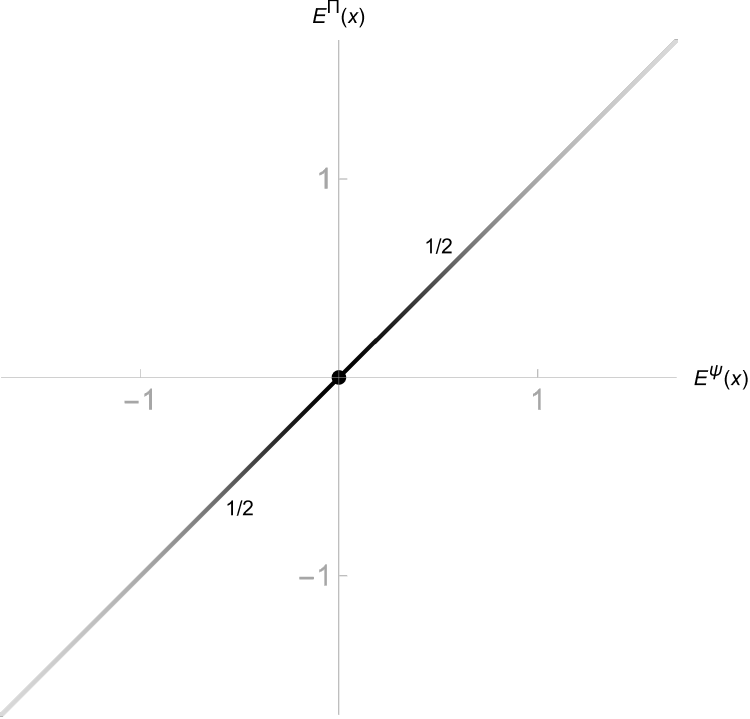}
\qquad
\includegraphics[width=3in]{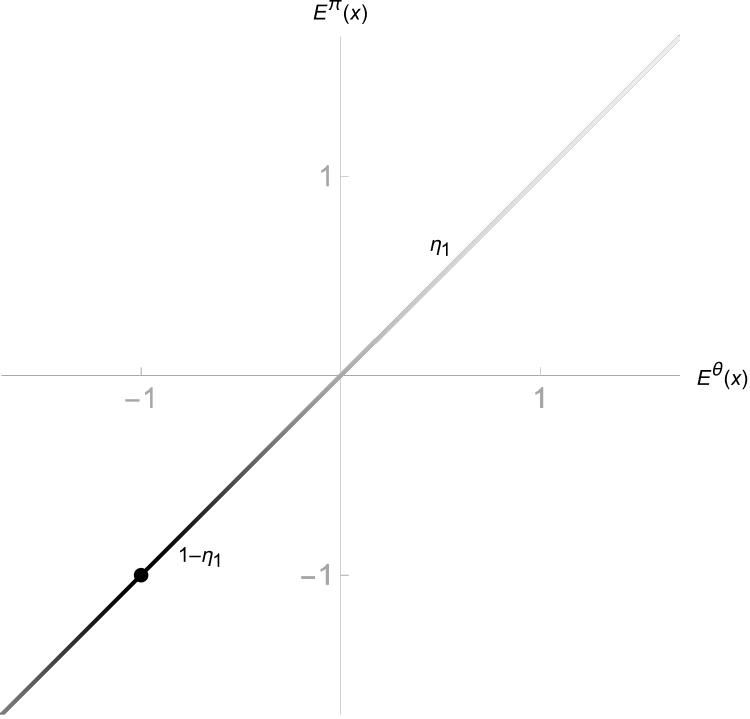}
\caption{Two pairs of standard functions with equal bias factors within each pair. Each joint distribution is singular with support equal to the line $y=x$ (perfect correlation).}
\end{figure}

\begin{figure}[ht]
\includegraphics[width=3in]{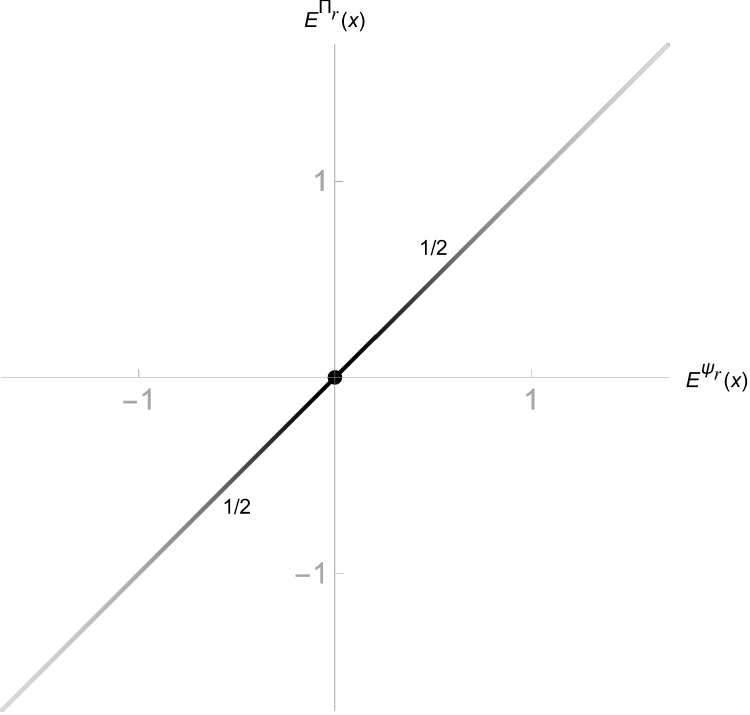}
\qquad
\includegraphics[width=3in]{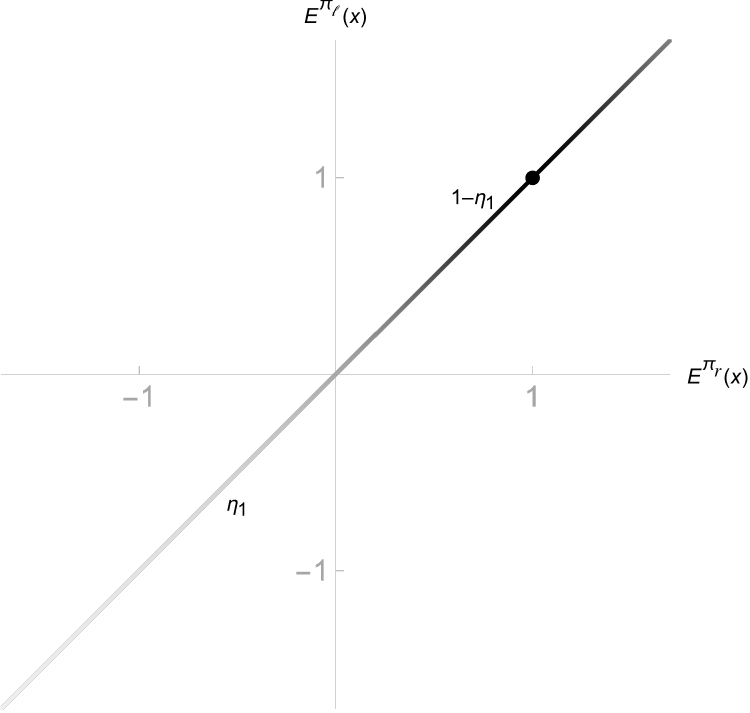}
\caption{Two pairs of reciprocal functions with equal bias factors within each pair. Each joint distribution is singular with support equal to the line $y=x$ (perfect correlation).}
\end{figure}

\begin{figure}[ht]
\includegraphics[width=3in]{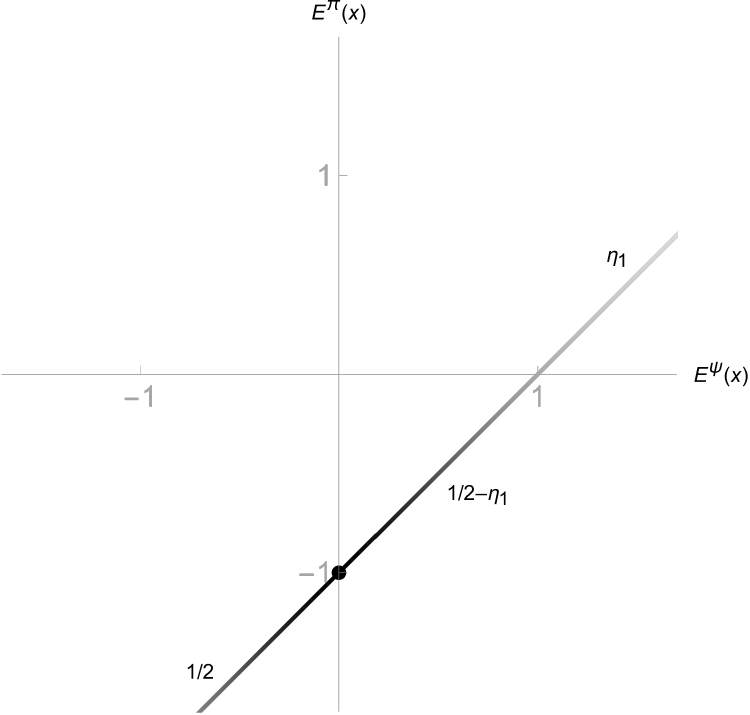}
\qquad
\includegraphics[width=3in]{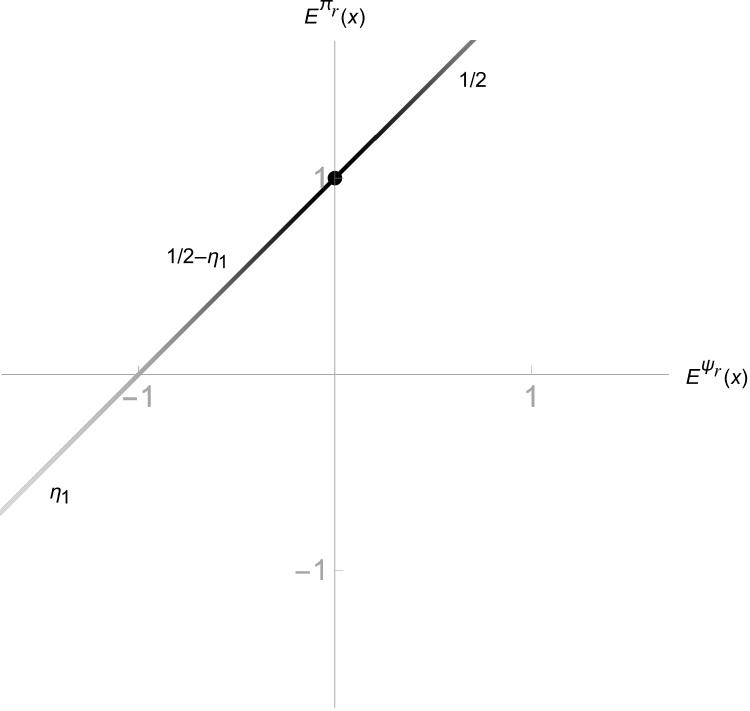}
\caption{A pair of standard functions and a pair of reciprocal functions, with unequal bias factors within each pair. Each joint distribution is singular with support equal to a line $y=x\pm1$ (perfect correlation up to the bias factors).}
\end{figure}

\begin{figure}[ht]
\includegraphics[width=3in]{graphs/psi_pi_r.pdf}
\qquad
\includegraphics[width=3in]{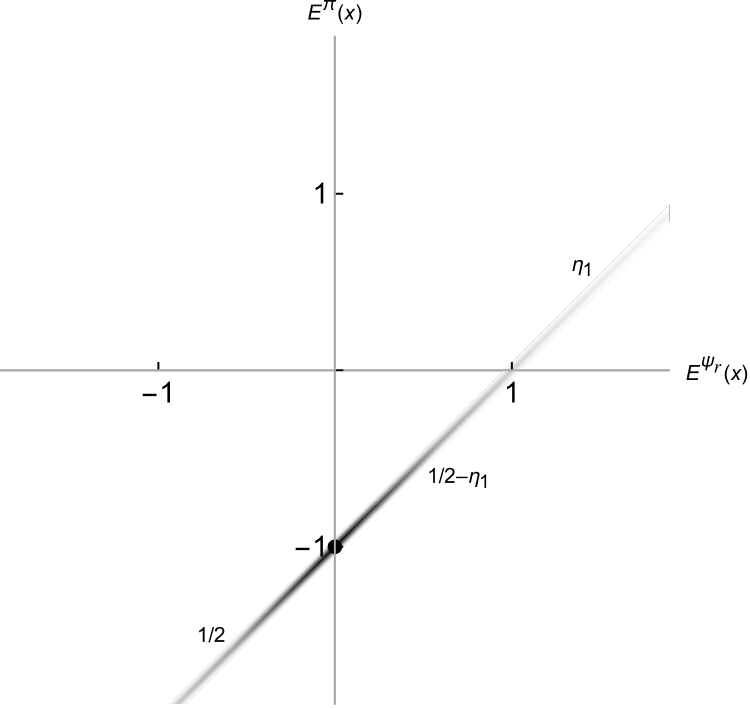}
\caption{Two pairs consisting of one standard function and one reciprocal function, with unequal bias factors within each pair. Each joint distribution is absolutely continuous supported on a thin diagonal strip of (horizontal or vertical) width~$2w$.}
\end{figure}

\begin{figure}[ht]
\includegraphics[width=3in]{graphs/psi_psi_r.pdf}
\qquad
\includegraphics[width=3in]{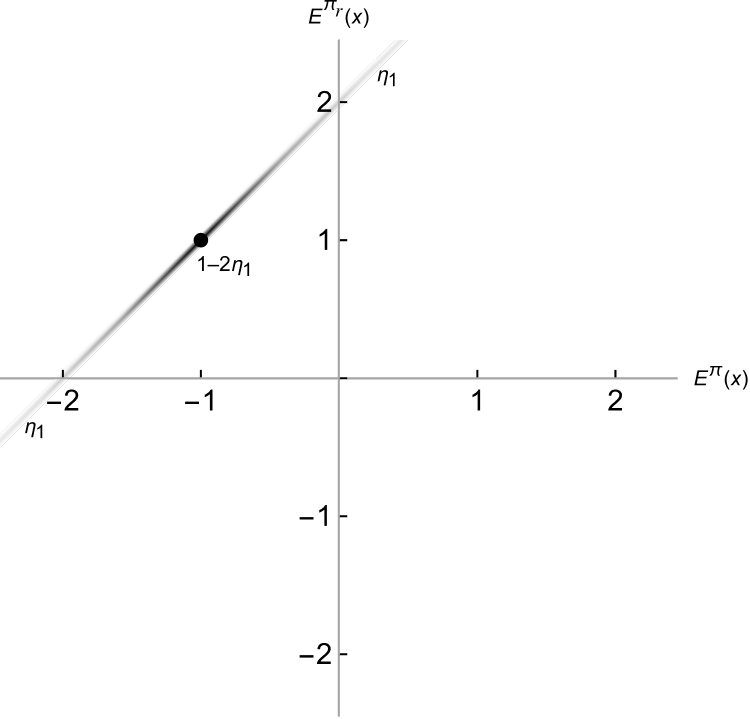}
\caption{Two pairs consisting of one standard function and one reciprocal function. Each joint distribution is absolutely continuous supported on a thin diagonal strip of width~$2w$. In the left-hand graph, both bias factors equal~$0$; in the right-hand graph, the bias factors are nonzero with opposite signs.}
\end{figure}

\end{document}